\documentclass[a4paper, 11pt, leqno]{article}
\usepackage{amscd, amsmath, amssymb, array, booktabs, braket, mathrsfs, tikz}
\usepackage{amsthm, color}
\usepackage[dvipdfm]{hyperref}
\usepackage[matrix,arrow]{xy}

\theoremstyle{plain}
\newtheorem{thm}{Theorem}[section]
\newtheorem{lem}[thm]{Lemma}
\newtheorem{prp}[thm]{Propostion}
\newtheorem{cor}[thm]{Corollary}
\newtheorem{cla}[thm]{Claim}
\newtheorem{conj}[thm]{Conjecture}
\newtheorem*{Classical}{Proposition}

\theoremstyle{definition}
\newtheorem{dfn}[thm]{Definition}
\newtheorem{example}[thm]{Example}
\newtheorem{rem}[thm]{Remark}
\newtheorem*{ack}{Acknowledgements}

\newtheorem*{notation}{Notation}

\newtheorem{s}{Step}

\newcommand{\bigslant}[2]{{\raisebox{.2em}{$#1$}\left/\raisebox{-.2em}{$#2$}\right.}}

\newcommand{\rect}{\mathrm{Rect}_{\left(0, \cdots, 0\right)}^{\left(n_1, \cdots, n_r \right)}}

\newcommand{\Sha}{\rotatebox[origin=c]{180}{$\Pi\kern-0.361em\Pi$}}

\begin{document}
\title{Stable free lattices in residually reducible Galois deformations and control theorem of Selmer groups}
\author{Dong Yan}
\date{}
\maketitle
\begin{abstract}
In this paper, we study the graph of homothety classes of stable free lattices in a two-dimensional representation over a local UFD. This generalizes a classical result of Serre in the case where the base ring is a discrete valuation ring. As applications, we consider the case where the representation comes from a residually reducible Hida family and we study the control theorem of Selmer groups. These results enable us to know the precise statement of the main conjecture in the residually reducible case as we will remark in \S 4 and \S 5.. 
\end{abstract}
\tableofcontents
\section{Introduction}
Let $\left(A, \varpi \right)$ be a discrete valuation ring with $\varpi$ a fixed uniformizer. Let $K:=\mathrm{Frac}\left(A\right)$ be the field of fractions and $V$ a $K$-vector space of dimension two. Let $\rho: G \rightarrow \mathrm{Aut}_{K}\left(V \right)$ be a representation of a group $G$ which is irreducible and has a $G$-stable lattice $T$. We denote by $C\left(\rho \right)$ the set of homothety classes of stable lattices. Following Serre \cite{SeTe}, we consider the graph, whose vertices are the elements of $C\left(\rho \right)$ and for $x, x^{\prime} \in C\left(\rho \right)$, we draw an edge form $x$ to $x^{\prime}$ if there exist representatives $T$ and $T^{\prime}$ of $x$ and $x^{\prime}$ respectively such that $T \subset T^{\prime}$ and $\left.T^{\prime}\right/{T} \stackrel{\sim}{\rightarrow} \left.A\right/{\left(\varpi \right)}$ as an $A$-module. Then one can establish a dictionary between the algebraic properties of the representation $\rho$ and the geometric properties of the graph $C\left(\rho \right)$. For example, if $A$ is complete then $\rho$ is irreducible if and only if $C\left(\rho \right)$ is bounded. One of the applications of this dictionary is to give a conceptual and simple proof of Ribet's lemma \cite[Proposition 2.1]{Ri76} which is deduced by the following proposition.

\begin{Classical}[{\cite[Section 1.2]{Belllec}}, {\cite[Section 17, Proposition 2]{M11}}]
Let $T$ be a stable lattice of $V$. Suppose that $\rho$ is irreducible and the semi-simplification $\left(\left.T\right/{\varpi T} \right)^{\mathrm{ss}}$ of the residual representation is decomposed into two distinct characters, then the graph $C\left(\rho \right)$ is a segment. 
\end{Classical}

Nowadays the graph $C\left(\rho \right)$ has been studied more deeply. For the representations of higher dimension, the graph $C\left(\rho \right)$ has been studied by Bella\"iche-Graftieaux \cite{BG06} if the semi-simplification $\overline{\rho}^{\mathrm{ss}}$ of the residual representation is residually disjoint (i.e. the Jordan-H\"older components of $\overline{\rho}^{\mathrm{ss}}$ are non-isomorphic to each other). For the case where $\overline{\rho}^{\mathrm{ss}}$ is not residually disjoint and $\rho$ has dimension two, the graph $C\left(\rho \right)$ is studied by Bella\"iche-Chenevier \cite{BC14}. However, all of previous works focus on the case where $A$ is a discrete valuation ring i.e. a regular local ring of dimension one. 

In this paper, we study the case where $A$ is a local UFD and $\rho$ is a two-dimensional residually disjoint representation. Then a lattice may not be free over $A$. We study the graph of all stable free lattices. Rather than Ribet's lemma, we applicate this graph to deduce some results on the Iwasawa theory for residually reducible Hida deformations. In particular, we study the change of Selmer groups by isogeny, control theorem of Selmer groups and the main conjecture under specialization. Note that when $\rho$ comes from a Hida family, the deformation ring $\mathbb{I}$ has Krull dimension two. If we assume that $\mathbb{I}$ is a regular local ring, all reflexive lattices are free lattices. Then the change of Selmer group only depends on the change of stable free lattice (cf. \cite[Lemma 4.3]{Ochiai08}). This is why we focus on free lattices instead of all lattices. In general, a reflexive lattice may not be free if the Krull dimension of $A$ is greater than two, however we remark that by our Proposition \ref{reffr}, all stable reflexive lattices are always free.

We begin our set-up. Let $R$ be a local UFD with maximal ideal $\mathfrak{m}$ and field of fractions $K$. Assume that $R$ is complete with respect to the $\mathfrak{m}$-adic topology. Let $V$ be a $K$-vector space of dimension two and $\rho: G \rightarrow \mathrm{Aut}_{K}\left(V \right)$ a representation of a group $G$ which satisfies the following conditions
\begin{list}{}{}
\item[(Ir)]The representation $\rho$ is irreducible.
\item[(Lat-fr)]The vector space $V$ has a free lattice $T$ which is  stable under the action of $G$.
\item[(Red)]The trace $\mathrm{tr}\,\rho \,\mathrm{mod}\,\mathfrak{m}=\psi+\psi^{\prime}$ for characters $\psi, \psi^{\prime}: G \rightarrow \left(\left.R\right/{\mathfrak{m}}\right)^{\times}$.
\item[($G$-dist)]The characters $\psi$ and $\psi^{\prime}$ are distinct.
\end{list}

Let $T$ be a stable free lattice. We denote by $[T]:=\set{ x T | x \in K^{\times} }$ the homothety class of $T$ and by $C^{\mathrm{fr}}\left(\rho \right)$ the set of homothety classes of stable free lattices. We define a graph structure on $C^{\mathrm{fr}}\left(\rho \right)$ as follows:
\begin{list}{}{}
\item[($\mathrm{vert}\,C^{\mathrm{fr}}\left(\rho \right)$)]The vertex set is $C^{\mathrm{fr}}\left(\rho \right)$.
\item[($\mathrm{edge}\,C^{\mathrm{fr}}\left(\rho \right)$)]For two points $x, x^{\prime} \in C^{\mathrm{fr}}\left(\rho \right)$, we draw an edge from $x$ to $x^{\prime}$ if there exists representatives $T$ and $T^{\prime}$ of $x$ and $x^{\prime}$ respectively such that $T \subset T^{\prime}$ and $\mathrm{char}_{R}\left(\left.T^{\prime}\right/{T} \right)=\mathfrak{p}$ for some height-one prime ideal $\mathfrak{p}$ of $R$, where $\mathrm{char}_{R}$ denotes the characteristic ideal. 
\end{list}
By the structure theorem of finitely generated torsion $R$-modules, the graph $C^{\mathrm{fr}}\left(\rho \right)$ is undirected. 

Now let $n\in\mathbb{Z}_{\geq 0}$ and we define the graph of $n$-dimensional rectangle in $\mathbb{R}^n$. When $n=0$ we define the rectangle to be only one point. When $n\geq 1$ and for fixed integers $a_i \in \mathbb{Z}_{\geq 0}\ \left(1 \leq i \leq n \right)$, we define the graph of an $n$-dimensional rectangle $\mathrm{Rect}_{\left(0, \cdots, 0\right)}^{\left(a_1, \cdots, a_n \right)}$ as follows.
\begin{list}{}{}
\item[($\mathrm{vert}\,\mathrm{Rect}_{\left(0, \cdots, 0\right)}^{\left(a_1, \cdots, a_n \right)}$)]The vertex set is $\left\{\left(j_1, \cdots, j_n \right) \in \mathbb{Z}^{n} |\ 0 \leq j_1 \leq a_1, \cdots, 0 \leq j_n \leq a_n \right\}.$
\item[($\mathrm{edge}\,\mathrm{Rect}_{\left(0, \cdots, 0\right)}^{\left(a_1, \cdots, a_n \right)}$)]We draw an edge from $\left(j_1, \cdots, j_n \right)$ to $\left(j_1^{\prime}, \cdots, j_n^{\prime} \right)$ if there exists an integer $1 \leq s \leq n$ such that 
\begin{equation}\label{12231}
|j_i^{\prime}-j_i|=\begin{cases}
0 & \left(i \neq s \right) \\
1 & \left(i=s \right). \\
\end{cases}
\end{equation}
\end{list}
In the graph $\mathrm{Rect}_{\left(0, \cdots, 0\right)}^{\left(a_1, \cdots, a_n \right)}$, we define $\left(j_1^{\prime}, \cdots, j_n^{\prime} \right) \leq \left(j_1, \cdots, j_n \right)$ if $j_i^{\prime} \leq j_i$ for all $i=1, \cdots, n.$ Then under the relation $\leq$, $\mathrm{Rect}_{\left(0, \cdots, 0\right)}^{\left(a_1, \cdots, a_n \right)}$ is a partially ordered graph and has minimal and maximal elements $\left(0, \cdots, 0\right)$ and $\left(a_1, \cdots, a_n \right)$ respectively. We have the following theorem:

\begin{thm}\label{Gal}
Let $R$ be a local UFD with maximal ideal $\mathfrak{m}$ and field of fractions $K$. Assume that $R$ is complete with respect to the $\mathfrak{m}$-adic topology and the characteristic of the residue field $\left.R\right/\mathfrak{m}$ is not equal to two. Let $V$ be a $K$-vector space of dimension two. Let $\rho: G \rightarrow \mathrm{Aut}_{K}\left(V \right)$ a representation of a group $G$ satisfying the conditions (Ir), (Lat-fr), (Red) and ($G$-dist). Then we have the following statements
\begin{enumerate}
\item[(1)]There exists an integer $r \in \mathbb{Z}_{\geq 0}$ such that the graph $C^{\mathrm{fr}}\left(\rho \right)$ is isomorphic to an $r$-dimensional rectangle. 
\item[(2)]More precisely, let $J:=J\left(\rho \right)$ be the ideal of reducibility of $\rho$ (cf. \cite[\S 2.3]{BC05}) with characters $\vartheta$ and $\vartheta^{\prime}: G \rightarrow \left(\left.R\right/J \right)^{\times}$ such that $\mathrm{tr}\,\rho \,\mathrm{mod}\,J=\vartheta+\vartheta^{\prime}$. Let $\mathcal{N}$ be the set of height-one prime ideals $\mathfrak{p}$ of $R$ such that $J_{\mathfrak{p}} \subsetneq R_{\mathfrak{p}}$, then $C^{\mathrm{fr}}\left(\rho \right)$ consists of a unique element if $\mathcal{N}$ is empty. Suppose that $\mathcal{N}$ is non-empty, let $\mathcal{N}=\set{\mathfrak{p}_1, \cdots, \mathfrak{p}_r}$ and $n_i:=\mathrm{ord}_{\mathfrak{p}_i}J_{\mathfrak{p}_i}$. Then we have an isomorphism
\begin{equation*}
\Phi_{\vartheta}: C^{\mathrm{fr}}\left(\rho \right) \stackrel{\sim}{\rightarrow} \mathrm{Rect}_{\left(0, \cdots, 0\right)}^{\left(n_1, \cdots, n_r \right)}
\end{equation*}
such that for points $\left(j_1, \cdots, j_r \right)$ and $\left(j_1^{\prime}, \cdots, j_r^{\prime}\right)$ in $\mathrm{Rect}_{\left(0, \cdots, 0\right)}^{\left(n_1, \cdots, n_r \right)}$, the following conditions are equivalent:
\begin{enumerate}
\item[(i)]We have $\left(j_1, \cdots, j_r \right) \leq \left(j_1^{\prime}, \cdots, j_r^{\prime} \right)$.
\item[(ii)]There exist representatives $T$ and $T^{\prime}$ of $\Phi_{\vartheta}^{-1}\left( \left(j_1, \cdots, j_r \right)   \right)$ and $\Phi_{\vartheta}^{-1}\left( \left(j_1^{\prime}, \cdots, j_r^{\prime} \right)   \right)$ respectively such that $T \subset T^{\prime}$ and $\left.T^{\prime}\right/{T}$ is a quotient of $\left.R\right/J\,\left(\vartheta \right)$.
\end{enumerate}
Moreover if these conditions are satisfied, $\left.T^{\prime}\right/{T}$ is isomorphic to $\left.R\right/{\mathfrak{p}_1^{j_1^{\prime}-j_1}\cdots\mathfrak{p}_r^{j_r^{\prime}-j_r}}\,\left(\vartheta \right).$
\end{enumerate}

\end{thm}

Theorem \ref{Gal} will be proved in \S 3. Now we study its application to number theory. First let us briefly recall some preliminaries on Hida family and its Galois representation. We fix an odd prime $p$ from now on to the end of the paper. We fix embeddings $\overline{\mathbb{Q}} \hookrightarrow \overline{\mathbb{Q}_p}$ and $\overline{\mathbb{Q}} \hookrightarrow \mathbb{C}$, where $\overline{\mathbb{Q}}$ and $\overline{\mathbb{Q}_p}$ are the algebraic closures of the rational number field $\mathbb{Q}$ and the $p$-adic number field $\mathbb{Q}_p$ respectively. Let $\Gamma$ be the Galois group $\mathrm{Gal}\left(\left.\mathbb{Q}_{\infty}\right/{\mathbb{Q}}\right)$ of the cyclotomic $\mathbb{Z}_p$-extension $\mathbb{Q}_{\infty}$ over $\mathbb{Q}$ and $\Gamma^{\prime}$ the $p$-Sylow subgroup of the group of diamond operators for the tower of modular curves $\left\{Y_1\left(p^t\right)\right\}_{t\geq 1}$. We have the following canonical isomorphisms
\begin{equation*}
\kappa_{\mathrm{cyc}}: \Gamma \stackrel{\sim}{\rightarrow} 1+p\mathbb{Z}_p,\ \ \kappa^{\prime}: \Gamma^{\prime} \stackrel{\sim}{\rightarrow} 1+p\mathbb{Z}_p.
\end{equation*}
We fix a topological generator $\gamma$ (resp. $\gamma^{\prime}$) of $\Gamma$ (resp. $\Gamma^{\prime}$) such that $\kappa_{\mathrm{cyc}}\left(\gamma \right)=\kappa^{\prime}\left(\gamma^{\prime} \right)$. In this section, let $\mathbb{I}$ be a quotient of Hida's ordinary cuspidal Hecke algebra $\mathbf{h}\left(1, \mathbb{Z}_p \right)^{\mathrm{ord}}$ of tame level one (cf. \cite[\S 1]{H1}) corresponding to a normalized $\mathbb{I}$-adic eigen cusp form $\mathcal{F}=\displaystyle\sum_{n=1}^{\infty}a\left(n, \mathcal{F} \right)q^n \in \mathbb{I}[[q]]$. Recall that $\mathbb{I}$ is a local domain of Krull dimension two which is finite flat over $\Lambda^{\prime}:=\mathbb{Z}_p[[\Gamma^{\prime}]]$. In \cite{H2}, Hida constructed a Galois representation $\mathbb{T}$ attached to $\mathcal{F}$ such that $\mathbb{T}$ is a finitely generated torsion-free $\mathbb{I}$-module and $\mathbb{V}_{\mathcal{F}}:=\mathbb{T}\otimes_{\mathbb{I}}\mathbb{K}$ has dimension two over the field of fractions $\mathbb{K}$ of $\mathbb{I}$. For any prime $l \neq p$, the trace of the action of the arithmetic Frobenius element $\mathrm{Frob}_l$ on $\mathbb{V}_{\mathcal{F}}$ is equal to the $l$-th Fourier coefficient $a\left(l, \mathcal{F} \right)$. 

We assume that $\mathbb{I}$ is a UFD and $\rho_{\mathcal{F}}$ satisfies the condition (Lat-fr) throughout the paper. For a stable free lattice $\mathbb{T}$ of $\mathbb{V}_{\mathcal{F}}$, let $\mathbb{A}:=\mathbb{T}\otimes_{\mathbb{I}}\mathbb{I}^{\lor}$ where $\mathbb{I}^{\lor}$ is the Pontryagin dual of $\mathbb{I}$. Due to Wiles \cite[Theorem 2.2.2]{Wi88}, there exists a $G_{\mathbb{Q}_p}$-submodule $F^{+}\mathbb{V}_{\mathcal{F}}$ of $\mathbb{V}_{\mathcal{F}}$ of dimension one such that $\left.\mathbb{V}_{\mathcal{F}}\right/{F^{+}\mathbb{V}_{\mathcal{F}}}$ is an unramified $G_{\mathbb{Q}_p}$-module. Let $F^{+}\mathbb{T}:=\mathbb{T}\cap F^{+}\mathbb{V}_{\mathcal{F}}$ and $F^{+}\mathbb{A}:=F^{+}\mathbb{T}\otimes_{\mathbb{I}}\mathbb{I}^{\lor}$ in $\mathbb{A}$. We define the Selmer group $\mathrm{Sel}_{\mathbb{A}}$ as Definition \ref{7131}. We now consider the two-variable deformation. Let $\mathcal{T}:=\mathbb{T}\hat{\otimes}_{\mathbb{Z}_p}\mathbb{Z}_p[[\Gamma]]\left(\tilde{\kappa}^{-1} \right)$, where $\tilde{\kappa}$ is the character $G_{\mathbb{Q}} \twoheadrightarrow \Gamma \hookrightarrow \mathbb{Z}_p[[\Gamma]]^{\times}$. Let $\mathcal{R}:=\mathbb{I}[[\Gamma]]$ and $\mathcal{A}:=\mathcal{T}\otimes_{\mathcal{R}}\mathcal{R}^{\lor}$. Let $F^{+}\mathcal{T}:=F^{+}\mathbb{T}\hat{\otimes}_{\mathbb{Z}_p}\mathbb{Z}_p[[\Gamma]]\left(\tilde{\kappa}^{-1} \right)$ and $F^{+}\mathcal{A}:=F^{+}\mathcal{T}\otimes_{\mathcal{R}}\mathcal{R}^{\lor}$. We define the Selmer group $\mathrm{Sel}_{\mathcal{A}}$ for two-variable Hida deformation in the same manner.

Nowadays it is well-known that an effective strategy for studying the Iwasawa theory of $\mathrm{Sel}_{\mathbb{A}}$ is working on the Selmer group $\mathrm{Sel}_{\mathcal{A}}$ and specializing its characteristic ideal. This strategy was first proposed by Ochiai \cite[\S 7]{Ochiai06}. Let $P$ be the height-one prime ideal of $\mathcal{R}$. The key tool in specializing the characteristic ideal is to study the following map
\begin{equation*}
\mathrm{res}_{P}: \mathrm{Sel}_{\mathcal{A}[P]} \rightarrow \mathrm{Sel}_{\mathcal{A}}[P]
\end{equation*}
induced by $h_P: H^1\left(\left.\mathbb{Q}_{\Sigma}\right/\mathbb{Q}, \mathcal{A}[P] \right) \rightarrow H^1\left(\left.\mathbb{Q}_{\Sigma}\right/\mathbb{Q}, \mathcal{A} \right)[P]$, where $\mathbb{Q}_{\Sigma}$ denotes the maximal extension of $\mathbb{Q}$ unramified outside the primes dividing $p$ and $\infty$.

Let $\overline{\rho}_{\mathcal{F}}$ be the residual representation of $\rho_{\mathcal{F}}$ modulo the maximal ideal $\mathfrak{m}_{\mathbb{I}}$ of $\mathbb{I}$. Then under the assumption (Lat-fr), $\overline{\rho}_{\mathcal{F}}$ is defined as the semi-simplification of $\left.\mathbb{T}\right/\mathfrak{m}_{\mathbb{I}}\mathbb{T}$ for any stable free lattice $\mathbb{T}$. In \cite[Proposition 5.1]{Ochiai06}, Ochiai studied the kernel of the map $l_{P}$ of local cohomology groups (see the diagram \eqref{190208} in \S 4.3 below). Then under the assumption that $\overline{\rho}_{\mathcal{F}}$ is irreducible, the map $h_P$ is injective and Ochiai determined the cokernel $\mathrm{Coker}\,\mathrm{res}_P$ by $\mathrm{Ker}\,l_P$. 

In this paper, we study the case where $\overline{\rho}_{\mathcal{F}}$ is reducible. We assume the following condition
\begin{list}{}{}
\item[(Eis)]The residual representation $\overline{\rho}_{\mathcal{F}}$ is isomorphic to $\mathbf{1}\oplus\overline{\chi}$, where $\chi$ is a primitive Dirichlet character modulo $p$ (which is odd) and $\overline{\chi}$ is the character $G_{\mathbb{Q}} \stackrel{\chi}{\rightarrow} \mathbb{I}^{\times} \twoheadrightarrow \left(\left.\mathbb{I}\right/\mathfrak{m}_{\mathbb{I}} \right)^{\times}.$
\end{list}
Then the map $h_P$ may not be injective and $\mathrm{Ker}\,h_P$ may depend on the choice of stable lattice. By the diagram \eqref{190208} in \S 4.3, we know that to study control theorem in the residually reducible case, it is important to study the group $H^0\left(\mathbb{Q}, \mathcal{A} \right)$ for each stable free lattice $\mathbb{T}$. The second motivation for studying $H^0\left(\mathbb{Q}, \mathcal{A} \right)$ in this paper is that it appears in the main conjecture in residually reducible case which we will discuss in \S 4.5.

Using Theorem \ref{Gal}, we study $H^0\left(\mathbb{Q}, \mathcal{A} \right)$ in Proposition \ref{H0}. Thus by combing Proposition \ref{H0} and Ochiai's result on $\mathrm{Ker}\,l_P$, we obtain a control theorem in the residually reducible case as follows. We remark that under the setting that $\mathcal{F}$ has tame level one and the assumption (Eis), one of the characters $\vartheta$ and $\vartheta^{\prime}$ in Theorem \ref{Gal} is trivial by class field theory. We denote by $\mathbb{J}$ the Eisenstein ideal of $\mathbb{I}$ which is generated by $a\left(l, \mathcal{F} \right)-1-\mathrm{det}\,\rho_{\mathcal{F}}\left(\mathrm{Frob}_l \right)$ for all primes $l \neq p$ and $a\left(p, \mathcal{F} \right)-1$. We assume that $\mathbb{J}$ is a principal ideal from now on to the end of this section. This follows if the Eisenstein components of Hida's Hecke algebra and cuspidal Hecke algebra (see \S 4.2 below) are Gorenstein by Ohta \cite[Theorem 2]{MO05}.

Let $P$ be a height-one prime ideal of $\mathcal{R}$. We remark that under the assumption that $\mathcal{R}$ is a UFD, the prime ideal $P$ is principal. In this section, we study the following two types of $P$:
\begin{enumerate}
\item[(I)]The ideal $P$ is generated by $\gamma-1$.
\item[(II)]The ideal $P$ is a factor of $\mathbb{J}$ in $\mathcal{R}$.
\end{enumerate}

\begin{thm}[Theorem \ref{T3}]\label{Con}
Suppose that $\mathbb{I}$ is a UFD and $\rho_{\mathcal{F}}$ satisfies the condition (Lat-fr). Assume $\chi \neq \omega$, (Eis) and that $\mathbb{J}$ is principal. We decompose $\mathbb{J}$ into product of height-one prime ideals as $\mathbb{J}=\mathfrak{p}_1^{n_1}\cdots\mathfrak{p}_r^{n_r}$, where $\mathfrak{p}_i$ and $\mathfrak{p}_j$ are distinct when $i\neq j$. Let us take a stable free lattice $\mathbb{T}$ and let $\left(j_1, \cdots, j_r \right)$ be the vertex of $\mathrm{Rect}_{\left(0, \cdots, 0\right)}^{\left(n_1, \cdots, n_r \right)}$ which is the image of $[\mathbb{T}]$ under $\Phi_{\mathbf{1}}$ in Theorem \ref{Gal}. Then we have the following statements. 
\begin{enumerate}
\item[(1)]When $P$ is in case (I), the map $\mathrm{res}_P$ is injective and the cokernel $\mathrm{Coker}\,\mathrm{res}_P$ is isomorphic to $\mathbb{I}^{\lor}\left[\dfrac{\left(a\left(p, \mathcal{F} \right)-1 \right)}{\mathfrak{p}_1^{n_1-j_1}\cdots\mathfrak{p}_r^{n_r-j_r}} \right]$.
\item[(2)]When $P$ is in case (II), the map $\mathrm{res}_{\mathfrak{p}_i}$ is surjective and we have the following statements for $\mathrm{Ker}\,\mathrm{res}_{\mathfrak{p}_i}$. 
\begin{enumerate}
\item[(2.1)]When $j_i=n_i$, the kernel $\mathrm{Ker}\,\mathrm{res}_{\mathfrak{p}_i}$ is trivial.
\item[(2.2)]When $j_i<n_i$, the kernel $\mathrm{Ker}\,\mathrm{res}_{\mathfrak{p}_i}$ is isomorphic to $\left(\left.\mathcal{R}\right/\mathfrak{p}_i\right)^{\lor}[\gamma-1]$. 
\end{enumerate}
\end{enumerate}
\end{thm}

Theorem \ref{Con} leads us to study the Iwasawa main conjecture by specialization. Let $\mathbb{T}$ be a stable free lattice and $\mathcal{T}:=\mathbb{T}\hat{\otimes}\mathbb{Z}_p[[\Gamma]]\left(\tilde{\kappa}^{-1} \right)$. If we assume that $\mathbb{I}$ is Gorenstein and $\chi \neq \omega$, the module of $\mathbb{I}$-adic modular symbols (cf. \cite[\S 5.5]{Ki94}) attached to $\mathbb{T}$ is free of rank two over $\mathbb{I}$. Then the two-variable $p$-adic $L$-function $L_p^{\mathrm{Ki}}\left(\mathcal{T} \right)$ constructed by Mazur and Kitagawa \cite{Ki94} is an element of $\mathcal{R}\otimes_{\mathbb{Z}_p}\mathbb{Q}_p$. The Iwasawa main conjecture for Galois deformations (see Conjecture \ref{IMCGalois}) was proposed by Greenberg and modified by Ochiai. Firstly we consider the deformation $\mathcal{T}$ and by Proposition \ref{H0}, one could establish the main conjecture for $\mathcal{T}$ as follows.
\begin{conj}\label{IMCtwovarintro}
Suppose that $\mathbb{I}$ is a UFD and Gorenstein. Assume (Eis) and $\chi \neq \omega$. Let $\mathbb{T}$ be a stable free lattice. Then we have the following equality
\begin{equation*}
\mathrm{char}_{\mathcal{R}}\left(\mathrm{Sel}_{\mathcal{A}} \right)^{\lor}=\left(L_p^{\mathrm{Ki}}\left(\mathcal{T} \right) \right).
\end{equation*}
\end{conj}

Now we study the main conjecture for $\left.\mathcal{T}\right/P\mathcal{T}$ when $P$ is in cases (I) and (II). In case (I), we denote by $L_p^{\mathrm{Ki}}\left(\mathbb{T} \right) \in \mathbb{I}\otimes_{\mathbb{Z}_p}\mathbb{Q}_p$ the image of $L_p^{\mathrm{Ki}}\left(\mathbb{T} \right)$ under $\mathcal{R} \twoheadrightarrow \left.R\right/\left(\gamma-1 \right).$ By Proposition \ref{H0}, the main conjecture for $\left.\mathcal{T}\right/\left(\gamma-1 \right)\mathcal{T}\cong \mathbb{T}$ is as follows.

\begin{conj}\label{IMCHidaintro}
Let us keep the assumptions of Conjecture \ref{IMCtwovarintro}. Then we have the following equality
\begin{equation*}
\left(a\left(p, \mathcal{F} \right)-1 \right)\left(\mathrm{char}_{\mathbb{I}}\left(H^0\left(\mathbb{Q}, \mathbb{A} \right) \right)^{\lor} \right)^{-1}\mathrm{char}_{\mathbb{I}}\left(\mathrm{Sel}_{\mathbb{A}} \right)^{\lor}=\left(L_p^{\mathrm{Ki}}\left(\mathbb{T} \right) \right).
\end{equation*}
\end{conj}

Now we consider the case when $P$ is a factor of $\mathbb{J}$. In this case the Selmer group $\left(\mathrm{Sel}_{\mathcal{A}[P]}\right)^{\lor}$ is not a torsion $\left.\mathcal{R}\right/P$-module for some $\mathbb{T}$. Hence first it is important to study that for a give $P\mid\mathbb{J}$, which free lattice makes the Selmer group torsion. We remark that we assume that $\mathbb{I}$ is a UFD and $\mathbb{J}$ is principal. Let us consider the following conditions
\begin{list}{}{}
\item[(Vand $\chi^{-1}\omega$)]The $\chi^{-1}\omega$-part of the ideal class group of $\mathbb{Q}\left(\mu_{p} \right)$ is trivial.
\item[($p$Four)]There exists a height-one prime ideal $\mathfrak{p}$ dividing $\mathbb{J}$ such that $\mathrm{ord}_{\mathfrak{p}}\mathbb{J}_{\mathfrak{p}}=\mathrm{ord}_{\mathfrak{p}}\left(a\left(p, \mathcal{F} \right)-1 \right)_{\mathfrak{p}}.$
\end{list}

\begin{thm}[Theorem \ref{toreis}]\label{23main}
Let us keep the assumptions and the notation of Theorem \ref{Con}. Assume further the conditions (Vand $\chi^{-1}\omega$) and ($p$Four). Let $\mathbb{T}$ be a stable free lattice and $\left(j_1, \cdots, j_r \right)$ the vertex of $\mathrm{Rect}_{\left(0, \cdots, 0\right)}^{\left(n_1, \cdots, n_r \right)}$ which is the image of $[\mathbb{T}]$ under $\Phi_{\mathbf{1}}$ in Theorem \ref{Gal}. Let $\mathfrak{p}_i$ be a height-one prime ideal which satisfies ($p$Four). Then we have the following statements. 
\begin{enumerate}
\item[(1)]When $j_i=0$, $\left(\mathrm{Sel}_{\mathcal{A}[\mathfrak{p}_i]}\right)^{\lor}$ is a finitely generated torsion $\left.\mathcal{R}\right/{\mathfrak{p}_i}$-module. Moreover, we have
\begin{equation*}
\mathrm{char}_{\left.\mathcal{R}\right/{\mathfrak{p}_i}}\left(\mathrm{Sel}_{\mathcal{A}[\mathfrak{p}_i]}\right)^{\lor} \subset \left(\gamma-1 \right).
\end{equation*} 
\item[(2)]When $j_i>0$, $\left(\mathrm{Sel}_{\mathcal{A}[\mathfrak{p}_i]}\right)^{\lor}$ has $\left.\mathcal{R}\right/{\mathfrak{p}_i}$-rank equal to one. 
\end{enumerate}
\end{thm}

Now let us take a prime ideal $\mathfrak{p}_i$ dividing $\mathbb{J}$ and let us consider the main conjecture for $\left.\mathcal{T}\right/{\mathfrak{p}_i\mathcal{T}}$. We denote by $L_p\left(\left.\mathcal{T}\right/{\mathfrak{p}_i\mathcal{T}} \right)$ the image of $L_p^{\mathrm{Ki}}\left(\mathcal{T} \right)$ in $\left.\mathcal{R}\right/{\mathfrak{p}_i}\otimes_{\mathbb{Z}_p}\mathbb{Q}_p$. By Conjecture \ref{IMCGalois}, the main conjecture for $\left.\mathcal{T}\right/{\mathfrak{p}_i}\mathcal{T}$ is as follows.
\begin{conj}\label{IMCEis}
Let us keep the assumptions of Conjecture \ref{IMCtwovarintro}. Assume that $\left(\mathrm{Sel}_{\mathcal{A}[\mathfrak{p}_i]}\right)^{\lor}$ is a finitely generated torsion $\left.\mathcal{R}\right/{\mathfrak{p}_i}$-module. Then we have the following equality
\begin{equation*}
\mathrm{char}_{\left.\mathcal{R}\right/\mathfrak{p}_i}\left(H^0\left(\mathbb{Q}, \mathcal{A} [\mathfrak{p}_i]\right)^{\lor} \right)^{-1}\mathrm{char}_{\left.\mathcal{R}\right/\mathfrak{p}_i}\left(\mathrm{Sel}_{\mathcal{A} [\mathfrak{p}_i]} \right)^{\lor}=\left(L_p\left(\left.\mathcal{T}\right/{\mathfrak{p}_i\mathcal{T}} \right) \right).
\end{equation*}

\end{conj}
By combining our control theorem (Theorem \ref{Con}) and the torsionness of $\mathrm{Sel}_{\mathcal{A} [\mathfrak{p}_i]}$ (Theorem \ref{23main}), we have the following corollary. 
\begin{cor}[Corollary \ref{IMCs}, Corollary \ref{sp}]\label{IMCordeisintro}
Let us keep the assumptions and the notation of Theorem \ref{Con}. Let $\mathbb{T}$ be a stable free lattice and $\left(j_1, \cdots, j_r \right)$ the vertex of $\mathrm{Rect}_{\left(0, \cdots, 0\right)}^{\left(n_1, \cdots, n_r \right)}$ which is the image of $[\mathbb{T}]$ under $\Phi_{\mathbf{1}}$ in Theorem \ref{Gal}. Then we have the following statements.
\begin{enumerate}
\item[(1)]Conjecture \ref{IMCtwovarintro} implies Conjecture \ref{IMCHidaintro}.
\item[(2)]Assume further the conditions (Vand $\chi^{-1}\omega$) and ($p$Four). Let $\mathfrak{p}_i$ be a height-one prime ideal which satisfies ($p$Four). Suppose $j_i=0$, then $\left(\mathrm{Sel}_{\mathcal{A}[\mathfrak{p}_i]}\right)^{\lor}$ is $\left.\mathcal{R}\right/{\mathfrak{p}_i}$-torsion and Conjecture \ref{IMCtwovarintro} implies Conjecture \ref{IMCordeisintro}.

\end{enumerate}
\end{cor}

\begin{rem}[Remark \ref{modifychange}, Remark \ref{remeis}]\label{IMCrem}
Let us explain the role of the term $H^{0}\left(\mathbb{Q}, \mathcal{A}[P] \right)$ in the main conjecture when $P$ is in cases (I) and (II). When $P=\left(\gamma-1\right)$. We remark that in this case, the term $a\left(p, \mathcal{F} \right)-1$ also appears when $\overline{\rho}_{\mathcal{F}}$ is irreducible by \cite[\S 7-(c)]{Ochiai06}. We will show in Corollary \ref{alglonevar} that the term $\mathrm{char}_{\mathbb{I}}\left(\mathrm{Sel}_{\mathbb{A}} \right)^{\lor}$ is independent on $\mathbb{T}$. However the analytic side $L_p^{\mathrm{Ki}}\left(\mathbb{T}\right)$ depends on $\mathbb{T}$ by the definition of $\mathbb{I}$-adic modular symbols associated to stable free lattices. Thus one may view $\mathrm{char}_{\mathbb{I}}\left(H^0\left(\mathbb{Q}, \mathbb{A} \right)\right)^{\lor}$ as making the algebraic side depends on the choice of stable free lattice. 

When $P=\mathfrak{p}_i$ divides $\mathbb{J}$. By Theorem \ref{23main} we know that the characteristic ideal $\mathrm{char}_{\left.\mathcal{R}\right/{\mathfrak{p}_i}}\left(\mathrm{Sel}_{\mathcal{A}[\mathfrak{p}_i]}\right)^{\lor}$ has a trivial zero. Thus by combining the proof of Theorem \ref{toreis} and Corollary \ref{sp} we know that the term $\mathrm{char}_{\left.\mathcal{R}\right/\mathfrak{p}_i}\left(H^0\left(\mathbb{Q}, \mathcal{A} [\mathfrak{p}_i]\right)\right)^{\lor}$ is a modification related to the trivial zero.

\end{rem}

\begin{notation}
For a fixed Noetherian integrally closed domain $R$, we denote by $\mathrm{Spec}_{\mathrm{ht}=1}\left(R \right)$ the set of height-one prime ideals of $R$ and by $\mathrm{Frac}\left(\mathcal{R} \right)$ the field of fractions of $\mathcal{R}$. For an ideal $I$ of $R$, we denote by $V\left(I \right)$ the set of prime ideals of $R$ containing $I$. For a finitely generated $R$-module $M$, we denote by $M^{*}$ (resp. $M^{**}$) the $R$-linear dual (resp. double $R$-linear dual) of $M$. We denote by $\mathbf{1}$ the trivial character. We denote by $\mu_r$ the group of $r$-th roots of unity for an integer $r$ and by $\mathbb{Q}\left(\mu_r \right)$ the field obtained by adjoining $\mu_r$ to $\mathbb{Q}$.
\end{notation}

\begin{ack}
The author expresses his sincere gratitude to Professor Tadashi Ochiai for stimulating conversation on this subject and encouragement. He thanks Kenji Sakugawa and Somnath Jha for useful discussion. He also thanks Professor Hansheng Diao and Emmanuel Lecouturier for reading this manuscript carefully and pointing out mistakes. This paper has been written during his stay at Yanqi Lake Beijing Institute of Mathematical Sciences and Applications (BIMSA). The author is also thankful to BIMSA for their hospitality.
\end{ack}

\section{Preliminaries on the ideal of reducibility}
In this section, we recall the definition of ideal of reducibility and its application to counting the number of homothety classes of stable lattices of a representation over a discrete valuation ring. First we recall the following lemma due to Bella\"iche-Chenevier \cite{BC05}.
\begin{lem}[{\cite[Lemme 1]{BC05}}]\label{21}
Let $R$ be a local domain with maximal ideal $\mathfrak{m}$ and field of fractions $K$. We assume that the characteristic of the residue field $R/\mathfrak{m}$ is not two. Let $V$ be a vector space of dimension two over $K$ and $$\rho : G \rightarrow \mathrm{Aut}_K\left(V \right)$$ a linear representation of a group $G$ such that $\mathrm{tr}\rho\left(G \right) \subset R$. Assume that there exists an element $g_0 \in G$ such that the characteristic polynomial of $\rho\left(g_0\right)$ has roots in $R$ which are distinct modulo $\mathfrak{m}$. Then there exists a unique ideal $J\left(\rho \right)$ of $R$ such that for any ideal $I$, $J\left(\rho \right) \subset I$ if and only if there exist characters $\vartheta, \vartheta^{\prime} : G \rightarrow \left(R/I \right)^{\times}$ such that $\mathrm{tr}\rho\ \mathrm{mod}\ I=\vartheta+\vartheta^{\prime}$. Furthermore, if $I \subset \mathfrak{m}$, the set of such characters $\set{\vartheta, \vartheta^{\prime}}$ is unique. 
\end{lem}

\begin{rem}\label{constrofidealofredu}
Let us keep the assumptions and the notation of Lemma \ref{21}. We recall the construction of $J\left(\rho \right)$. Let $\lambda_1$ and $\lambda_2$ be the roots of the characteristic polynomial of $\rho\left(g_0\right)$. Let us choose a $K$-basis of $V$ such that $\rho\left(g_0\right)=\begin{pmatrix} \lambda_1 & 0 \\ 0 & \lambda_2 \end{pmatrix}$. For any $g \in G$, write $\rho\left(g \right)=\begin{pmatrix} a\left(g \right) & b\left(g \right) \\ c\left(g \right) & d\left(g \right) \end{pmatrix}$ with respect to the basis. We define $J\left(\rho \right)$ $R$-submodule of $K$ which is generated by $b(g)c(g^{\prime})$ for all $g, g^{\prime} \in G$. By calculation $\mathrm{tr}\,\rho\left(g \right)$ and  $\mathrm{tr}\,\rho\left(g_0g \right)$, one can show that $J\left(\rho \right)$ is contained in $R$. We call $J\left(\rho \right)$ the \textit{ideal of reducibility of $R$ corresponding to $\rho$}. 

\end{rem}

Now we restrict to the case where $R$ is a discrete valuation ring. Then we have the following proposition on the relation between $J\left(\rho \right)$ and $\sharp C\left(\rho \right)$.

\begin{prp}\label{22}
Let $A$ be a discrete valuation ring with $\varpi$ a fixed uniformizer and filed of fractions  $K$. We assume the characteristic of the residue field $A/\left(\varpi \right)$ is not two. Let $V$ be a two-dimensional $K$-vector space and $$\rho : G \rightarrow \mathrm{Aut}_K\left(V \right)$$
a linear representation of a group $G$ such that $\rho$ has a $G$-stable lattice $T$. Assume that there exists an element $g_0 \in G$ such that the characteristic polynomial of $\rho\left(g_0\right)$ has roots in $A$ which are distinct modulo $\left(\varpi\right)$.  We denote by $J\left(\rho \right)$ the ideal of reducibility of $A$. Then we have the following statements:
\begin{enumerate}
\renewcommand{\labelenumi}{(\arabic{enumi})}
\item We have $J\left(\rho \right) \subset \left(\varpi \right)$ if and only if the semi-simplification $\left(T/\varpi T \right)^{\mathrm{ss}}$ is decomposed into two characters i.e. $\left(T/\varpi T \right)^{\mathrm{ss}} \cong A/\left(\varpi \right)(\psi) \oplus A/\left(\varpi \right)(\psi^{\prime})$, where $\psi$ and $\psi^{\prime} : G \rightarrow \left(A/\left(\varpi \right) \right)^{\times}$ are characters. Moreover if they satisfied, we have $\psi\neq\psi^{\prime}$. 
\item We have $\sharp C\left(\rho \right)=\mathrm{ord}_{\varpi}J\left(\rho \right)+1$. 
\item More precisely, let $\mathrm{ord}_{\varpi}J\left(\rho \right)=n>0$. Then there exists a chain of $G$-stable lattices $T_n\supsetneq\cdots\supsetneq T_0$ which is a system of representatives of $C\left(\rho \right)$ such that for any $1 \leq j \leq n$, $T_j/T_0$ is isomorphic to $A/\left(\varpi \right)^j\,\left(\vartheta \right)$ as an $A[G]$-module, where $\set{\vartheta, \vartheta^{\prime}}$ is the set of characters with values in $\left(\left.A\right/\left(\varpi \right)^n \right)^{\times}$ such that $\mathrm{tr}\,\rho\,\mathrm{mod}\,\left(\varpi \right)^n=\vartheta+\vartheta^{\prime}$.

\end{enumerate}
\end{prp}

Note that the second assertion is mentioned in \cite[Remarques of Th\'eor\`eme 4.1.3]{BG06} and others follows by Lemma \ref{21}. We do not assume that $A$ is complete since there is a bijection of the set of $V$ onto the set of lattices of $\hat{V}:=V\otimes_{K}\hat{K}$ for the completion $\hat{K}$ of $K$ (cf. \cite[Chap. II, \S 1.1, \textit{completion}.]{SeTe}).

\section{Proof of Theorem \ref{Gal}}
We prove Theorem \ref{Gal} in this section. We divide our proof into two parts. In the first half of this section, we prove a weaken result of Theorem \ref{Gal} on the graph structure of stable reflexive lattices rather than free lattices. Then in the later half of this section, we will show that all stable reflexive lattices are free over $R$. 

Let $R$ be a local UFD with maximal ideal $\mathfrak{m}$ and field of fractions $K$. Assume that $R$ is complete with respect to the $\mathfrak{m}$-adic topology and the characteristic of the residue field $\left.R\right/\mathfrak{m}$ is not equal to two. Let $V$ be a $K$-vector space of dimension two and $\rho: G \rightarrow \mathrm{Aut}_{K}\left(V \right)$ a representation of a group $G$. We consider a weak condition (Lat) of (Lat-fr) as follows:
\begin{list}{}{}
\item[(Lat)]The vector space $V$ has a lattice $T$ which is  stable under the action of $G$.
\end{list}
Let $T$ be a stable lattice, then $T^{**}$ is reflexive and hence we may assume that $T$ is reflexive. We denote by $[T]:=\set{ x T | x \in K^{\times} }$ the homothety class of $T$ and by $C^{\mathrm{ref}}\left(\rho \right)$ the set of homothety classes of stable reflexive lattices. We define a graph on $C^{\mathrm{ref}}\left(\rho \right)$ in the same way as $C^{\mathrm{fr}}\left(\rho \right)$ in \S 1. The goal of the first half of the section is to prove the following proposition.

\begin{prp}\label{Gal1}
Let $\rho: G \rightarrow \mathrm{Aut}_{K}\left(V \right)$ be a representation of a group $G$ such that $\mathrm{tr}\,\rho \subset R$. Suppose the conditions (Ir), (Lat) and (Red). Let $J:=J\left(\rho \right)$ be the ideal of reducibility of $\rho$ and let $\vartheta, \vartheta^{\prime}: G \rightarrow \left(\left.R\right/J \right)^{\times}$ be characters such that $\mathrm{tr}\,\rho \,\mathrm{mod}\,J=\vartheta+\vartheta^{\prime}$. Let $\mathcal{N}:=\set{\mathfrak{p}\in \mathrm{Spec}_{\mathrm{ht}=1}\left(R \right) | J_{\mathfrak{p}} \subsetneq R_{\mathfrak{p}}}$. Then $C^{\mathrm{ref}}\left(\rho \right)$ consists of a unique element if $\mathcal{N}$ is empty. Suppose that $\mathcal{N}$ is non-empty. Let $\mathcal{N}=\set{\mathfrak{p}_1, \cdots, \mathfrak{p}_r}$ and let $\mathrm{ord}_{\mathfrak{p}_i}J_{\mathfrak{p}_i}=n_i$. Then the graph $C^{\mathrm{ref}}\left(\rho \right)$ is isomorphic to the rectangle $\mathrm{Rect}_{\left(0, \cdots, 0\right)}^{\left(n_1, \cdots, n_r \right)}$. More precisely, we have an isomorphism
\begin{equation*}
\Phi_{\vartheta}: C^{\mathrm{ref}}\left(\rho \right) \stackrel{\sim}{\rightarrow} \mathrm{Rect}_{\left(0, \cdots, 0\right)}^{\left(n_1, \cdots, n_r \right)}
\end{equation*}
such that for points $\left(j_1, \cdots, j_r \right)$ and $\left(j_1^{\prime}, \cdots, j_r^{\prime}\right)$ in $\mathrm{Rect}_{\left(0, \cdots, 0\right)}^{\left(n_1, \cdots, n_r \right)}$, the following conditions are equivalent:
\begin{enumerate}
\item[(1)]We have $\left(j_1, \cdots, j_r \right) \leq \left(j_1^{\prime}, \cdots, j_r^{\prime} \right)$.
\item[(2)]There exist representatives $T$ and $T^{\prime}$ of $\Phi_{\vartheta}^{-1}\left( \left(j_1, \cdots, j_r \right)   \right)$ and $\Phi_{\vartheta}^{-1}\left( \left(j_1^{\prime}, \cdots, j_r^{\prime} \right)   \right)$ respectively such that $T \subset T^{\prime}$ and $\left(\left.T^{\prime}\right/{T} \right)_{\mathfrak{p}}$ is a quotient of $\left.R_{\mathfrak{p}}\right/J_{\mathfrak{p}}\,\left(\vartheta_1 \right)$ for any $\mathfrak{p}\in \mathrm{Spec}_{\mathrm{ht}=1}\left(R \right)$.
\end{enumerate}
Moreover, if these conditions are satisfied, $\left(\left.T^{\prime}\right/{T} \right)_{\mathfrak{p}_i}$ is isomorphic to $\left.R_{\mathfrak{p}_i}\right/{\mathfrak{p}_i^{j_i^{\prime}-j_i}}\left(\vartheta \right)$ for any $\mathfrak{p}_i \in \mathcal{N}$.

\end{prp}

\begin{lem}\label{redu}
Let us keep the assumptions and the notation of Theorem \ref{Gal1}. Then under the assumption (Red), $J_{\mathfrak{p}}$ coincide with the ideal of reducibility of $\rho\otimes_R R_{\mathfrak{p}}$.
\end{lem}
\begin{proof}
The proof is by definition. Clearly we have $\mathrm{tr}\,\rho \subset R \subset R_{\mathfrak{p}}$. Let us take an element $g_0 \in G$ such that $\psi\left(g_0 \right)\neq\psi^{\prime}\left(g_0 \right)$. Since $R$ is complete, the characteristic polynomial of $\rho\left(g_0 \right)$ has roots $\lambda_{\psi} \neq \lambda_{\psi^{\prime}}$ in $R$ such that $\overline{\lambda}_{\xi}=\xi\left(g_0 \right)$ for $\xi \in \set{\psi, \psi^{\prime}}$ by Hensel's lemma. Choose a basis $\set{v_{\psi}, v_{\psi^{\prime}}}$ of $V$ such that $\rho\left(g_0 \right)v_{\xi}=\lambda_{\xi} v_{\xi}$. Then the lemma follows immediately by the definition of ideal of reducibility (cf. Remark \ref{constrofidealofredu}).

\end{proof}

Let us take a stable reflexive lattice $T$ which is fixed to the end of the proof of Lemma \ref{1108}. For any $\mathfrak{p}\in P^{1}\left(R \right)$, we denote by $C\left(\rho\otimes R_{\mathfrak{p}}\right)$ the set of homothety classes of stable $R_{\mathfrak{p}}$-lattices of $V$. Since $R_{\mathfrak{p}}$ is a discrete valuation ring, we have
\begin{equation}\label{190228}
\sharp C\left(\rho\otimes R_{\mathfrak{p}} \right)=\mathrm{ord}_{\mathfrak{p}}J_{\mathfrak{p}}+1
\end{equation}
by Lemma \ref{redu} and Proposition \ref{22}. First we assume that $\mathcal{N}=\set{\mathfrak{p}_1, \cdots, \mathfrak{p}_r}$ which is non-empty and let us take an element $\mathfrak{p}_i \in \mathcal{N}$. Let $n_i:=\mathrm{ord}_{\mathfrak{p}_i}\mathbb{J}_{\mathfrak{p}_i}$, we have $$\mathrm{tr}\,\rho \equiv \vartheta+\vartheta^{\prime} \pmod{\mathfrak{p}_i^{n_i}R_{\mathfrak{p}_i}}.$$ Then by Proposition \ref{22}-(3), there exists a chain of stable $R_{\mathfrak{p}_i}$-lattices 
\begin{equation}\label{20200516}
T_i^{(n_i)} \supsetneq \cdots \supsetneq T_i^{(0)}
\end{equation}
of $V$ such that $T_i^{(n_i)} \supsetneq \cdots \supsetneq T_i^{(0)}$ is a system of  representatives of $C\left(\rho\otimes R_{\mathfrak{p}_i} \right)$ and $T_i^{(j_i^{\prime})}/T_i^{(j_i)}$ is isomorphic to $R_{\mathfrak{p}_i}/\mathfrak{p}_i^{j_i^{\prime}- j_i}\,\left(\vartheta_1 \right)$ for every $j_i \leq j_i^{\prime}$.

Now we construct a system of representatives of $C^{\mathrm{ref}}\left(\rho \right)$ by gluing all local lattices as follows.  For each $\mathfrak{p}_i \in \mathcal{N}$ and $0 \leq j_i \leq n_i$, we define the module $T\left(j_1, \cdots, j_r \right)$ as follows:
\begin{equation}\label{04281}
T\left(j_1, \cdots, j_r \right)=\left(\bigcap_{\mathfrak{p} \not\in \mathcal{N}}T_{\mathfrak{p}}\right)\bigcap\left(\bigcap_{\mathfrak{p}_i \in \mathcal{N}}T_i^{(j_i)}\right).
\end{equation}
Then $T\left(j_1, \cdots, j_r \right)$ is a reflexive lattice, more precisely we have
\begin{equation}\label{1901128}
T\left(j_1, \cdots, j_r \right)_{\mathfrak{p}}=\begin{cases}
T_{\mathfrak{p}} & \left(\mathfrak{p} \not\in \mathcal{N} \right) \\
T_i^{\left(j_i \right)} & \left(\mathfrak{p}=\mathfrak{p}_i\in \mathcal{N} \right). \\
\end{cases}
\end{equation}
Since every $T_{\mathfrak{p}}$ and $T_i^{\left(j_i \right)}$ are stable under the action of $G$, $T\left(j_1, \cdots, j_r \right)$ is a stable lattice. Thus $T\left(j_1, \cdots, j_r \right)$ is a stable reflexive lattice of $V$. 
\begin{lem}\label{1108}
When $\mathcal{N}$ is non-empty, $\set{T\left(j_1, \cdots, j_r \right) | i=1, \cdots, r, j_i=0, \cdots, n_i}$ is a set of representatives of $C^{\mathrm{ref}}\left(\rho \right)$. When $\mathcal{N}$ is empty, $C^{\mathrm{ref}}\left(\rho \right)$ consists only of the homothety class of $T$.
\end{lem}
\begin{proof}
First we assume that $\mathcal{N}$ is non-empty. Let us take a stable reflexive lattice $T^{\prime}$. By multiplying an element of $R$ if necessary, we may assume $T^{\prime} \subset T\left(n_1, \cdots, n_r \right)$. Let us take an element $\mathfrak{p} \in P^1\left(R \right)$ and let us consider the following cases:
\begin{enumerate}
\item[(a)]When $\mathfrak{p} \not\in \mathcal{N}$, $C\left(\rho\otimes_R R_{\mathfrak{p}} \right)$ consists only of the homothety class of $T_{\mathfrak{p}}$ by Lemma \ref{redu}. Since the localization $T\left(n_1, \cdots, n_r \right)_{\mathfrak{p}}=T_{\mathfrak{p}}$ by \eqref{1901128}, under the assumption $T^{\prime} \subset T\left(n_1, \cdots, n_r \right)$, there exists an integer $e_{\mathfrak{p}} \in \mathbb{Z}_{\geq 0}$ such that $T_{\mathfrak{p}}^{\prime}=\mathfrak{p}^{e_{\mathfrak{p}}}T_{\mathfrak{p}}.$
\item[(b)]When $\mathfrak{p}=\mathfrak{p}_i \in \mathcal{N}$, $T_i^{(n_i)} \supsetneq \cdots \supsetneq T_i^{(0)}$ is a system of  representatives of $C\left(\rho\otimes_R R_{\mathfrak{p}_i} \right)$. Then under the assumption $T^{\prime} \subset T\left(n_1, \cdots, n_r \right)$, there exist an integer $e_{\mathfrak{p}_i} \in \mathbb{Z}_{\geq 0}$ and an integer $0 \leq j_i \leq n_i$ such that $T_{\mathfrak{p}_i}=\mathfrak{p}_i^{e_{\mathfrak{p}_i}}T_i^{\left(j_i \right)}$. 
\end{enumerate}
Since $T$ and $T^{\prime}$ are lattices, we have $T^{\prime}_{\mathfrak{p}}=T_{\mathfrak{p}}$ for all but finitely many $\mathfrak{p} \in P^1\left(R \right)$. Thus the integers $e_{\mathfrak{p}}$ in cases (a) and (b) are $0$ for all but finitely many $\mathfrak{p}\in P^1\left(R \right)$. Furthermore, since $R$ is a UFD, every height-one prime ideal is principal. Then $\displaystyle\prod_{\mathfrak{p} \in P^{1}\left(R \right)}\mathfrak{p}^{e_{\mathfrak{p}}}$ is generated by an element $x \in R$. Let $T^{\prime\prime}=xT\left(j_1, \cdots, j_r \right).$ We have $T^{\prime}_{\mathfrak{p}}=T^{\prime\prime}_{\mathfrak{p}}$ for any $\mathfrak{p} \in P^{1}\left(R \right)$. Since $T^{\prime}$ and $T^{\prime\prime}$ are reflexive lattices, we have $$T^{\prime}=\bigcap_{\mathfrak{p} \in P^{1}\left(R \right)}T^{\prime}_{\mathfrak{p}}=\bigcap_{\mathfrak{p} \in P^{1}\left(R \right)}T^{\prime\prime}_{\mathfrak{p}}=T^{\prime\prime}.$$ This implies that $\set{T\left(j_1, \cdots, j_r \right) | i=1, \cdots, r, j_i=0, \cdots, n_i}$ is a system of  representatives of $C^{\mathrm{ref}}\left(\rho \right)$.

Now we assume that $\mathcal{N}$ is empty. Let us take a stable reflexive lattice $T^{\prime}$. By multiplying an element of $R$ if necessary, we may assume $T^{\prime} \subset T$. Under the assumption that $\mathcal{N}$ is empty, any prime ideal $\mathfrak{p} \in P^1\left(R \right)$ belongs to case (a) above. Then by the same argument, we have that there exists an element $x^{\prime}\in R$ such that $T^{\prime}=x^{\prime}T$. This completes the proof of Lemma \ref{1108}.

\end{proof}

Let us return to the proof of Proposition \ref{Gal1}

\begin{proof}[Proof of Proposition \ref{Gal1}]
We define $\Phi_{\vartheta_1}: C^{\mathrm{ref}}\left(\rho \right) \rightarrow \mathrm{Rect}_{\left(0, \cdots, 0\right)}^{\left(n_1, \cdots, n_r \right)}$ by sending the homothety class $[T\left(j_1, \cdots, j_r \right)]$ to $\left(j_1, \cdots, j_r \right)$. Let us take an edge of $\mathrm{Rect}_{\left(0, \cdots, 0\right)}^{\left(n_1, \cdots, n_r \right)}$ with origin and terminus $\left(j_1, \cdots, j_r \right)$ and $\left(j_1^{\prime}, \cdots, j_r^{\prime} \right)$ respectively. We may assume $j_i^{\prime}-j_i=1$ when $i=s$ and $j_i^{\prime}=j_i$ when $i\neq s$ for some $1 \leq s \leq n$. Then by our construction, we have $T\left(j_1, \cdots, j_r \right)_{\mathfrak{p}}=T\left(j_1^{\prime}, \cdots, j_r^{\prime} \right)_{\mathfrak{p}}$ when $\mathfrak{p}\neq\mathfrak{p}_s$ and $\left.T\left(j_1^{\prime}, \cdots, j_r^{\prime} \right)_{\mathfrak{p}_s}\right/{T\left(j_1, \cdots, j_r \right)_{\mathfrak{p}_s}}$ is isomorphic to $\left.R_{\mathfrak{p}_s}\right/{\mathfrak{p}_s}$. This implies 
\begin{equation*}
\mathrm{char}_{R}\left(\left.T\left(j_1^{\prime}, \cdots, j_r^{\prime} \right) \right/T\left(j_1, \cdots, j_r \right)   \right)=\mathfrak{p}_s.
\end{equation*}
Thus $\Phi_{\vartheta_1}$ is an isomorphism of graphs. 

Now we prove the equivalence between (1) and (2). Note that $(1) \Rightarrow (2)$ follows by our construction of $T\left(j_1, \cdots, j_r \right)$. Now we assume (2). We may assume $T=T\left(j_1, \cdots, j_r \right)$ and $T^{\prime}=y T\left(j_1^{\prime}, \cdots, j_r^{\prime} \right)$ for some $y \in K^{\times}$. By the assumption of Proposition \ref{Gal1}-(2) we have $\mathrm{ord}_{\mathfrak{p}}\left(y \right)=0$ for all $\mathfrak{p} \not\in \mathcal{N}$. Since $\left.T^{\prime}_{\mathfrak{p}}\right/{T_{\mathfrak{p}}}$ is cyclic and $G$ acts on $\left.T^{\prime}_{\mathfrak{p}}\right/{T_{\mathfrak{p}}}$ by $\vartheta_1$, we must have that $T^{\prime}_{\mathfrak{p}}$ lies in the chain \eqref{20200516}. By our construction of $T\left(j_1^{\prime}, \cdots, j_r^{\prime} \right)$, this implies that $y$ is a unit in $R$. Thus by localizing $T$ and $T^{\prime}$ at every $\mathfrak{p}_i$, we have $j_i \leq j_i^{\prime}$ and $\left(\left.T\left(j_1^{\prime}, \cdots, j_r^{\prime} \right)\right/{T\left(j_1, \cdots, j_r \right)}\right)_{\mathfrak{p}_i}$ is isomorphic to $\left.R_{\mathfrak{p}_i}\right/{\mathfrak{p}_i^{j_i^{\prime}-j_i}}\,\left(\vartheta_1 \right)$ by the property of the chain \eqref{20200516}. This completes the proof of Proposition \ref{Gal1}.
\end{proof}
Let us keep the notation which is used in the proof of Lemma \ref{redu}. Let $V\left(\xi \right)$ be the $K$-subspace of $V$ of dimension one which is generated by $v_{\xi}$ for $\xi \in \set{\psi, \psi^{\prime}}$. Then $V=V\left(\psi \right)\oplus V\left(\psi^{\prime} \right)$ as an $R[g_0]$-module. Let $g_{\psi}$ and $g_{\psi^{\prime}}$ be elements of $R[g_0]$ defined as follows
\begin{equation*}
g_{\psi}:=\dfrac{g_0-\lambda_{\psi^{\prime}}}{\lambda_{\psi}-\lambda_{\psi^{\prime}}},  g_{\psi^{\prime}}:=\dfrac{g_0-\lambda_{\psi}}{\lambda_{{\psi}^{\prime}}-\lambda_{\psi}}.
\end{equation*}
Then we have $g_{\psi}+g_{\psi^{\prime}}=\mathbf{1}_G$ and 
\begin{equation}
g_{\nu}v_{\xi}=\begin{cases}
0 & \left(\nu \neq \xi \right) \\
v_{\xi} & \left(\nu=\xi \right), \\
\end{cases}
\end{equation}
which implies that $g_{\xi}$ projects $V$ onto $V\left(\xi \right)$. For a stable lattice $T$, let $T\left(\xi \right):=g_{\xi}T$, we have $T=T\left(\psi \right)\oplus T\left(\psi^{\prime} \right)$ as an $R[g_0]$-module.

\begin{lem}\label{basic}
Let us keep the assumptions of Theorem \ref{Gal} and let $T$ be a stable lattice. Then we have $T\left(\xi \right)=T\cap V\left(\xi \right)$ for $\xi \in \set{\psi, \psi^{\prime}}$.
\end{lem}
\begin{proof}
The inclusion $T\left(\xi \right) \subset T\cap V\left(\xi \right)$ is obvious. Any element of $T\cap V\left(\xi \right)$ can be written can be written as $a v_{\xi}$, for some $a \in K$. Hence $av_{\xi}=\rho\left(g_{\xi} \right)av_{\xi} \in T\left(\xi \right)$. 
\end{proof}

\begin{lem}\label{frrkone}
Let us keep the assumption of Theorem \ref{Gal}. Let $T$ be a stable lattice which satisfies the condition (Lat-fr). Then the $R[g_0]$-modules $T\left(\psi \right)$ and $T\left(\psi^{\prime} \right)$ are free of rank one over $R$.

\end{lem}
\begin{proof}
We have the following split exact sequence 
\begin{equation}\label{20200914}
0 \rightarrow T\left(\psi^{\prime} \right) \rightarrow T \rightarrow T\left(\psi \right) \rightarrow 0.
\end{equation}
Hence the assertion for $T\left(\psi \right)$ and $T\left(\psi^{\prime} \right)$ are done exactly by the same argument. By taking the base extension $\otimes_{R}\left.R\right/{\mathfrak{m}}$, we have the following exact sequence
\begin{equation*}
T\left(\psi^{\prime} \right)\otimes_{R}\left.R\right/{\mathfrak{m}} \rightarrow T\otimes_{R}\left.R\right/{\mathfrak{m}} \rightarrow T\left(\psi \right)\otimes_{R}\left.R\right/{\mathfrak{m}}\rightarrow 0.
\end{equation*}
Since $T$ is free of rank two over $R$, the $\left.R\right/{\mathfrak{m}}$-vector space $T\otimes_{R}\left.R\right/{\mathfrak{m}}$ has dimension two. This implies that the dimension of $T\left(\psi \right)\otimes_{R}\left.R\right/{\mathfrak{m}}$ is less than or equal to two. On the other hand, $T\left(\psi \right)\otimes_{R}\left.R\right/{\mathfrak{m}}$ is nontrivial by Nakayama's lemma. Thus the dimension of $T\left(\psi \right)\otimes_{R}\left.R\right/{\mathfrak{m}}$ is one or two.

Assume that $T\left(\psi \right)\otimes_{R}\left.R\right/{\mathfrak{m}}$ has $\left.R\right/{\mathfrak{m}}$-dimension two. Then  $T\left(\psi \right)\otimes_{R}\left.R\right/{\mathfrak{m}}$ is isomorphic to $T\otimes_{R}\left.R\right/{\mathfrak{m}}$. This contradicts to that the eigenvalues of the action of $g_0$ on $T\otimes_{R}\left.R\right/{\mathfrak{m}}$ are distinct. Thus $T\left(\psi \right)\otimes_{R}\left.R\right/{\mathfrak{m}}$ has $\left.R\right/{\mathfrak{m}}$-dimension one, which implies that $T\left(\psi \right)$ is a cyclic $R$-module by Nakayama's lemma. Furthermore, since $T\left(\psi \right)$ is $R$-torsion free, we have that $T\left(\psi \right)$ is a free $R$-module of rank one. 
\end{proof}

In general, a reflexive lattice may not be a free lattice if the Krull dimension of $R$ is greater than two. However, we have the following result for stable lattices. 

\begin{prp}\label{reffr}
Let us keep the assumptions and the notation of Theorem \ref{Gal}. Then under the assumptions (Lat-fr) and ($G$-dist), every stable reflexive lattice is free over $R$. 
\end{prp}
\begin{proof}
Let $T$ be a stable lattice which satisfies (Lat-fr). By Proposition \ref{Gal1}, the homothety classes of stable reflexive lattices is isomorphic to the rectangle $\mathrm{Rect}_{\left(0, \cdots, 0 \right)}^{\left(n_1, \cdots, n_r \right)}$. Let $\left(j_1, \cdots, j_r \right)$ be the vertex corresponding to $[T]$ under the map $\Phi_{\vartheta}$ in \ref{Gal1}. Let us admit the following claim for a while.
\begin{cla}\label{c1}
For the vertex $\left(j_1^{\prime}, \cdots, j_r^{\prime} \right) \geq \left(j_1, \cdots, j_r \right)$ or $\left(j_1^{\prime}, \cdots, j_r^{\prime} \right) \leq \left(j_1, \cdots, j_r \right)$, the corresponding reflexive lattice is free. 
\end{cla}
Then since the graph $\mathrm{Rect}_{\left(0, \cdots, 0 \right)}^{\left(n_1, \cdots, n_r \right)}$ is partially ordered and has maximal (and minimal) vertex, Proposition \ref{reffr} follows immediately by Claim \ref{c1}.

We will show Claim \ref{c1} in the rest of the proof. Since the proof for the vertexes $\left(j_1^{\prime}, \cdots, j_r^{\prime} \right) \geq \left(j_1, \cdots, j_r \right)$ and $\left(j_1^{\prime}, \cdots, j_r^{\prime} \right) \leq \left(j_1, \cdots, j_r \right)$ are done by the same argument, we only consider the case when $\left(j_1^{\prime}, \cdots, j_r^{\prime} \right) \geq \left(j_1, \cdots, j_r \right)$. By Proposition \ref{Gal1}, there exists a representative $T^{\prime}$ of $\Phi_{\vartheta}^{-1}\left( \left(j_1^{\prime}, \cdots, j_r^{\prime} \right)   \right)$ such that $T^{\prime} \supset T$, that $T^{\prime}_{\mathfrak{p}}=T_{\mathfrak{p}}$ for all $\mathfrak{p} \in \mathrm{Spec}_{\mathrm{ht}=1}\left(R \right)\setminus \mathcal{N}$ and that $\left(\left.T^{\prime}\right/{T} \right)_{\mathfrak{p}_i}$ is isomorphic to $\left.R_{\mathfrak{p}_i}\right/{\mathfrak{p}_i^{j_i^{\prime}-j_i}}\left(\vartheta \right)$ for any $\mathfrak{p}_i \in \mathcal{N}$. Since $T^{\prime}=T^{\prime}\left(\psi \right)\oplus T^{\prime}\left(\psi^{\prime} \right)$, it is sufficient to prove that $T^{\prime}\left(\psi \right)$ and $T^{\prime}\left(\psi^{\prime} \right)$ are $R$-free. By the definition of $T^{\prime}\left(\psi \right)$ and $T^{\prime}\left(\psi^{\prime} \right)$, we have $T\left(\psi \right) \subset T^{\prime}\left(\psi \right)$ and $T\left(\psi^{\prime} \right) \subset T^{\prime}\left(\psi^{\prime} \right)$. Then we have the following commutative diagram of $R[g_0]$-modules
\begin{equation}\label{2020660616}
 \xymatrix{
   0 \ar[r]   & T\left(\psi \right) \ar[r] \ar@{^{(}-_>}[d] & T \ar[r] \ar@{^{(}-_>}[d] & T\left(\psi^{\prime} \right) \ar[r] \ar@{^{(}-_>}[d] & 0 \\
    0 \ar[r] & T^{\prime}\left(\psi \right) \ar[r] & T^{\prime} \ar[r] & T^{\prime}\left(\psi^{\prime} \right) \ar[r] & 0.
  }
\end{equation}
Thus we have the following exact sequence
\begin{equation}\label{202009170}
0 \rightarrow \left.T^{\prime}\left(\psi \right)\right/{T\left(\psi \right)}\rightarrow \left.T^{\prime}\right/{T} \rightarrow \left.T^{\prime}\left(\psi^{\prime} \right)\right/{T\left(\psi^{\prime} \right)} \rightarrow 0
\end{equation}
by the snake lemma. Since the horizontal exact sequences in \eqref{2020660616} split, we may assume $\vartheta\,\mod\,\mathfrak{m}=\psi$ without loss of generality. Then since $\psi\left(g_0 \right)\neq\psi^{\prime}\left(g_0 \right)$, by taking base extension $\otimes_{R}R_{\mathfrak{p}}$ for any $\mathfrak{p} \in \mathrm{Spec}_{\mathrm{ht}=1}\left(R \right)$ we have $T\left(\psi^{\prime} \right)_{\mathfrak{p}}=T^{\prime}\left(\psi^{\prime} \right)_{\mathfrak{p}}$ for any $\mathfrak{p} \in \mathrm{Spec}_{\mathrm{ht}=1}\left(R \right)$. Since $T^{\prime}$ (resp. $T$) is reflexive, so is the $R$-module $T^{\prime}\left(\psi^{\prime} \right)$ (resp. $T\left(\psi^{\prime} \right)$) by Lemma \ref{basic}, hence we have
\begin{equation}\label{202009171}
T\left(\psi^{\prime} \right)=T^{\prime}\left(\psi^{\prime} \right). 
\end{equation}
Then since $T\left(\psi^{\prime} \right)$ is a free $R$-module of rank one by Lemma \ref{frrkone}, so is the $R$-module $T^{\prime}\left(\psi^{\prime} \right)$. Now we prove the freeness of $T^{\prime}\left(\psi \right)$. Again by taking base extension $\otimes_{R}R_{\mathfrak{p}}$, we have $T^{\prime}\left(\psi \right)_{\mathfrak{p}}=T\left(\psi \right)_{\mathfrak{p}}$ for all $\mathfrak{p} \in \mathrm{Spec}_{\mathrm{ht}=1}\left(R \right)$ and $\left(\left.T^{\prime}\left(\psi \right)\right/T\left(\psi \right) \right)_{\mathfrak{p}_i} \stackrel{\sim}{\rightarrow} \left(\left.T^{\prime}\right/{T} \right)_{\mathfrak{p}_i} \stackrel{\sim}{\rightarrow} \left.R_{\mathfrak{p}_i}\right/{\mathfrak{p}_i^{j_i^{\prime}-j_i}}$ as an $R_{\mathfrak{p}_i}$-module. Since $R_{\mathfrak{p}_i}$ is a discrete valuation ring and $T^{\prime}\left(\psi \right)_{\mathfrak{p}_i}$ is an $R_{\mathfrak{p}_i}$-lattice of $V\left(\psi \right)$, we have that $T^{\prime}\left(\psi \right)_{\mathfrak{p}_i}$ is $R_{\mathfrak{p}_i}$-free of rank one. Thus $\mathfrak{p}_i^{j_i^{\prime}-j_i}T^{\prime}\left(\psi \right)_{\mathfrak{p}_i}=T\left(\psi \right)_{\mathfrak{p}_i}$. On the other hand, since $T^{\prime}\left(\psi \right)$ and $T\left(\psi \right)$ are reflexive by Lemma \ref{basic}, we have the equality
\begin{equation}\label{202009172}
\displaystyle\prod_{i=1}^{r}\mathfrak{p}_i^{j_i^{\prime}-j_i}T^{\prime}\left(\psi \right)=T\left(\psi \right).
\end{equation}
Thus $T^{\prime}\left(\psi \right)$ is also $R$-free of rank one by Lemma \ref{frrkone}. This completes the proof of Claim \ref{c1}.

\end{proof}
We complete the proof of Theorem \ref{Gal}
\begin{proof}[Proof of Theorem \ref{Gal}]
By Proposition \ref{reffr}, the graphs $C^{\mathrm{fr}}\left(\rho \right)$ and $C^{\mathrm{ref}}\left(\rho \right)$ coincide. The assertion (2) $\Rightarrow$ (1) follows by Proposition \ref{Gal1}. Thus it remains to prove the assertion (1) $\Rightarrow$ (2). By Proposition \ref{Gal1}, there exist reflexive lattices $T$ and $T^{\prime}$ of $\Phi_{\vartheta_1}^{-1}\left( \left(j_1, \cdots, j_r \right)   \right)$ and $\Phi_{\vartheta_1}^{-1}\left( \left(j_1^{\prime}, \cdots, j_r^{\prime} \right)   \right)$ respectively such that $T \subset T^{\prime}$, that $T^{\prime}_{\mathfrak{p}}=T_{\mathfrak{p}}$ for all $\mathfrak{p} \in \mathrm{Spec}_{\mathrm{ht}=1}\left(R \right)\setminus \mathcal{N}$ and that $\left(\left.T^{\prime}\right/{T^{\prime}} \right)_{\mathfrak{p}_i}$ is isomorphic to $\left.R_{\mathfrak{p}_i}\right/{\mathfrak{p}_i^{j_i^{\prime}-j_i}}\left(\vartheta_1 \right)$ for any $\mathfrak{p}_i \in \mathcal{N}$. Then we have $\left.T^{\prime}\right/T \stackrel{\sim}{\rightarrow} \left.T^{\prime}\left(\psi \right)\right/{\displaystyle\prod_{i=1}^r \mathfrak{p}_i^{j_i^{\prime}-j_i} T^{\prime}\left(\psi \right)}$ by combining the equality \eqref{202009170}, \eqref{202009171} and \eqref{202009172}. By Proposition \ref{reffr}, $T$ and $T^{\prime}$ are free over $R$. Then by Lemma \ref{frrkone}, $T^{\prime}\left(\psi \right)$ is $R$-free of rank one. Thus we have 
\begin{equation*}
\left.T^{\prime}\right/T \stackrel{\sim}{\rightarrow} \left.R\right/{\displaystyle\prod_{i=1}^r \mathfrak{p}_i^{j_i^{\prime}-j_i}}\,\left(\vartheta \right)
\end{equation*}
as $R[g_0]$-modules. Moreover, $G$ acts on $\left.T^{\prime}\right/T$ by one of the characters $\vartheta$ and $\vartheta^{\prime}$ by Lemma \ref{21}. Thus we completes the proof since $\vartheta\left(g_0 \right)\neq\vartheta^{\prime}\left(g_0 \right)$ .
\end{proof}

\begin{example}\label{Til}
In this example, we consider the case when $\rho$ comes from a Hida family. Let $\mathcal{F}$ be a normalized $\mathbb{I}$-adic eigen cusp form and $\mathbb{T}$ the Galois representation attached to $\mathcal{F}$. Note that $\mathbb{T}$ is a finitely generated torsion-free $\mathbb{I}$-module and $\mathbb{V}_{\mathcal{F}}:=\mathbb{T}\otimes_{\mathbb{I}}\mathbb{K}$ has $\mathbb{K}$-dimension equal to two. Since $\mathbb{I}$ has Krull dimension two, we have that when $\mathbb{I}$ is regular, every reflexive lattice is free over $\mathbb{I}$. Now we consider a weaker condition when $\mathbb{I}$ is UFD and Gorenstein. Then the Gorenstein's assumption (with some extra assumptions on the Neben character, see \S 4.2 below) implies the existence of a stable free lattice by a result of Tilouine \cite[Theorem 4.4]{T}. Thus Proposition \ref{reffr} tells us that every stable reflexive lattice is free even $\mathbb{I}$ is not regular and Theorem \ref{Gal} tells us that the graph $C^{\mathrm{fr}}\left(\rho_{\mathcal{F}} \right)$ is connected and rectangle. 

Now let us consider the two-variable case. Let $\mathcal{V}=\mathbb{T}\hat{\otimes}_{\mathbb{Z}_p}\mathbb{Z}_p[[\Gamma]]\left(\tilde{\kappa}^{-1} \right)\otimes_{\mathcal{R}}\mathcal{K}$, where $\tilde{\kappa}$ denotes the character $G_{\mathbb{Q}} \twoheadrightarrow \Gamma \hookrightarrow \mathbb{Z}_p[[\Gamma]]^{\times}$, $\mathcal{R}:=\mathbb{I}[[\Gamma]]$ and $\mathcal{K}$ is the field of fractions of $\mathcal{R}$. Since $\mathcal{R}$ has Krull dimension three, a reflexive lattice may not be free over $\mathcal{R}$ even $\mathcal{R}$ is regular. However under the above assumptions we also have that any stable reflexive lattice is free.  On the other hand, by definition we have that the ideal of reducibility of $\rho_{\mathcal{F}}^{\mathrm{n. ord}}$ is generated by the the ideal of reducibility of $\rho_{\mathcal{F}}$. Thus by applying Theorem \ref{Gal} to $\rho_{\mathcal{F}}$ and $\rho_{\mathcal{F}}^{\mathrm{n. ord}}$, we know that every stable free lattice of $\mathcal{V}$ comes from stable free lattice of $\mathbb{V}_{\mathcal{F}}$. 

\end{example}

\section{Application to Selmer groups for residually reducible Hida deformation}
\subsection{Perrin-Riou and Ochiai's formula}
In this section, we recall a formula of  Perrin-Riou \cite{PR} Ochiai \cite{Ochiai08} on calculating the difference of Selmer groups for different choices of stable lattices in Galois deformation. Then by applying Theorem \ref{Gal}, we study how dose Selmer group change concretely when a stable lattice varies in two-variable Hida deformation.

Let $\mathscr{R}$ be an integrally closed local domain which is finite flat over $\mathbb{Z}_p[[X_1, \cdots, X_n]]$ with $\mathscr{M}$ the maximal ideal of $\mathscr{R}$ and $\mathscr{K}$ the field of fractions. Let $\mathscr{V}$ be a finite-dimensional $\mathscr{K}$-vector space and $$\rho: G_{\mathbb{Q}} \rightarrow \mathrm{Aut}_{\mathscr{K}}\left(\mathscr{V} \right)$$ a linear representation such that 
\begin{itemize}
\item The representation $\rho$ has a stable lattice $\mathscr{T}$.
\item The action of $G_{\mathbb{Q}}$ on $\mathscr{T}$ is continuous with respect to the $\mathscr{M}$-adic topology on $\mathscr{T}$.
\item The the action of $G_{\mathbb{Q}}$ on $\mathscr{T}$ is unramified outside a finite set of primes $\Sigma \supset \set{p, \infty}$. 
\end{itemize}
We fix the above $\left(\mathscr{R}, \mathscr{V}, \rho \right)$ to the end of this subsection. 
\begin{dfn}\label{7131}
Suppose that we have a $G_{\mathbb{Q}_p}$-stable subspace $F^{+}\mathscr{V}$ of $\mathscr{V}$. For a stable lattice $\mathscr{T}$, let $F^{+}\mathscr{T}:=\mathscr{T}\cap F^{+}\mathscr{V}$ and $F^{+}\mathscr{A}:=F^{+}\mathscr{T}\otimes_{\mathscr{R}}\mathscr{R}^{\lor}$ in $\mathscr{A}:=\mathscr{T}\otimes_{\mathscr{R}}\mathscr{R}^{\lor}$. We define the Selmer group $\mathrm{Sel}_{\mathscr{A}}$ for $\mathscr{T}$ as follows
\begin{equation*}
\mathrm{Sel}_{\mathscr{A}}:=\mathrm{Ker}\left[H^{1}\left(\mathbb{Q}_{\Sigma}/\mathbb{Q}, \mathscr{A} \right) \rightarrow \displaystyle\prod_{v \in \Sigma\setminus\{p, \infty\}}H^{1}\left(I_v, \mathscr{A} \right)\times H^{1}\left(I_p, \left.\mathscr{A}\right/F^{+}\mathscr{A} \right) \right].
\end{equation*}

\end{dfn}

\begin{dfn}[{\cite[Definition 1.5]{Ochiai08}}]\label{O1}
Let us introduce the following conditions on a stable lattice $\mathscr{T}$:
\begin{enumerate}
\item[$(\mathbf{F}_{\mathbb{Q}})$]The coinvariant quotient $\left(\mathscr{T}^{*} \right)_{G_{\mathbb{Q}}}$ is a pseudo-null $\mathscr{R}$-module.
\item[$(\mathbf{F_p})$]$I_p$ acts non-trivially on every element of $\mathscr{T}/F^{+}\mathscr{T}$. $G_{\mathbb{Q}_p}$ acts non-trivially on every elements of $F^{+}\mathscr{T}\left(-1 \right)$ and $\mathscr{T}\left(-1 \right)/F^{+}\mathscr{T}\left(-1 \right)$, where $\mathscr{T}\left(-1 \right):=\mathscr{T}\otimes_{\mathbb{Z}_p}\mathbb{Z}_p\left(-1 \right)$. $\left(\left(\mathscr{A}/F^{+}\mathscr{A} \right)^{G_{\mathbb{Q}_p}} \right)^{\lor}$ is a pseudo-null $\mathscr{R}$-module.
\item[$(\mathbf{T})$]$\left(\mathrm{Sel}_{\mathscr{A}} \right)^{\lor}$ is a torsion $\mathscr{R}$-module. 
\item[$(\mathbf{T}^{*})$]$\left(\mathrm{Sel}_{\mathscr{T}^{*}\left(1 \right)\otimes_{\mathscr{R}}\mathscr{R}^{\lor}} \right)^{\lor}$ is a torsion $\mathscr{R}$-module, where $\mathscr{T}^{*}\left(1 \right)=\mathscr{T}^{*}\otimes_{\mathbb{Z}_p} \mathbb{Z}_p\left(1 \right)$.
\end{enumerate}
Furthermore, for each prime $l \in \Sigma\setminus\set{p, \infty}$, we introduce the following condition:
\begin{enumerate}
\item[$(\mathbf{F_v})$]$G_{\mathbb{Q}_v}$ acts non-trivially on every element of $\mathscr{T}\left(-1 \right)$. The modules $\mathscr{A}^{G_{\mathbb{Q}_v}}$ and $\left(\mathscr{A}^{I_v} \right)_{G_{\mathbb{Q}_v}}$ are pseudo-null $\mathscr{R}$-modules. 
\end{enumerate}
\end{dfn}

\begin{rem}[{\cite[Remark 1.7]{Ochiai08}}]\label{O2}
The conditions $(\mathbf{F}_{\mathbb{Q}}), (\mathbf{F_p}), (\mathbf{T}), (\mathbf{T}^{*})$ and $(\mathbf{F_v})$ for every $v \in \Sigma\setminus\set{p, \infty}$ are satisfied for the following deformation $\mathscr{T}$ and its all sub-lattices:
\begin{enumerate}
\renewcommand{\labelenumi}{(\arabic{enumi})}
\item $\mathscr{T}=T\otimes_{\mathbb{Z}_p}\mathbb{Z}_p[[\Gamma]]\left(\tilde{\kappa}^{-1} \right)\otimes_{\mathbb{Z}_p}\omega^i$, where $T$ is a stable lattice of the Galois representation attached to a $p$-ordinary normalized eigen cusp form. Recall that $\tilde{\kappa}$ is the character $\tilde{\kappa}: G_{\mathbb{Q}} \twoheadrightarrow \Gamma \hookrightarrow \mathbb{Z}_p[[\Gamma]]^{\times}$.
\item $\mathscr{T}=\mathbb{T}\hat{\otimes}_{\mathbb{Z}_p}\mathbb{Z}_p[[\Gamma]]\left(\tilde{\kappa}^{-1} \right)\otimes_{\mathbb{Z}_p}\omega^i$, where $\mathbb{T}$ is the Galois representation attached to an $\mathbb{I}$-adic normalized eigen cusp form and $\mathbb{I}$ is isomorphic to $\mathcal{O}[[X]]$ for the ring of integers of a finite extension of $\mathbb{Q}_p$.
\end{enumerate}
\end{rem}

\begin{rem}\label{hold not Hida}
We remark that the condition $(\mathbf{F_p})$ does not hold for the one-variable Hida deformation $\mathbb{T}$ since under the action of $G_{\mathbb{Q}_p}$ on $\left.\mathbb{T}\right/{F^{+}\mathbb{T}}$ is unramified due to Wiles \cite[Theorem 2.2.2]{Wi88}. Also by Proposition \ref{H0} below, the condition $(\mathbf{F}_{\mathbb{Q}})$ dose not hold for some stable free lattices (in fact it holds for only one free lattice by Proposition \ref{H0}). 
\end{rem}

Let $\mathscr{T}$ and $\mathscr{T}^{\prime}$ be $G_{\mathbb{Q}}$-stable $\mathscr{R}$-free lattices of $\mathscr{V}$. The following comparison formula enables us to calculate the difference between $\mathrm{char}_{\mathscr{R}}\left(\mathrm{Sel}_{\mathscr{A}}\right)^{\lor}$ and $\mathrm{char}_{\mathscr{R}}\left(\mathrm{Sel}_{\mathscr{A}^{\prime}}\right)^{\lor}$. Note that the formula is first proved by Schneider \cite{PS} for the case where $\mathscr{T}$ is the cyclotomic deformation of a $p$-adic representation attached to an abelian variety and by Perrin-Riou \cite{PR} for the case where $\mathscr{T}$ is the cyclotomic deformation of an ordinary $p$-adic representation. Then Ochiai \cite{Ochiai08} generalized the formula to more general Galois deformation. 
\begin{thm}[Ochiai {\cite[Theorem 1.6]{Ochiai08}}]\label{O}
Suppose that $\mathscr{R}$ is isomorphic to $\mathcal{O}[[X_1, \cdots, X_n]]$, where $\mathcal{O}$ is the ring of integers of a finite extension of $\mathbb{Q}_p$. Let us take $G_{\mathbb{Q}}$-stable $\mathscr{R}$-free lattices $\mathscr{T}$ and $\mathscr{T}^{\prime}$ with $\mathscr{T}^{\prime} \subset \mathscr{T}$. Assume the conditions $(\mathbf{F}_{\mathbb{Q}}), (\mathbf{F_p}), (\mathbf{T}), (\mathbf{T}^{*})$ and $(\mathbf{F_v})$ for every $v \in \Sigma\setminus\set{p, \infty}$ hold for $\mathscr{T}$ and $\mathscr{T}^{\prime}$. Then we have the following equality:
\begin{equation*}
\dfrac{\mathrm{char}_{\mathscr{R}}\left(\mathrm{Sel}_{\mathscr{A}}\right)^{\lor}}{\mathrm{char}_{\mathscr{R}}\left(\mathrm{Sel}_{\mathscr{A}^{\prime}}\right)^{\lor}}=\displaystyle\prod_{\mathscr{P}\in \mathrm{Spec}_{\mathrm{ht}=1}\left(\mathscr{R} \right)}\mathscr{P}^{\mathrm{length}_{\mathscr{R}_{\mathscr{P}}}\left(\left(\mathscr{T}/\mathscr{T}^{\prime} \right)_{G_{\mathbb{R}}} \right)_{\mathscr{P}}-\mathrm{length}_{\mathscr{R}_{\mathscr{P}}}\left(\left(F^{+}\mathscr{T}/F^{+}\mathscr{T}^{\prime} \right) \right)_{\mathscr{P}}} .
\end{equation*}

\end{thm}

\subsection{Application to two-variable Hida deformations}

In this subsection, by applying Theorem \ref{Gal}, we give an example of Theorem \ref{O} to study how does the Selmer group change in a two-variable Hida deformation with reducible residual representation for our later use. First let us prepare some notation on Hida deformation. We fix an integer $N \in \mathbb{Z}_{\geq 1}$ which is prime to $p$ and $\chi$ a Dirichlet character of conductor $N$. We denote by $\mathbf{H}\left(N, \chi, \mathbb{Z}_p[\chi] \right)^{\mathrm{ord}}$ and $\mathbf{h}\left(N, \chi, \mathbb{Z}_p[\chi] \right)^{\mathrm{ord}}$ be the universal Hida's Hecke algebra and cuspidal Hecke algebra respectively on which $\left(\left.\mathbb{Z}\right/{Np}\mathbb{Z} \right)^{\times}$ acts by $\chi$. Let $\Lambda_{\chi}:=\mathbb{Z}_p[\chi][[\Gamma^{\prime}]]$ where $\Gamma^{\prime}$ is the $p$-Sylow subgroup of the group of diamond operators for the tower of modular curves $\left\{Y_1\left(p^t\right)\right\}_{t\geq 1}$ as in \S 1. We omit the subscript $\chi$ of $\Lambda_{\chi}$ when the image of $\chi$ is contained in $\mathbb{Z}_p$.

We denote by $\mathcal{E}_{\chi}=\displaystyle\sum_{n=0}^{\infty}a\left(n, \mathcal{E}_{\chi} \right)q^n \in \Lambda_{\chi}[[q]]$ the $\Lambda_{\chi}$-adic Eisenstein series $\mathcal{E}\left(\chi\omega^{-1}, \mathbf{1} \right)$ defined in \cite[\S 1.2]{MO05}. Let $I$ be the ideal of $\mathbf{H}\left(N, \chi, \mathbb{Z}_p[\chi] \right)^{\mathrm{ord}}$ which is generated by $T_l-a\left(l, \mathcal{E}_{\chi} \right)$ for all primes $l \nmid Np$ and by $U_p-1$. Let $\mathfrak{M}$ be a maximal ideal of $\mathbf{H}\left(N, \chi, \mathbb{Z}_p[\chi] \right)^{\mathrm{ord}}$ generated by $I, \gamma^{\prime}-1, \varpi$ for a fixed uniformizer of $\mathbb{Z}_p[\chi]$. Let $\mathbb{I}$ be a quotient of $\mathbf{h}\left(N, \chi, \mathbb{Z}_p[\chi] \right)^{\mathrm{ord}}_{\mathfrak{M}}$ which corresponds to a certain normalized $\mathbb{I}$-adic eigen cusp form $\mathcal{F}$. The algebra $\mathbb{I}$ is a local domain which is finite flat over $\Lambda_{\chi}$. Let $\mathbb{K}$ be the field of fractions of $\mathbb{I}$ and we replace $\mathbb{I}$ with its integral closure in $\mathbb{K}$ from now on to the end of this paper.

Let $\rho_{\mathcal{F}}$ be the Galois representation attached to $\mathcal{F}$. We assume that the representation $\rho_{\mathcal{F}}$ satisfies the condition (Lat-fr) in \S 1 throughout the paper. The condition (Lat-fr) is satisfied in the following cases
\begin{enumerate}
\item[(1)]The algebra $\mathbb{I}$ is a regular local ring.
\item[(2)]The algebra $\mathbb{I}$ is Gorenstein and $\left.\chi\right|_{\left(\left.\mathbb{Z}\right/p\mathbb{Z} \right)^{\times}} \neq \mathbf{1}, \omega$ (Tilouine \cite[Theorem 4.4]{T}). 
\end{enumerate}
Then the residual representation $\overline{\rho}_{\mathcal{F}}$ modulo the maximal ideal $\mathfrak{m}_{\mathbb{I}}$ of $\mathbb{I}$ is defined by the semi-simplification of $\left.\mathbb{T}\right/\mathfrak{m}_{\mathbb{I}}\mathbb{T}$ for a stable free lattice $\mathbb{T}$ of $\rho_{\mathcal{F}}$. We assume $\left.\chi\right|_{\left(\left.\mathbb{Z}\right/p\mathbb{Z} \right)^{\times}} \neq \omega^{-1}$ from now on to the end of this paper. Then since $\mathcal{F}$ corresponds to the algebra homomorphism $\mathbf{h}\left(N, \chi, \mathbb{Z}_p[\chi] \right)^{\mathrm{ord}}_{\mathfrak{M}} \rightarrow \mathbb{I}$, by the duality theorem \cite[Proposition 1.5.3]{MO05} we have that $\overline{\rho}_{\mathcal{F}}$ is isomorphic to $\mathbf{1}\oplus\chi$.

Recall that for a stable lattice $\mathbb{T}$, we denote by $\mathcal{T}$ and $\mathcal{A}$ the $\mathcal{R}$-modules $\mathbb{T}\hat{\otimes}_{\mathbb{Z}_p}\mathbb{Z}_p[[\Gamma]]\left(\tilde{\kappa}^{-1} \right)$ and $\mathcal{T}\otimes_{\mathcal{R}}\mathcal{R}^{\lor}$ respectively, where $\mathcal{R}:=\mathbb{I}[[\Gamma]]$. Suppose $p \nmid \varphi\left(N \right)$, then by class field theory, the ideal of reducibility $J\left(\rho_{\mathcal{F}} \right)$ coincides with the ideal $\mathbb{J}$ which is generated by $a\left(l, \mathcal{F} \right)-1-a\left(l, \mathcal{E}_{\chi}\right)$ for all primes $l \nmid Np$ and $a\left(p, \mathcal{F} \right)-1$. Let $\mathscr{L}_p^{\mathrm{alg}}\left(\rho_{\mathcal{F}}^{\mathrm{n.ord}, \left(i \right)} \right)$ be the set of all $\mathrm{char}_{\mathcal{R}}\left(\mathrm{Sel}_{\mathcal{A}\otimes\omega^i} \right)^{\lor}$ when $\mathbb{T}$ varies in the set of all stable lattices (free and non-free) of $\mathbb{V}_{\mathcal{F}}$. We obtain the following result. 

\begin{cor}\label{C2}
Suppose that $\mathbb{I}$ is isomorphic to $\mathcal{O}[[X]]$ for the ring of integers of a finite extension of $\mathbb{Q}_p$. Assume $p \nmid \varphi\left(N \right)$ and that $\left.\chi\right|_{\left(\left.\mathbb{Z}\right/{p\mathbb{Z}} \right)^{\times}}$ is nontrivial. We decompose $\mathbb{J}^{**}$ into the product of height-one prime ideals as $\mathbb{J}^{**}=\mathfrak{p}_1^{n_1}\cdots\mathfrak{p}_r^{n_r}$, where $\mathfrak{p}_i$ and $\mathfrak{p}_j$ are distinct when $i\neq j$. Let $\Phi_{\mathbf{1}}$ be the isomorphism in Theorem \ref{Gal} and $\mathbb{T}^{\mathrm{min}}$ a representative of $\Phi_{\mathbf{1}}^{-1}\left(\left(0, \cdots, 0 \right) \right)$. Let $i$ be an even integer, we have the following equality:
\begin{equation*}
\mathscr{L}_p^{\mathrm{alg}}\left(\rho_{\mathcal{F}}^{\mathrm{n.ord}, \left(i \right)} \right)=
\left\{ \left.\mathrm{char}_{\mathcal{R}}\left(\mathrm{Sel}_{\mathcal{A}^{\mathrm{min}}\otimes\omega^i} \right)^{\lor}\cdot\mathfrak{p}_1^{j_1}\cdots\mathfrak{p}_r^{j_r}\mathcal{R} \right| j_s=0, \cdots, n_i, s=1, \cdots r \right\}.
\end{equation*}
\end{cor}

\begin{proof}
Let $\alpha$ (resp. $\beta$) be the character of $G_{\mathbb{Q}_p}$ such that $G_{\mathbb{Q}_p}$ acts on $F^{+}\mathbb{V}_{\mathcal{F}}$ (resp. $F^{-}\mathbb{V}_{\mathcal{F}}:=\left.\mathbb{V}_{\mathcal{F}}\right/{F^{+}\mathbb{V}_{\mathcal{F}}}$) by $\alpha$ (resp. $\beta$). We have $\set{\overline{\alpha}, \overline{\beta} }=\set{\mathbf{1}, \overline{\chi} }$. Under the assumption that $\left.\chi\right|_{\left(\left.\mathbb{Z}\right/{p\mathbb{Z}} \right)^{\times}}$ is nontrivial, we have that $\overline{\alpha}\overline{\beta}$ is ramified at $p$. Since $\overline{\beta}$ is unramified, we have $\overline{\alpha}=\overline{\chi}$ and $\overline{\beta}=\mathbf{1}$. This implies that the character $\alpha$ is odd and $\beta$ is even.

Let $\mathbb{T}$ be a stable lattice. By \cite[Lemma 4.3]{Ochiai08} it is sufficient to assume that $\mathbb{T}$ is a free lattice. Let $\left(j_1, \cdots, j_r \right)$ be a vertex of $\rect$ corresponding to $[\mathbb{T}]$ under $\Phi_{\mathbf{1}}$. Then by Theorem \ref{Gal}-(2), we have the isomorphism $\left.\mathbb{T}\right/{\mathbb{T}^{\mathrm{min}}}\stackrel{\sim}{\rightarrow} \left.\mathbb{I}\right/{\mathfrak{p}_1^{j_1}\cdots\mathfrak{p}_r^{j_r}}\left(\mathbf{1} \right)$ by multiplying $\mathbb{T}$ by elements of $\mathbb{K}^{\times}$. Let us consider the following commutative diagram of $G_{\mathbb{Q}_p}$-modules:
\begin{equation*}
 \xymatrix{
   0 \ar[r]   & F^{+}\mathbb{T}^{\mathrm{min}} \ar[r] \ar@{^{(}-_>}[d] & \mathbb{T}^{\mathrm{min}} \ar[r] \ar@{^{(}-_>}[d] & F^{-}\mathbb{T}^{\mathrm{min}} \ar[r] \ar[d] & 0 \\
    0 \ar[r] & F^{+}\mathbb{T} \ar[r] & \mathbb{T} \ar[r] & F^{-}\mathbb{T} \ar[r] & 0,
  }
\end{equation*}
where $F^{-}\mathbb{T}:=\left.\mathbb{T}\right/{F^{+}\mathbb{T}}$. Since $F^{-}\mathbb{T}^{\mathrm{min}}$ is isomorphic to the image of the map $\mathbb{T}^{\mathrm{min}} \hookrightarrow \mathbb{V}_{\mathcal{F}} \twoheadrightarrow F^{-}\mathbb{V}_{\mathcal{F}}$ which implies that $F^{-}\mathbb{T}^{\mathrm{min}}$ is $\mathbb{I}$-torsion free. Thus by the snake lemma, the map $F^{-}\mathbb{T}^{\mathrm{min}} \rightarrow F^{-}\mathbb{T}$ is injective and we have
\begin{equation*}
0 \rightarrow \left.F^{+}\mathbb{T}\right/{F^{+}\mathbb{T}^{\mathrm{min}}}\,\left(\alpha \right) \rightarrow \left.\mathbb{T}\right/{\mathbb{T}^{\mathrm{min}}}\,\left(\mathbf{1} \right) \rightarrow \left.F^{-}\mathbb{T}\right/{F^{-}\mathbb{T}^{\mathrm{min}}}\,\left(\beta \right) \rightarrow 0.
\end{equation*}
Since $\alpha$ is odd, we have $F^{+}\mathbb{T}=F^{+}\mathbb{T}^{\mathrm{min}}$. Then by applying Theorem \ref{O} to $\mathcal{T}\otimes\omega^i:=\mathbb{T}\hat{\otimes}_{\mathbb{Z}_p}\mathbb{Z}_p[[\Gamma]]\left(\tilde{\kappa}^{-1} \right)\otimes\omega^i$ and $\mathcal{T}^{\mathrm{min}}\otimes\omega^i:=\mathbb{T}^{\mathrm{min}}\hat{\otimes}_{\mathbb{Z}_p}\mathbb{Z}_p[[\Gamma]]\left(\tilde{\kappa}^{-1} \right)\otimes\omega^i$, we have 
\begin{equation}\label{20200609}
\dfrac{\mathrm{char}_{\mathcal{R}}\left(\mathrm{Sel}_{\mathcal{A}\otimes\omega^i}\right)^{\lor}}{\mathrm{char}_{\mathcal{R}}\left(\mathrm{Sel}_{\mathcal{A}^{\mathrm{min}}\otimes\omega^i}\right)^{\lor}}=\displaystyle\prod_{s=1}^r\mathfrak{P}_s^{\mathrm{length}_{\mathcal{R}_{\mathfrak{P}_s}}\left(\left(\left.\mathcal{R}_{\mathfrak{P}_s}\right/{\mathfrak{P}_s^{j_s}}\left(\tilde{\kappa}^{-1}\omega^i \right)\right)_{G_{\mathbb{R}}} \right)_{\mathfrak{P}_s}},
\end{equation}
where $\mathfrak{P}_s=\mathfrak{p}_s\mathcal{R}$. Since $\mathbb{Q}_{\infty}$ is totally real and $i$ is even, $\left.\tilde{\kappa}^{-1}\omega^i\right|_{G_{\mathbb{R}}}$ is a trivial character. Thus the equality \eqref{20200609} becomes to 
\begin{equation*}
\dfrac{\mathrm{char}_{\mathcal{R}}\left(\mathrm{Sel}_{\mathcal{A}\otimes\omega^i}\right)^{\lor}}{\mathrm{char}_{\mathcal{R}}\left(\mathrm{Sel}_{\mathcal{A}^{\mathrm{min}}\otimes\omega^i}\right)^{\lor}}=\mathfrak{P}_1^{j_1}\cdots\mathfrak{P}_r^{j_r}
\end{equation*}
and the proof is complete.

\end{proof}

\begin{rem}\label{irregular pair}
\begin{enumerate}
\item[(1)]Let $\left(p, k \right)$ be an irregular pair i.e. $p$ is an irregular prime which divides the numerator of the $k$-th Bernoulli number $B_k$. By an idea of Ribet \cite{Ri76}, there is an eigen cusp form $f_k \in S_k\left(\mathrm{SL}_2\left(\mathbb{Z} \right) \right)$ such that the Galois representation $\rho_{f_k}$ attached to $f_k$ is residually reducible. This Galois representation $\rho_{f_k}$ and its Hida deformation play an important role in the Iwasawa theory for ideal class groups which is studied by Ribet \cite{Ri76}, Mazur-Wiles \cite{MW84} and Wiles \cite{Wi90}. 

We recall that the following condition holds for all irregular pairs $\left(p, k \right)$ with $p<10^7$ and $k<8000$ except for the pair $\left(p, k \right)=\left(547, 486 \right)$:
\begin{enumerate}
\item[($\Lambda$)]The Eisenstein component $\mathbf{h}\left(1, \omega^{k-1}, \mathbb{Z}_p \right)^{\mathrm{ord}}_{\mathfrak{M}}$ of Hida's cuspidal Hecke algebra  is free of rank one over $\Lambda$.
\end{enumerate}
This implies that there exists a unique $\Lambda$-adic eigen cusp form $\mathcal{F}_{p, k}$ with character $\omega^{k-1}$ which can specialize to $f_k$ and the Eisenstein ideal $J\left(\mathcal{F}_{p, k} \right)$ is generated by the Kubota-Leopoldt $p$-adic $L$-function $L_p\left(\omega^{k}; \gamma^{\prime} \right)$ for those $\left(p, k \right)$. For all examples we known, $L_p\left(\omega^{k}; \gamma^{\prime} \right)$ is always a prime ideal of $\Lambda$. Thus Theorem \ref{Gal} and Corollary \ref{C2} tell us that for all examples which we know, there are exactly two stable $\Lambda$-free lattices up to homothety and $\mathscr{L}_p^{\mathrm{alg}}\left(\rho_{\mathcal{F}}^{\mathrm{n.ord}, \left(i \right)} \right)$ contains exactly two elements for even $i$.
\item[(2)]As is remarked in Remark \ref{hold not Hida},  one could not apply Theorem \ref{O} directly to study the change of $\mathrm{Sel}_{\mathbb{A}}$ for one-variable Hida deformation $\mathbb{T}$. However, we will use the control theorem which is studied in the next subsection to specialize $\mathrm{char}_{\mathcal{R}}\left(\mathrm{Sel}_{\mathcal{A}} \right)^{\lor}$. We prove that $\mathrm{Sel}_{\mathbb{A}}$ does not change in Corollary \ref{alglonevar}.
\end{enumerate}
\end{rem}

\subsection{Control theorem of Selmer groups}
Let us keep the notation of the previous subsection. We assume that $\rho_{\mathcal{F}}$ satisfies the condition (Lat-fr) in \S 1. As is recalled in \S 4.2, a sufficient condition of (Lat-fr) is that the ring $\mathbb{I}$ is Gorenstein and $\left.\chi\right|_{\left(\left.\mathbb{Z}\right/p\mathbb{Z} \right)^{\times}} \neq \mathbf{1}, \omega$ due to Tilouine. From now on to the end of this paper, we will also assume that the Eisenstein ideal $\mathbb{J}$ is principle. This principality of $\mathbb{J}$ holds if $\mathbf{H}\left(N, \chi, \mathbb{Z}_p[\chi] \right)^{\mathrm{ord}}_{\mathfrak{M}}$ and $\mathbf{h}\left(N, \chi, \mathbb{Z}_p[\chi] \right)^{\mathrm{ord}}_{\mathfrak{M}}$ are Gorenstein due to Ohta \cite[Theorem 2]{MO05}. 

Let $\mathbb{T}$ be a stable free lattice and $\mathcal{T}:=\mathbb{T}\hat{\otimes}_{\mathbb{Z}_p}\mathbb{Z}_p[[\Gamma]]\left(\tilde{\kappa}^{-1} \right)$. Let $\mathcal{R}:=\mathbb{I}[[\Gamma]]$ and $\mathcal{A}=\mathcal{T}\otimes_{\mathcal{R}}\mathcal{R}^{\lor}$. Let $\Sigma$ be the union of the set of primes dividing $Np$ and $\infty$. Recall that $\mathbb{Q}_{\Sigma}$ is the maximal extension of $\mathbb{Q}$ unramified outside $\Sigma$. Let $P$ be a height-one prime ideal of $\mathcal{R}$ and we consider the following map 
\begin{equation*}
h_P: H^1\left(\left.\mathbb{Q}_{\Sigma}\right/{\mathbb{Q}}, \mathcal{A}\otimes\omega^i[P] \right) \rightarrow H^1\left(\left.\mathbb{Q}_{\Sigma}\right/{\mathbb{Q}}, \mathcal{A}\otimes\omega^i\right)[P].
\end{equation*}

We call a homomorphism $\phi \in \mathrm{Hom}_{\mathbb{Z}_p}\left(\mathbb{I}, \overline{\mathbb{Q}_p} \right)$ an \textit{arithmetic specialization of weight $k_{\phi}$} if there exists an open subgroup $U$ of $\Gamma^{\prime}$ such that $\left.\phi\right|_U$ coincide with the character ${\kappa^{\prime}}^{k_{\phi}-2}$ for some $k_{\phi} \in \mathbb{Z}_{\geq 2}$. We obtain the following result.

\begin{thm}\label{T3}
Suppose that $\mathbb{I}$ is a UFD and Gorenstein. Assume the conditions $p \nmid \varphi\left(N \right)$, and that the Eisenstein ideal $\mathbb{J}$ is principal. We decompose $\mathbb{J}$ into product of height-one prime ideals as $\mathbb{J}=\mathfrak{p}_1^{n_1}\cdots\mathfrak{p}_r^{n_r}$, where $\mathfrak{p}_i$ and $\mathfrak{p}_j$ are distinct when $i\neq j$. Assume further that $\left.\chi\right|_{\left(\left.\mathbb{Z}\right/p\mathbb{Z} \right)^{\times}} \neq \mathbf{1}, \omega$ and that the tame conductor $N$ is square free. Let us take a stable free lattice $\mathbb{T}$ and let $\left(j_1, \cdots, j_r \right)$ be the vertex of $\mathrm{Rect}_{\left(0, \cdots, 0\right)}^{\left(n_1, \cdots, n_r \right)}$ which is the image of $[\mathbb{T}]$ under $\Phi_{\mathbf{1}}$ in Theorem \ref{Gal}. Let $P$ be a height-one prime ideal of $\mathcal{R}$ in the following cases
\begin{enumerate}
\item[(a)]$P$ is $P=\mathrm{Ker}\left(\phi \right)\mathcal{R}$ for some arithmetic specialization $\phi$.
\item[(b)]$P$ is $P=\left(\gamma-1 \right) \subset \mathcal{R}$.
\item[(c)]$P$ is $P=\mathfrak{p}_t$ which is a factor of $\mathbb{J}$ in $\mathcal{R}$ for some $1\leq t \leq r$.
\end{enumerate}
Let $i$ be an even integer, then we have the following statements
\begin{enumerate}
\item[(1)]In case (a), the map $\mathrm{res}_P$ is an isomorphism.
\item[(2)]In case (b), the map $\mathrm{res}_P$ is an isomorphism except when $\omega^i=\mathbf{1}$ for which we have the following statements:

When $\omega^i=\mathbf{1}$, the map $\mathrm{res}_P$ is injective and $\mathrm{Coker}\,\mathrm{res}_P$ is isomorphic to $\mathbb{I}^{\lor}\left[ \dfrac{\left(a\left(p, \mathcal{F} \right)-1 \right)}{\mathfrak{p}_1^{n_1-j_1}\cdots\mathfrak{p}_r^{n_r-j_r}} \right]$.
\item[(3)]In case (c), the map $\mathrm{res}_P$ is an isomorphism except when $\omega^i=\mathbf{1}$ and $j_t<n_t$ for which we have the following statements:

When $\omega^i=\mathbf{1}$ and $j_t<n_t$, the map $\mathrm{res}_P$ is surjective and $\mathrm{Ker}\,\mathrm{res}_P$ is isomorphic to $\left(\left.\mathcal{R}\right/\mathfrak{p}_t\right)^{\lor}\left[\gamma-1\right]$.
\end{enumerate}
\end{thm}

We prove Theorem \ref{T3} in this section. For a prime $v \mid N$, we define $H_{\mathrm{ur}}^1\left(\mathbb{Q}_v, \mathcal{A}\otimes\omega^i \right)$ as follows:
\begin{equation*}
H_{\mathrm{ur}}^1\left(\mathbb{Q}_v, \mathcal{A}\otimes\omega^i \right):=\mathrm{Ker}\left[H^1\left(\mathbb{Q}_v, \mathcal{A}\otimes\omega^i \right) \rightarrow H^1\left(I_v, \mathcal{A}\otimes\omega^i \right) \right].
\end{equation*}
We define $H_{\mathrm{Gr}}^1\left(\mathbb{Q}_p, \mathcal{A}\otimes\omega^i \right)$ as follows:
\begin{equation*}
H_{\mathrm{Gr}}^1\left(\mathbb{Q}_p, \mathcal{A}\otimes\omega^i \right):=\mathrm{Ker}\left[H^1\left(\mathbb{Q}_p, \mathcal{A}\otimes\omega^i \right) \rightarrow H^{1}\left(I_p, \left.\mathcal{A}\otimes\omega^i\right/F^{+}\mathcal{A}\otimes\omega^i \right) \right].
\end{equation*}
Let us consider the following commutative diagram:

\begin{equation}\label{190208}
\xymatrix{
    0 \ar[r]   & \mathrm{Sel}_{\mathcal{A}\otimes\omega^i[P]} \ar[r] \ar[d]^{\mathrm{res}_P} & H^1\left(\mathbb{Q}_{\Sigma}/\mathbb{Q}, \mathcal{A}\otimes\omega^i[P] \right) \ar[r] \ar[d]^{h_P} & Y_P \ar[d]^{l_P} \\
    0 \ar[r] & \mathrm{Sel}_{\mathcal{A}\otimes\omega^i}[P] \ar[r] & H^1\left(\mathbb{Q}_{\Sigma}/\mathbb{Q}, \mathcal{A}\otimes\omega^i \right)[P] \ar[r] & Y[P], \\
   }
\end{equation}
where $$Y_P:=\dfrac{H^1\left(\mathbb{Q}_p, \mathcal{A}\otimes\omega^i[P] \right)}{H_{\mathrm{Gr}}^1\left(\mathbb{Q}_p, \mathcal{A}\otimes\omega^i[P] \right)} \oplus \displaystyle\bigoplus_{v\mid N}\dfrac{H^1\left(\mathbb{Q}_v, \mathcal{A}\otimes\omega^i[P] \right)}{H_{\mathrm{ur}}^1\left(\mathbb{Q}_v, \mathcal{A}\otimes\omega^i[P] \right)},$$ $$Y:=\dfrac{H^1\left(\mathbb{Q}_p, \mathcal{A}\otimes\omega^i \right)}{H_{\mathrm{Gr}}^1\left(\mathbb{Q}_p, \mathcal{A}\otimes\omega^i \right)} \oplus \displaystyle\bigoplus_{v\mid N}\dfrac{H^1\left(\mathbb{Q}_v, \mathcal{A}\otimes\omega^i \right)}{H_{\mathrm{ur}}^1\left(\mathbb{Q}_v, \mathcal{A}\otimes\omega^i \right)}.$$ Note that the map $h_P$ is surjective. Thus to prove Theorem \ref{T3}, it is important to study the groups $\mathrm{Ker}\,l_P$ and $\mathrm{Ker}\,h_P$. Following an argument of Ochiai \cite[Proposition 5.1]{Ochiai06}, we understand the group $\mathrm{Ker}\,l_P$ as follows.

\begin{prp}[Ochiai {\cite[Proposition 5.1, Proposition 5.2]{Ochiai06}}]\label{Ochiai06s5}
Let us keep the assumptions of Theorem \ref{T3}. Then we have
\begin{equation*}
\mathrm{Ker}\,l_P \cong
\begin{cases}
0 & \text{in case (a)}\\
0 & \text{in case (b) when}\ \omega^i \neq \mathbf{1}\\
\mathbb{I}^{\lor}[a\left(p, \mathcal{F} \right)-1] & \text{in case (b) when}\ \omega^i=\mathbf{1}\\
0 & \text{in case (c)}.
\end{cases}
\end{equation*}

\end{prp}
Since our normalization is different from Ochiai's paper which leads the statement slightly different and case (c) does not exist in residually irreducible case, we recall the proof briefly.
\begin{proof}
First we study $\mathrm{Ker}\,l_{P, v}$ for the restriction $l_{P, v}$ of $l_P$ to the $v$-part for the prime $v \mid N$. We have that $\mathrm{Ker}\,l_{P, v}$ is the Pontryagin dual of the following module
\begin{equation*}
\left(\left(\left(\mathbb{T}^{*} \right)_{I_v}\hat{\otimes}\mathbb{Z}_p[[\Gamma]]\left(\tilde{\kappa} \right)\otimes\omega^{-i} \right)[P] \right)_{G_{\mathbb{Q}_v}}
\end{equation*}
by the proof of \cite[Proposition 5.2]{Ochiai06}. Thus the module $\left(\left(\mathbb{T}^{*} \right)_{I_v}\hat{\otimes}\mathbb{Z}_p[[\Gamma]]\left(\tilde{\kappa} \right)\otimes\omega^{-i} \right)[P]$ vanishes when $P$ is in case (b). Under the assumption that $\chi$ is primitive and $N$ is square-free, we have that the local automorphic representation $\pi_{f, v}$ at $v$ is principal series for every cusp form $f$ which is the image of $\mathcal{F}$ under an arithmetic specialization. Thus the image of $I_v$ under $G_{\mathbb{Q}} \rightarrow \mathrm{Aut}_{\mathbb{I}}\left(\mathbb{T} \right)$ is a finite group by the proof of \cite[Lemma 2.14]{FO12}. Then there exist a ring of integers $\mathcal{O} \subset \mathbb{I}$ of a finite extension of $\mathbb{Q}_p$ and integers $0 \leq r, r^{\prime} \leq \infty$ such that $\left(\mathbb{T}^{*}  \right)_{I_v}\cong \left.\mathbb{I}\right/\left(\varpi\right)^{r} \oplus \left.\mathbb{I}\right/\left(\varpi\right)^{r^{\prime}}$ by \cite[Theorem 3.3-(1)]{Ochiai06}, where $\varpi$ is a fixed uniformizer of $\mathcal{O}$. Thus the module $\left(\left(\mathbb{T}^{*} \right)_{I_v}\hat{\otimes}\mathbb{Z}_p[[\Gamma]]\left(\tilde{\kappa} \right)\otimes\omega^{-i} \right)[P]$ vanishes when $P$ is in case (a). Now we consider case (c). We have that $\mathbb{J}$ is a factor of Kubota-Leopoldt $p$-adic $L$-function $L_p\left(\chi\omega; \gamma^{\prime} \right)\in\Lambda_{\chi}$ (see \cite[Proposition 3.8]{Y} for example). Thus $P$ is not a uniformizer of any finite extension $\mathcal{O}$ of $\mathbb{Z}_p$ by Ferrero-Washington's theorem \cite{FW77}, hence $\left(\left(\mathbb{T}^{*} \right)_{I_v}\hat{\otimes}\mathbb{Z}_p[[\Gamma]]\left(\tilde{\kappa} \right)\otimes\omega^{-i} \right)[P]$ vanishes.

Now we study $\mathrm{Ker}\,l_{P, p}$ for the restriction $l_{P, p}$ of $l_P$ to the $p$-part. We have $\mathrm{Ker}\,l_{P, p}$ is isomorphic to $\left(\left.H^0\left(I_p, F^{-}\mathcal{A}\otimes\omega^i \right)\right/{P H^0\left(I_p, F^{-}\mathcal{A}\otimes\omega^i \right)} \right)^{G_{\mathbb{Q}_p}}$ and 
\begin{equation}\label{2008121}
H^0\left(I_p, F^{-}\mathcal{A}\otimes\omega^i \right)\cong
\begin{cases}
F^{-}\mathcal{A}[\gamma-1] & \text{if}\ \omega^i=\mathbf{1} \\
0& \text{otherwise}.
\end{cases}
\end{equation}
This is different from \cite[page 1178]{Ochiai06} since our normalization is the dual of Ochiai's paper. Thus Proposition \ref{Ochiai06s5} follows. 
\end{proof}

Now we study the group $\mathrm{Ker}\,h_P$ in the following proposition. 

\begin{prp}\label{H0}
Let us keep the  assumptions and the notation of Theorem \ref{T3}. Let us take a stable free lattice $\mathbb{T}$ and let $\left(j_1, \cdots, j_r \right)$ be the vertex of $\mathrm{Rect}_{\left(0, \cdots, 0\right)}^{\left(n_1, \cdots, n_r \right)}$ which is the image of $[\mathbb{T}]$ under $\Phi_{\mathbf{1}}$ in Theorem \ref{Gal}. Let $i$ be an even integer, then we have 
\begin{equation*}
H^{0}\left(\mathbb{Q}, \mathcal{A}\otimes\omega^i \right) \cong 
\begin{cases}
\mathcal{R}^{\lor}[\mathfrak{p}^{n_1-j_1}\cdots\mathfrak{p}_r^{n_r-j_r}\mathcal{R}, \gamma-1] & \text{when}\ \omega^i=\mathbf{1}\\
0 & \text{otherwise}.
\end{cases}
\end{equation*}
\end{prp}
\begin{proof}
When $\omega^i \neq \mathbf{1}$, the assertion follows immediately since $\mathcal{A}\otimes\omega^i$ has no composition factor on which $G_{\mathbb{Q}}$ acts trivially. Hence it remains to study the case when $\omega^i=\mathbf{1}$. We see that $H^0\left(\mathbb{Q}, \mathcal{A} \right)$ is the Pontryagin dual of $\left(\mathcal{T}^{*} \right)_{G_{\mathbb{Q}}}$. We divide our argument on studying the $\mathcal{R}$-module $\left(\mathcal{T}^{*} \right)_{G_{\mathbb{Q}}}$ into the following steps. 
\begin{s}\label{s1}
Let us consider the vertex $\left(n_1, \cdots, n_r \right)$ in  $\mathrm{Rect}_{\left(0, \cdots, 0\right)}^{\left(n_1, \cdots, n_r \right)}$. Since $\left(j_1, \cdots, j_r \right) \leq \left(n_1, \cdots, n_r \right)$, by Theorem \ref{Gal}, one could take a free lattice $\mathbb{T}^{\mathrm{max}}$ which is a representative of $\Phi_1^{-1}\left(  \left(n_1, \cdots, n_r \right)        \right)$ such that $\mathcal{T} \subset \mathcal{T}^{\mathrm{max}}$ and $\left.\mathcal{T}^{\mathrm{max}}\right/\mathcal{T} \stackrel{\sim}{\rightarrow} \bigslant{\mathcal{R}}{\mathfrak{p}_1^{n_1-j_1}\cdots\mathfrak{p}_r^{n_r-j_r}\mathcal{R}}\,\left(\tilde{\kappa}^{-1} \right)$, where $\mathcal{T}:=\mathbb{T}^{\mathrm{max}}\hat{\otimes}_{\mathbb{Z}_p}\mathbb{Z}_p[[\Gamma]]\left(\tilde{\kappa}^{-1} \right)$. Let $\mathcal{K}:=\mathrm{Frac}\left(\mathcal{R} \right)$ and $\mathcal{V}:=\mathcal{T}\otimes_{\mathcal{R}}\mathcal{K}$. Recall that $\mathcal{T}^{*}$ consists of the linear map $f : \mathcal{V} \rightarrow \mathcal{K}$ such that $f\left(\mathcal{T} \right) \subset \mathcal{R}$. Thus $\left.\mathcal{T}^{*}\right/{\mathcal{T}^{\mathrm{max}}}^{*} \stackrel{\sim}{\rightarrow} \bigslant{\mathcal{R}}{\mathfrak{p}_1^{n_1-j_1}\cdots\mathfrak{p}_r^{n_r-j_r}\mathcal{R}}\,\left(\tilde{\kappa}\right)$ and we have the following exact sequence 
\small
\begin{equation}\label{20200521}
0 \rightarrow \bigslant{\mathcal{R}}{\mathfrak{p}_1^{n_1-j_1}\cdots\mathfrak{p}_r^{n_r-j_r}\mathcal{R}}\,\left(\tilde{\kappa}\eta^{-1} \right) \rightarrow \bigslant{\mathcal{T}^*}{\mathfrak{p}_1^{n_1-j_1}\cdots\mathfrak{p}_r^{n_r-j_r}\mathcal{T}^*} \rightarrow \bigslant{\mathcal{R}}{\mathfrak{p}_1^{n_1-j_1}\cdots\mathfrak{p}_r^{n_r-j_r}\mathcal{R}}\,\left(\tilde{\kappa}\right) \rightarrow 0,
\end{equation}
\normalsize
where we denote by $\eta$ the determinant $\mathrm{det}\,\rho_{\mathcal{F}}$ to the end of the proof. This implies the surjection $$\mathcal{T}^{*} \twoheadrightarrow \left(\bigslant{\mathcal{R}}{\mathfrak{p}_1^{n_1-j_1}\cdots\mathfrak{p}_r^{n_r-j_r}\mathcal{R}}\,\left(\tilde{\kappa}\right) \right)_{G_{\mathbb{Q}}} \stackrel{\sim}{\rightarrow} \bigslant{\mathcal{R}}{\left(\gamma-1, \mathfrak{p}_1^{n_1-j_1}\cdots\mathfrak{p}_r^{n_r-j_r}\mathcal{R}\right)},$$ which factors through $\left(\mathcal{T}^{*} \right)_{G_{\mathbb{Q}}}$. Hence
\begin{equation}\label{202005254}
\left(\mathcal{T}^{*} \right)_{G_{\mathbb{Q}}} \twoheadrightarrow \bigslant{\mathcal{R}}{\left(\gamma-1, \mathfrak{p}_1^{n_1-j_1}\cdots\mathfrak{p}_r^{n_r-j_r}\mathcal{R}\right)}.
\end{equation}
Thus $\left(\mathcal{T}^{*} \right)_{G_{\mathbb{Q}}}$ is nontrivial if $\left(j_1, \cdots, j_r \right) < \left(n_1, \cdots, n_r \right)$.
\end{s}

\begin{s}\label{s2}
Now we assume $\left(j_1, \cdots, j_r \right) < \left(n_1, \cdots, n_r \right)$ from now on to the end of Step \ref{s3}. Let $\mathfrak{m}_{\mathcal{R}}$ be the maximal ideal of $\mathcal{R}$. We have the surjection 
\begin{equation}\label{2001271}
\left.\mathcal{T}^{*}\right/{\mathfrak{m}_{\mathcal{R}}\mathcal{T}^{*}} \twoheadrightarrow \left.\left(\mathcal{T}^{*} \right)_{G_{\mathbb{Q}}}\right/{\mathfrak{m}_{\mathcal{R}}\left(\mathcal{T}^{*} \right)_{G_{\mathbb{Q}}}}.
\end{equation}
If the $\left.\mathcal{R}\right/{\mathfrak{m}_{\mathcal{R}}}$-vector space $\left.\left(\mathcal{T}^{*} \right)_{G_{\mathbb{Q}}}\right/{\mathfrak{m}_{\mathcal{R}}\left(\mathcal{T}^{*} \right)_{G_{\mathbb{Q}}}}$ has dimension two, the surjection \eqref{2001271} becomes an isomorphism, which contradicts to that the determinant $\eta$ is odd. On the other hand, the module $\left(\mathcal{T}^{*} \right)_{G_{\mathbb{Q}}}$ is nontrivial which is proved in step \ref{s1}. Thus $\left.\left(\mathcal{T}^{*} \right)_{G_{\mathbb{Q}}}\right/{\mathfrak{m}_{\mathcal{R}}\left(\mathcal{T}^{*} \right)_{G_{\mathbb{Q}}}}$ must have dimension one by Nakayama's lemma. Hence $\left(\mathcal{T}^{*} \right)_{G_{\mathbb{Q}}}$ is a cyclic $\mathcal{R}$-module and we have 
\begin{equation}\label{2008251}
\left(\mathcal{T}^{*} \right)_{G_{\mathbb{Q}}} \stackrel{\sim}{\rightarrow} \left.\mathcal{R}\right/{\mathrm{Ann}_{\mathcal{R}}(  \left(\mathcal{T}^{*} \right)_{G_{\mathbb{Q}}} )}
\end{equation}
for the annihilator ideal $\mathrm{Ann}_{\mathcal{R}}(  \left(\mathcal{T}^{*} \right)_{G_{\mathbb{Q}}} )$. Thus by \eqref{202005254}, we have
\begin{equation}\label{2008252}
\mathrm{Ann}_{\mathcal{R}}(  \left(\mathcal{T}^{*} \right)_{G_{\mathbb{Q}}} ) \subset \left(\gamma-1, \mathfrak{p}_1^{n_1-j_1}\cdots\mathfrak{p}_r^{n_r-j_r}\mathcal{R}\right).
\end{equation}
We will prove the converse of the inclusion in the following step. 
\end{s}

\begin{s}\label{s3} 
Let us keep the assumption that $\left(j_1, \cdots, j_r \right) < \left(n_1, \cdots, n_r \right)$ and let $\mathcal{M}:=\mathrm{Ker}\left(\mathcal{T}^{*} \twoheadrightarrow \left(\mathcal{T}^{*} \right)_{G_{\mathbb{Q}}} \right).$ Since $\left(\mathcal{T}^{*} \right)_{G_{\mathbb{Q}}}$ is a torsion $\mathcal{R}$-module (this is because that $\rho_{\mathcal{F}}^{\mathrm{n. ord}}$ is irreducible), $\mathcal{M}$ is a stable lattice of $\mathcal{V}$. We define $\mathcal{M}^{+}:=\dfrac{1+c}{2}\mathcal{M}$ and $\mathcal{M}^{-}:=\dfrac{1-c}{2}\mathcal{M}$ for the complex conjugation $c \in G_{\mathbb{R}}$, then one can decompose $\mathcal{M}=\mathcal{M}^+\oplus\mathcal{M}^-$ as an $\mathcal{R}[G_{\mathbb{R}}]$-module. We define the $G_{\mathbb{R}}$-modules $\mathcal{T}^{*, +}$ and $\mathcal{T}^{*, -}$ in the same manner. Then we have the following exact sequence of $G_{\mathbb{R}}$-modules
\begin{equation*}
0 \rightarrow \left.\mathcal{T}^{*, +}\right/{\mathcal{M}^{+}}\,\left(\mathbf{1} \right) \rightarrow \left.\mathcal{T}^{*}\right/{\mathcal{M}}\,\left(\mathbf{1} \right) \rightarrow \left.{\mathcal{T}^{*}}^{-}\right/{\mathcal{M}^{-}}\,\left(-\mathbf{1} \right) \rightarrow 0
\end{equation*}
by the same arguments as the proof of Proposition \ref{reffr}. Then 
\begin{equation}\label{20201208}
\mathcal{M}^{-}={\mathcal{T}^{*}}^{-}
\end{equation}
and $\left.\mathcal{T}^{*, +}\right/{\mathcal{M}^{+}}$ is isomorphic to $\left.\mathcal{T}^{*}\right/{\mathcal{M}}$ as $G_{\mathbb{R}}$-modules. Then by \eqref{2008251}, the quotient $\left.\mathcal{T}^{*, +}\right/{\mathcal{M}^{+}}$ is isomorphic to $\left.\mathcal{R}\right/{\mathrm{Ann}_{\mathcal{R}}(  \left(\mathcal{T}^{*} \right)_{G_{\mathbb{Q}}} )}$. Moreover, since the $\mathcal{R}$-module $\mathcal{T}^{*, +}$ is free of rank one by Lemma \ref{frrkone}, we have the following equality
\begin{equation}\label{200921}
\mathcal{M}^{+}=\mathrm{Ann}_{\mathcal{R}}(  \left(\mathcal{T}^{*} \right)_{G_{\mathbb{Q}}} )\mathcal{T}^{* +}.
\end{equation}

Now let us consider the vertex $\left(0, \cdots, 0 \right)$ in  $\mathrm{Rect}_{\left(0, \cdots, 0\right)}^{\left(n_1, \cdots, n_r \right)}$. By Theorem \ref{Gal}, one can take a free lattice $\mathbb{T}^{\mathrm{min}}$ which is a representative of $\Phi_1^{-1}\left(  \left(0, \cdots, 0 \right)        \right)$ such that $\mathcal{T}^{\mathrm{min}} \subset \mathcal{T}$ and $\left.\mathcal{T}\right/\mathcal{T}^{\mathrm{min}} \stackrel{\sim}{\rightarrow} \bigslant{\mathcal{R}}{\mathfrak{p}_1^{j_1}\cdots\mathfrak{p}_r^{j_r}\mathcal{R}}\,\left(\tilde{\kappa}^{-1} \right)$, where $\mathcal{T}^{\mathrm{min}}=\mathbb{T}^{\mathrm{min}}\hat{\otimes}_{\mathbb{Z}_p}\mathbb{Z}_p[[\Gamma]]\left(\tilde{\kappa}^{-1} \right)$. Then the $G_{\mathbb{Q}}$-module $\left.{\mathcal{T}^{\mathrm{min}}}^{*}\right/{\mathcal{T}^{*}}$ is isomorphic to $\bigslant{\mathcal{R}}{\mathfrak{p}_1^{j_1}\cdots\mathfrak{p}_r^{j_r}\mathcal{R}}\,\left(\tilde{\kappa} \right)$. Since $\tilde{\kappa}$ is even, we have
\begin{equation}\label{202009211}
\begin{cases}
{\mathcal{T}^{*}}^{-}={{\mathcal{T}^{\mathrm{min}}}^{*}}^{-}\\
\mathcal{T}^{* +}=\mathfrak{p}_1^{j_1}\cdots\mathfrak{p}_r^{j_r}\mathcal{R}{\mathcal{T}^{\mathrm{min}}}^{* +}.
\end{cases}
\end{equation}
Thus we have the following equality by combining \eqref{20201208}, \eqref{200921} and \eqref{202009211}:
\begin{equation}\label{2020092112}
\begin{cases}
\mathcal{M}^{-}={{\mathcal{T}^{\mathrm{min}}}^{*}}^{-}\\
\mathcal{M}^{+}=\mathfrak{p}_1^{j_1}\cdots\mathfrak{p}_r^{j_r}\mathcal{R}\cdot\mathrm{Ann}_{\mathcal{R}}(  \left(\mathcal{T}^{*} \right)_{G_{\mathbb{Q}}} ){\mathcal{T}^{\mathrm{min}}}^{* +}.
\end{cases}
\end{equation}
This implies that $\mathrm{tr}\,{\mathcal{\rho}_{\mathcal{F}}^{\mathrm{n. ord}}}^{*}$ modulo $\mathfrak{p}_1^{j_1}\cdots\mathfrak{p}_r^{j_r}\mathcal{R}\cdot\mathrm{Ann}_{\mathcal{R}}(  \left(\mathcal{T}^{*} \right)_{G_{\mathbb{Q}}} )$ is sum of two characters. Note that the ideal of reducibility $J\left({\mathcal{\rho}_{\mathcal{F}}^{\mathrm{n. ord}}}^{*}\right)$ is generated by $\mathbb{J}$, then we have $\mathbb{J}\mathcal{R} \subset \mathfrak{p}_1^{j_1}\cdots\mathfrak{p}_r^{j_r}\mathcal{R}\cdot\mathrm{Ann}_{\mathcal{R}}(  \left(\mathcal{T}^{*} \right)_{G_{\mathbb{Q}}} )$. Thus under the assumption that $\mathbb{J}$ is principal, we have 
\begin{equation}\label{202005211}
\mathfrak{p}_1^{n_1-j_1}\cdots\mathfrak{p}_r^{n_r-j_r}\mathcal{R} \subset \mathrm{Ann}_{\mathcal{R}}(  \left(\mathcal{T}^{*} \right)_{G_{\mathbb{Q}}} ). 
\end{equation}

By \eqref{202005211} we have the following commutative diagram
\small
\begin{equation}\label{01250125}
 \xymatrix{
   0 \ar[r]   & \bigslant{\mathcal{R}}{\mathfrak{p}_1^{n_1-j_1}\cdots\mathfrak{p}_r^{n_r-j_r}\mathcal{R}}\,\left(\tilde{\kappa}\eta^{-1} \right) \ar[r] \ar[d] & \bigslant{\mathcal{T}^*}{\mathfrak{p}_1^{n_1-j_1}\cdots\mathfrak{p}_r^{n_r-j_r}\mathcal{T}^*} \ar[r] \ar[d] & \bigslant{\mathcal{R}}{\mathfrak{p}_1^{n_1-j_1}\cdots\mathfrak{p}_r^{n_r-j_r}\mathcal{R}}\,\left(\tilde{\kappa} \right) \ar[r] \ar[d] & 0 \\
    0 \ar[r] & \bigslant{\mathcal{R}}{\mathrm{Ann}_{\mathcal{R}}(  \left(\mathcal{T}^{*} \right)_{G_{\mathbb{Q}}} )}\,\left(\tilde{\kappa}\eta^{-1} \right) \ar[r] & \bigslant{\mathcal{T}^*}{\mathrm{Ann}_{\mathcal{R}}(  \left(\mathcal{T}^{*} \right)_{G_{\mathbb{Q}}} )\mathcal{T}^*} \ar[r] & \bigslant{\mathcal{R}}{\mathrm{Ann}_{\mathcal{R}}(  \left(\mathcal{T}^{*} \right)_{G_{\mathbb{Q}}} )}\,\left(\tilde{\kappa} \right) \ar[r] & 0,
  }
\end{equation}
\normalsize
in which all vertical maps are surjective. By the equality \eqref{200921} we have that $\mathrm{tr}\,{\mathcal{\rho}_{\mathcal{F}}^{\mathrm{n. ord}}}^{*}$ modulo the ideal $\mathrm{Ann}_{\mathcal{R}}(  \left(\mathcal{T}^{*} \right)_{G_{\mathbb{Q}}} )$ decomposes into two characters one of which is trivial. Thus,  since $\tilde{\kappa}$ is even and $\tilde{\kappa}\eta$ is odd, we must have that $\tilde{\kappa}$ modulo $\mathrm{Ann}_{\mathcal{R}}(  \left(\mathcal{T}^{*} \right)_{G_{\mathbb{Q}}} )$ is trivial by the uniqueness statement of the characters in Lemma \ref{21}. Hence we have the surjection $\bigslant{\mathcal{R}}{\mathfrak{p}_1^{n_1-j_1}\cdots\mathfrak{p}_r^{n_r-j_r}\mathcal{R}}\,\left(\tilde{\kappa}\right) \twoheadrightarrow \left.\mathcal{R}\right/{\mathrm{Ann}_{\mathcal{R}}(  \left(\mathcal{T}^{*} \right)_{G_{\mathbb{Q}}} )}\,\left(\mathbf{1} \right)$ which factors through $\left(\bigslant{\mathcal{R}}{\mathfrak{p}_1^{n_1-j_1}\cdots\mathfrak{p}_r^{n_r-j_r}\mathcal{R}}\,\left(\tilde{\kappa}\right) \right)_{G_{\mathbb{Q}}}$. Thus we have 
\begin{equation}\label{2008253}
\left(\gamma-1, \mathfrak{p}_1^{n_1-j_1}\cdots\mathfrak{p}_r^{n_r-j_r}\mathcal{R}\right) \subset \mathrm{Ann}_{\mathcal{R}}(  \left(\mathcal{T}^{*} \right)_{G_{\mathbb{Q}}} ).
\end{equation}
\end{s}

\begin{s}\label{s5}
We complete the proof in this step. For the lattice $\mathbb{T}$ corresponding to the vertex $\left(j_1, \cdots, j_r \right)<\left(n_1, \cdots, n_r \right)$, we have
\begin{equation}\label{20200524}
\left(\mathcal{T}^{*}\right)_{G_{\mathbb{Q}}} \stackrel{\sim}{\rightarrow} \left.\mathcal{R}\right/{\mathrm{Ann}_{\mathcal{R}}(  \left(\mathcal{T}^{*} \right)_{G_{\mathbb{Q}}} )} \stackrel{\sim}{\rightarrow} \bigslant{\mathcal{R}}{\left(\gamma-1, \mathfrak{p}_1^{n_1-j_1}\cdots\mathfrak{p}_r^{n_r-j_r}\mathcal{R}\right)}
\end{equation}
by combining \eqref{2008251}, \eqref{2008252} and \eqref{2008253}. Now we consider the lattice $\mathbb{T}^{\mathrm{max}}$ which is a representative of $\Phi_{\mathbf{1}}^{-1}\left(\left(n_1, \cdots, n_r \right)\right)$. Under the assumption that $\mathbb{J}$ is principal, there exists a stable free lattice $\mathbb{T}^{\mathrm{min}}$ such that $\left.\mathbb{T}^{\mathrm{max}}\right/{\mathbb{T}^{\mathrm{min}}} \stackrel{\sim}{\rightarrow} \left.\mathbb{I}\right/{\mathbb{J}}\left(\mathbf{1} \right)$. Let us take an arithmetic specialization $\phi$, then we have that $\left.\mathbb{T}^{\mathrm{max}}\otimes_{\mathbb{I}}\phi\left(\mathbb{I} \right)\right/\mathbb{T}^{\mathrm{min}}\otimes_{\mathbb{I}}\phi\left(\mathbb{I} \right) \stackrel{\sim}{\rightarrow} \left.\phi\left(\mathbb{I} \right)\right/{\phi\left(\mathbb{J} \right)}\left(\mathbf{1} \right)$. Since $\mathbb{J}$ coincide with the ideal of reducibility of $\rho_{\mathcal{F}}$, we have that $\mathrm{ord}_{\phi\left(\mathfrak{m}_{\mathbb{I}} \right)}\phi\left(\mathbb{J} \right)+1$ is the length of the segment $C\left(\rho_{\phi} \right)$ for the Galois representation $\rho_{\phi}$ attached to $f_{\phi}$ by Proposition \ref{22}. Thus the homothety class $[\mathbb{T}^{\mathrm{max}}\otimes_{\mathbb{I}}\phi\left(\mathbb{I} \right)]$ is the extremity of $C\left(\rho_{\phi} \right)$. This implies that $\left.\mathcal{T}^{\mathrm{max}}\right/{\mathfrak{m}_{\mathcal{R}}\mathcal{T}^{\mathrm{max}}}$, which is isomorphic to $\left.\mathbb{T}^{\mathrm{max}}\right/{\mathfrak{m}_{\mathbb{I}}\mathbb{T}^{\mathrm{max}}}$, has no sub-module on which $G_{\mathbb{Q}}$ acts trivially. Hence $\left.{\mathcal{T}^{\mathrm{max}}}^{*}\right/{\mathfrak{m}_{\mathcal{R}}{\mathcal{T}^{\mathrm{max}}}^{*}}$ has no quotient on which $G_{\mathbb{Q}}$ acts trivially. Using the argument on the dimension of $\left.\left({\mathcal{T}^{\mathrm{max}}}^{*} \right)_{G_{\mathbb{Q}}}\right/{\mathfrak{m}_{\mathcal{R}}\left({\mathcal{T}^{\mathrm{max}}}^{*} \right)_{G_{\mathbb{Q}}}}$ in Step \ref{s2}, we have that $\left({\mathcal{T}^{\mathrm{max}}}^{*} \right)_{G_{\mathbb{Q}}}$ must be trivial. Thus \eqref{20200524} also holds for $\left(j_1, \cdots, j_r \right)=\left(n_1, \cdots, n_r \right)$ and the proof is complete. 

\end{s}

\end{proof}

Now we return to the proof of Theorem \ref{T3}
\begin{proof}[Proof of Theorem \ref{T3}]
Since the proof are done in the same way, we only prove cases (b) and (c) when $\omega^i=\mathbf{1}$. First let us consider case (b) which needs more arguments. Since $\mathcal{R}$ is Gorenstein, under the assumption $\left.\chi\right|_{\left(\left.\mathbb{Z}\right/p\mathbb{Z} \right)^{\times}} \neq \omega$, the map $H^1\left(\mathbb{Q}_{\Sigma}/\mathbb{Q}, \mathcal{A}[P] \right) \rightarrow Y_P$ is surjective by the same proof of \cite[Theorem 4.10 and Corollary 4.12]{Ochiai06}. Then by the commutative diagram \eqref{190208} and the snake lemma, we have the following exact sequence
\begin{equation}\label{2008151}
0 \rightarrow \mathrm{Ker}\,\mathrm{res}_P \rightarrow \mathrm{Ker}\,h_P \rightarrow \mathrm{Ker}\,l_P \rightarrow \mathrm{Coker}\,\mathrm{res}_P \rightarrow 0.
\end{equation}

We prove that the map $\mathrm{Ker}\,h_P \rightarrow \mathrm{Ker}\,l_P$ is injective, which is equivalent to that $\left(\mathrm{Ker}\,l_P \right)^{\lor} \rightarrow \left(\mathrm{Ker}\,h_P \right)^{\lor}$ is surjective. We have that $\left(\mathrm{Ker}\,l_P \right)^{\lor}$ is isomorphic to $\left(\left(F^{-}\mathcal{T} \right)^{*}_{I_p}[P] \right)_{G_{\mathbb{Q}_p}}$ by the proof of Proposition \ref{Ochiai06s5}. The proof of Proposition \ref{Ochiai06s5} also tells us that $\left(F^{-}\mathcal{T} \right)^{*}_{I_p}$ is $\left.\left(F^{-}\mathcal{T} \right)^{*}\right/{\left(\gamma-1 \right)\left(F^{-}\mathcal{T} \right)^{*}}$. Thus we have the following isomorphism
\begin{equation}\label{miss1}
\left(\mathrm{Ker}\,l_P \right)^{\lor}\stackrel{\sim}{\rightarrow}\left(F^{-}\mathcal{T} \right)^{*}_{G_{\mathbb{Q}_p}}. 
\end{equation}
Let $\mathfrak{a}:=\left(\gamma-1, \mathfrak{p}_1^{n_1-j_1}\cdots\mathfrak{p}_r^{n_r-j_r}\right)$ and $\eta:=\mathrm{det}\,\rho_{\mathcal{F}}$. We have the following exact sequence of $G_{\mathbb{Q}}$-modules
\begin{equation}\label{012523}
0 \rightarrow \left.\mathcal{R}\right/{\mathfrak{a}}\,\left(\eta^{-1} \right) \rightarrow \left.\mathcal{T}^{*}\right/{\mathfrak{a}\mathcal{T}^{*}} \rightarrow \left(\mathcal{T}^{*}\right)_{G_{\mathbb{Q}}} \rightarrow 0
\end{equation}
by combining \eqref{01250125} and \eqref{20200524}. On the other hand, since $\overline{\rho}_{\mathcal{F}}\cong\mathbf{1}\oplus\chi$ as $G_{\mathbb{Q}}$-module and $\chi$ is primitive modulo $Np$, the action of $G_{\mathbb{Q}_p}$ on $F^{+}\mathcal{T}\otimes_{\mathcal{R}}\left.\mathcal{R}\right/{\mathfrak{m}_{\mathcal{R}}}$ is ramified. This implies that as $G_{\mathbb{Q}_p}$-modules, the semi-simplification of $\left.\mathcal{T}\right/\mathfrak{m}_{\mathcal{R}}\mathcal{T}$ is decomposed into two distinct characters. Thus we have that $F^{+}\mathcal{T}$ and $F^{-}\mathcal{T}$ as free $\mathcal{R}$-modules of rank one by \cite[Remark 2.13-(3)]{FO12}. Hence we have the following exact sequence of $G_{\mathbb{Q}_p}$-modules
\begin{equation}\label{0125233}
0 \rightarrow \left.\left(F^{-}\mathcal{T} \right)^{*}\right/{\mathfrak{a}\left(F^{-}\mathcal{T} \right)^{*}} \rightarrow \left.\mathcal{T}^{*}\right/{\mathfrak{a}\mathcal{T}^{*}} \rightarrow \left.\left(F^{+}\mathcal{T} \right)^{*}\right/{\mathfrak{a}\left(F^{+}\mathcal{T} \right)^{*}} \rightarrow 0.
\end{equation}
Since $F^{+}\mathcal{T}$ and $F^{-}\mathcal{T}$ are free $\mathcal{R}$-modules of rank one, we denote by $\varepsilon$ (resp. $\delta$) the character such that $G_{\mathbb{Q}_p}$ acts on $\left(F^{+}\mathcal{T} \right)^{*}$ (resp. $\left(F^{-}\mathcal{T} \right)^{*}$) by $\varepsilon$ (resp. $\delta$). We have that $\left(\mathcal{T}^{*}\right)_{G_{\mathbb{Q}}}$ is isomorphic to $\left.\mathcal{R}\right/\mathfrak{a}\,\left(\mathbf{1} \right)$ by \eqref{20200524}. Then combining with \eqref{012523} we have $\set{\varepsilon,  \delta}=\set{\mathbf{1}, \eta^{-1}}$ as characters of $G_{\mathbb{Q}_p}$ with values in $\left.\mathcal{R}\right/\mathfrak{a}$ by Lemma \ref{21}. Since $\gamma-1 \in \mathfrak{a}$, the action of $G_{\mathbb{Q}_p}$ on $\left.\left(F^{-}\mathcal{T} \right)^{*}\right/{\mathfrak{a}\left(F^{-}\mathcal{T} \right)^{*}}$ is unramified and hence $\delta=\mathbf{1}$ in $\left.\mathcal{R}\right/\mathfrak{a}$. Thus the exact sequence \eqref{012523} splits as $G_{\mathbb{Q}_p}$-modules and $\left.\left(F^{-}\mathcal{T} \right)^{*}\right/{\mathfrak{a}\left(F^{-}\mathcal{T} \right)^{*}}$ is isomorphic to $\left(\mathcal{T}^{*}\right)_{G_{\mathbb{Q}}}$. This implies that the following map 
\begin{equation}\label{012533}
\left(F^{-}\mathcal{T} \right)^{*} \hookrightarrow \mathcal{T}^{*} \twoheadrightarrow \left(\mathcal{T}^{*}\right)_{G_{\mathbb{Q}}}
\end{equation}
is surjective (we remark that we identify $\left(F^{-}\mathcal{T} \right)^{*}$ with the $\mathcal{R}$-lattice of $\mathrm{Hom}_{\mathcal{K}}\left(F^{-}\mathcal{T}\otimes_{\mathcal{R}}\mathcal{K}, \mathcal{K} \right)$, where $\mathcal{K}=\mathrm{Frac}\left(\mathcal{R} \right)$). Then \eqref{012533} induces the following surjection of $G_{\mathbb{Q}_p}$-modules
\begin{equation}\label{miss}
\left(F^{-}\mathcal{T} \right)^{*}_{G_{\mathbb{Q}_p}} \twoheadrightarrow \left(\mathcal{T}^{*}\right)_{G_{\mathbb{Q}}}. 
\end{equation}
Hence the injectivity of $\mathrm{Ker}\,h_P \rightarrow \mathrm{Ker}\,l_P$ follows by combining \eqref{miss1}, \eqref{012533} and \eqref{miss}.

Combining \eqref{2008151} and the injectivity of $\mathrm{Ker}\,h_P \rightarrow \mathrm{Ker}\,l_P$. We have that $\mathrm{Ker}\,\mathrm{res}_P$ is trivial and $\mathrm{Coker}\,\mathrm{res}_P$ is isomorphic to $\mathbb{I}^{\lor}\left[\dfrac{\left(a\left(p, \mathcal{F} \right)-1 \right)}{\mathfrak{p_1}^{n_1-j_1}\cdots\mathfrak{p}_r^{n_r-j_r}} \right]$. This completes the proof of case (b).

Now let us consider case (c). By Proposition \ref{Ochiai06s5}, the group $\mathrm{Ker}\,l_P$ is trivial in this case. Then $\mathrm{Ker}\,\mathrm{res}_P$ is isomorphic to $\mathrm{Ker}\,h_P$, which is isomorphic to $\left.H^0\left(\mathbb{Q}, \mathcal{A} \right)\right/\mathfrak{p}_t H^0\left(\mathbb{Q}, \mathcal{A} \right)$. When $j_t<n_t$, we have the following isomorphism by Proposition \ref{H0}:
\begin{equation*}
\left.H^0\left(\mathbb{Q}, \mathcal{A} \right)\right/\mathfrak{p}_t H^0\left(\mathbb{Q}, \mathcal{A} \right) \stackrel{\sim}{\rightarrow} \mathcal{R}^{\lor}[\gamma-1, \mathfrak{p}_t].
\end{equation*}
When $j_t=n_t$, Proposition \ref{H0} again implies that $H^0\left(\mathbb{Q}, \mathcal{A} \right)$ is $\mathfrak{p}_t$-divisible, hence the module $\left.H^0\left(\mathbb{Q}, \mathcal{A} \right)\right/\mathfrak{p}_t H^0\left(\mathbb{Q}, \mathcal{A} \right)$ is trivial. This completes the proof of Theorem \ref{T3}.

\end{proof}

\begin{rem}
When $i$ is odd, one could prove that $H^0\left(\mathbb{Q}, \mathcal{A}\otimes\omega^i \right)$ is non-trivial when $\omega^i=\chi^{-1}$. In that case, the map $\mathrm{res}_P$ is surjective and has non-trivial kernel for some $P$. We will study the control theorem for odd $i$ and its relation to the Iwasawa $\lambda$-invariants of cyclotomic fields in a forthcoming paper.
\end{rem}

\subsection{Greenberg's criterion for pseudo-null submodule}
Let us keep the notation of the previous section. To specialize the characteristic ideal $\mathrm{char}_{\mathcal{R}}\left(\mathrm{Sel}_{\mathcal{A}} \right)^{\lor}$ and study its relation to $\mathrm{char}_{\mathbb{I}}\left(\mathrm{Sel}_{\mathbb{A}} \right)^{\lor}$, the first important work is to study the control theorem of Selmer groups which is studied in the previous subsection. The second is to prove that $\left(\mathrm{Sel}_{\mathcal{A}} \right)^{\lor}$ has no nontrivial pseudo-null $\mathcal{R}$-submodule (see the proof of \cite[Lemma 3.4]{Ochiai05} for the importance of this result in specialization). In \cite[Proposition 4.1.1]{RG16}, Greenberg gives a criterion for Selmer groups of more general Galois deformation to be almost divisible, which is equivalent to that its Pontryagin dual has no nontrivial pseudo-null submodule (cf. \cite[Proposition 2.4]{RG06}). We verify Greenberg's criterion in the case of Hida deformation in the following proposition. We denote by $\mathfrak{m}_{\mathcal{R}}$ the maximal ideal of $\mathcal{R}$.

\begin{prp}\label{nontrivialpseudo-null}
Suppose that $\mathbb{I}$ is a UFD and Gorenstein. Assume $\left.\chi\right|_{\left(\left.\mathbb{Z}\right/p\mathbb{Z} \right)^{\times}} \neq \mathbf{1}, \omega$ and that $N$ is square-free. Let us take a stable free lattice $\mathbb{T}$. Let $i$ be an integer such that $\omega^i\neq\omega$ and $\left.\chi^{-1}\omega\right|_{\left(\mathbb{Z}/p\mathbb{Z} \right)^{\times}}$. Then $\left(\mathrm{Sel}_{\mathcal{A}\otimes\omega^i} \right)^{\lor}$ has no nontrivial pseudo-null $\mathcal{R}$-submodule.
\end{prp}
\begin{proof}
Let us fix an integer $i$ satisfying the assumption of this lemma. By abuse of notation, we write a new $\mathcal{T}$ (resp. $\mathcal{A}$) for old $\mathcal{T}\otimes\omega^i$ (resp. $\mathcal{T}\otimes\omega^i\otimes_{\mathcal{R}}\mathcal{R}^{\lor}$) from now on to the end of the proof. Due to Greenberg \cite[Proposition 4.1.1]{RG16}, it is sufficient to verify the following conditions.
\begin{enumerate}
\item[(i)]The $\mathcal{R}$-module $\mathcal{T}$ is reflexive.
\item[(ii)]The $\mathcal{R}$-module $\Sha^2\left(\Sigma, \mathcal{A} \right)$ is trivial, where $\Sha^i\left(\Sigma, \mathcal{A} \right):=\mathrm{Ker}\left[H^i\left(\left.\mathbb{Q}_{\Sigma}\right/{\mathbb{Q}}, \mathcal{A} \right) \rightarrow \displaystyle\bigoplus_{v\mid Np}H^1\left(\mathbb{Q}_v, \mathcal{A} \right) \right]$ is the Tate Shafarevich group. 
\item[(iii)]The $\mathcal{R}$-module $\left.\mathcal{T}^{*}(1)\right/{H^0\left(\mathbb{Q}_v, \mathcal{T}^{*}(1) \right)}$ is reflexive for any prime $v \mid Np$.
\item[(iv)]There exists a prime $v\mid Np$ such that $H^0\left(\mathbb{Q}_v, \mathcal{T}^{*}(1) \right)=0$.
\item[(v)]The global-to-local map $H^1\left(\left.\mathbb{Q}_{\Sigma}\right/{\mathbb{Q}}, \mathcal{A} \right) \rightarrow \dfrac{H^1\left(\mathbb{Q}_p, \mathcal{A} \right)}{H_{\mathrm{Gr}}^1\left(\mathbb{Q}_p, \mathcal{A} \right)} \oplus \displaystyle\bigoplus_{v\mid N}\dfrac{H^1\left(\mathbb{Q}_v, \mathcal{A} \right)}{H_{\mathrm{ur}}^1\left(\mathbb{Q}_v, \mathcal{A} \right)}$ is surjective. 
\item[(vi)]The local condition $H_{\mathrm{ur}}^{1}\left(\mathbb{Q}_v, \mathcal{A} \right)^{\lor}$ has no nontrivial pseudo-null $\mathcal{R}$-submodule for any $v \mid N$.
\item[(vii)]The local condition $H_{\mathrm{Gr}}^{1}\left(\mathbb{Q}_p, \mathcal{A} \right)^{\lor}$ has no nontrivial pseudo-null $\mathcal{R}$-submodule.
\item[(viii)]The $G_{\mathbb{Q}}$-module $\left.\mathcal{T}\right/{\mathfrak{m}_{\mathcal{R}}\mathcal{T}}$ has no composition factor which is isomorphic to $\omega$. 
\end{enumerate}

Since $\mathcal{T}$ is a free $\mathcal{R}$-module, the assertion (i) is obvious. The assertion (ii) follows since we have $H^2\left(\left.\mathbb{Q}_{\Sigma}\right/{\mathbb{Q}}, \mathcal{A} \right)$ is trivial by \cite[Lemma 8.3]{Ochiai06}. 

Let us verify the condition (iii) by discussion on the rank of ${\mathcal{T}^{*}\left(1\right)}^{G_{\mathbb{Q}_v}}$ i.e. the dimension of $\mathcal{K}$-vector space ${\mathcal{T}^{*}\left(1\right)}^{G_{\mathbb{Q}_v}}\otimes_{\mathcal{R}}\mathcal{K}$ for the field of fractions $\mathcal{K}$ of $\mathcal{R}$. If the ${\mathcal{T}^{*}\left(1\right)}^{G_{\mathbb{Q}_v}}$ has rank zero, we have ${\mathcal{T}^{*}\left(1\right)}^{G_{\mathbb{Q}_v}}$ is trivial since the $\mathcal{R}$-module ${\mathcal{T}^{*}\left(1\right)}^{G_{\mathbb{Q}_v}}$ is torsion-free,  hence (iii) follows. If ${\mathcal{T}^{*}\left(1\right)}^{G_{\mathbb{Q}_v}}$ has rank two, ${\mathcal{T}^{*}\left(1\right)}^{G_{\mathbb{Q}_v}}$ and $\mathcal{T}^{*}\left(1 \right)$ are $G_{\mathbb{Q}_v}$-isogeny. Since $\chi$ is primitive, this contradicts to that the action of $G_{\mathbb{Q}_v}$ on $\mathcal{T}^{*}\left(1 \right)$ is ramified. Now we assume that ${\mathcal{T}^{*}\left(1\right)}^{G_{\mathbb{Q}_v}}$ has rank one. Let us consider the following exact sequence
\begin{equation*}
{\mathcal{T}^{*}\left(1\right)}^{G_{\mathbb{Q}_v}}\otimes\left.\mathcal{R}\right/{\mathfrak{m}_{\mathcal{R}}}\rightarrow\mathcal{T}^{*}\left(1\right)\otimes\left.\mathcal{R}\right/{\mathfrak{m}_{\mathcal{R}}}\rightarrow \left.\mathcal{T}^{*}\left(1\right)\right/{{\mathcal{T}^{*}\left(1\right)}^{G_{\mathbb{Q}_v}}}\otimes\left.\mathcal{R}\right/{\mathfrak{m}_{\mathcal{R}}}\rightarrow 0.
\end{equation*}
We discuss the dimension of the $\left.\mathcal{R}\right/{\mathfrak{m}_{\mathcal{R}}}$-vector space $\left.\mathcal{T}^{*}\left(1\right)\right/{{\mathcal{T}^{*}\left(1\right)}^{G_{\mathbb{Q}_v}}}\otimes\left.\mathcal{R}\right/{\mathfrak{m}_{\mathcal{R}}}$. If the dimension is zero, $\mathcal{T}^{*}\left(1\right)\otimes\left.\mathcal{R}\right/{\mathfrak{m}_{\mathcal{R}}}$ is a quotient of ${\mathcal{T}^{*}\left(1\right)}^{G_{\mathbb{Q}_v}}\otimes\left.\mathcal{R}\right/{\mathfrak{m}_{\mathcal{R}}}$, hence $G_{\mathbb{Q}_v}$ acts trivially. Under the assumption that $\chi$ is primitive, this contradicts to that the action of $G_{\mathbb{Q}_v}$ on $\mathcal{T}^{*}\left(1\right)\otimes\left.\mathcal{R}\right/{\mathfrak{m}_{\mathcal{R}}}$ is ramified. If $\left.\mathcal{T}^{*}\left(1\right)\right/{{\mathcal{T}^{*}\left(1\right)}^{G_{\mathbb{Q}_v}}}\otimes\left.\mathcal{R}\right/{\mathfrak{m}_{\mathcal{R}}}$ has dimension two, the map $\mathcal{T}^{*}\left(1\right)\otimes\left.\mathcal{R}\right/{\mathfrak{m}_{\mathcal{R}}}\rightarrow \left.\mathcal{T}^{*}\left(1\right)\right/{{\mathcal{T}^{*}\left(1\right)}^{G_{\mathbb{Q}_v}}}\otimes\left.\mathcal{R}\right/{\mathfrak{m}_{\mathcal{R}}}$ is an isomorphism. Hence the image of the map ${\mathcal{T}^{*}\left(1\right)}^{G_{\mathbb{Q}_v}}\otimes\left.\mathcal{R}\right/{\mathfrak{m}_{\mathcal{R}}}\rightarrow\mathcal{T}^{*}\left(1\right)\otimes\left.\mathcal{R}\right/{\mathfrak{m}_{\mathcal{R}}}$ is trivial.  This implies ${\mathcal{T}^{*}\left(1\right)}^{G_{\mathbb{Q}_v}} \subset \mathcal{\mathfrak{m}_{\mathcal{R}}}\mathcal{T}^{*}(1)$. Hence  ${\mathcal{T}^{*}\left(1\right)}^{G_{\mathbb{Q}_v}} \subset H^0\left(\mathbb{Q}_v, \mathcal{\mathfrak{m}_{\mathcal{R}}}\mathcal{T}^{*}(1) \right)=\mathfrak{m}_{\mathcal{R}}\cdot{\mathcal{T}^{*}\left(1\right)}^{G_{\mathbb{Q}_v}}$. Thus ${\mathcal{T}^{*}\left(1\right)}^{G_{\mathbb{Q}_v}}$ is trivial by Nakayama's lemma which contradicts to our assumption (that ${\mathcal{T}^{*}\left(1\right)}^{G_{\mathbb{Q}_v}}$ has rank one). Thus $\left.\mathcal{T}^{*}\left(1\right)\right/{{\mathcal{T}^{*}\left(1\right)}^{G_{\mathbb{Q}_v}}}\otimes\left.\mathcal{R}\right/{\mathfrak{m}_{\mathcal{R}}}$ must have dimension one, whence $\left.\mathcal{T}^{*}\left(1\right)\right/{{\mathcal{T}^{*}\left(1\right)}^{G_{\mathbb{Q}_v}}}$ is a cyclic $\mathcal{R}$-module. This implies that $\left.\mathcal{T}^{*}\left(1\right)\right/{{\mathcal{T}^{*}\left(1\right)}^{G_{\mathbb{Q}_v}}}$ is either a torsion $\mathcal{R}$-module or $\mathcal{R}$-torsion free. Moreover, since we have assumed that ${\mathcal{T}^{*}\left(1\right)}^{G_{\mathbb{Q}_v}}$ has rank one, by taking the base extension $\otimes_{\mathcal{R}}\mathcal{K}$ to the following exact sequence
\begin{equation*}
0 \rightarrow {\mathcal{T}^{*}\left(1\right)}^{G_{\mathbb{Q}_v}} \rightarrow \mathcal{T}^{*}\left(1\right) \rightarrow \left.\mathcal{T}^{*}\left(1\right)\right/{{\mathcal{T}^{*}\left(1\right)}^{G_{\mathbb{Q}_v}}} \rightarrow 0
\end{equation*}
we have that $\left.\mathcal{T}^{*}\left(1 \right)\right/{{\mathcal{T}^{*}\left(1\right)}^{G_{\mathbb{Q}_v}}}$ is not $\mathcal{R}$-torsion. Thus $\left.\mathcal{T}^{*}\left(1\right)\right/{{\mathcal{T}^{*}\left(1\right)}^{G_{\mathbb{Q}_v}}}$ is a free $\mathcal{R}$-module of rank one and the assertion (iii) follows.

Let us verify the condition (iv). Let us take $v=p$ and recall that we have the exact sequence of $G_{\mathbb{Q}_p}$-modules
\begin{equation*}
0 \rightarrow F^{+}\mathcal{T}^{*}(1) \rightarrow \mathcal{T}^{*}(1) \rightarrow F^{-}\mathcal{T}^{*}(1) \rightarrow 0. 
\end{equation*}
Since $\overline{\rho}_{\mathcal{F}}\cong\mathbf{1}\oplus\chi$ and $\chi$ is primitive modulo $Np$, we have that the character by which $G_{\mathbb{Q}_p}$ acts on $F^{+}\mathbb{T}$ and $F^{-}\mathbb{T}$ are distinct modulo $\mathfrak{m}_{\mathcal{R}}$. Thus $F^{+}\mathbb{T}$ and $F^{-}\mathbb{T}$ are free $\mathbb{I}$-modules of rank one by \cite[Remark 2.13]{FO12}, so are the $\mathcal{R}$-modules $F^{+}\mathcal{T}^{*}(1)$ and $F^{-}\mathcal{T}^{*}(1)$. Since the semi-simplification of the $I_p$-module $\left.{\mathcal{T}^{*}(1)}\right/{\mathfrak{m}_{\mathcal{R}}\mathcal{T}^{*}(1)}$ is isomorphic to $\omega^{1-i}\oplus\left.\chi^{-1}\right|_{\left(\mathbb{Z}/p\mathbb{Z} \right)^{\times}}\omega^{1-i}$, we have that $F^{+}\mathcal{T}^{*}(1)$ and $F^{-}\mathcal{T}^{*}(1)$ are ramified under the assumptions $\omega^i\neq\omega$ and $\omega^i\neq\left.\chi^{-1}\omega\right|_{\left(\mathbb{Z}/p\mathbb{Z} \right)^{\times}}$. Hence the groups $H^0\left(\mathbb{Q}_p, F^{+}\mathcal{T}^{*}(1) \right)$ and $H^0\left(\mathbb{Q}_p, F^{-}\mathcal{T}^{*}(1) \right)$ vanish. This implies that $H^0\left(\mathbb{Q}_p, \mathcal{T}^{*}(1) \right)$ vanishes and the assertion (iv) follows. 

Under the assumption $\omega^i \neq \omega$ and $\chi^{-1}\omega$, the assertion (v) follows by the same proof of \cite[Theorem 4.10 and Corollary 4.12]{Ochiai06}. 

Let us verify the assertion (vi). Since $H_{\mathrm{ur}}^1\left(\mathbb{Q}_{v}, \mathcal{A} \right)$ is isomorphic to $\left(\mathcal{A}^{I_v} \right)_{G_{\mathbb{Q}_{v}}}$, whose Pontryagin dual is $\left(\left(\mathcal{T}^{*}\right)_{I_v} \right)^{G_{\mathbb{Q}_v}}$. Thus it is sufficient to prove that $\left(\mathcal{T}^{*}\right)_{I_v}$ has no nontrivial pseudo-null $\mathcal{R}$-submodule. Since $v\neq p$, $\left(\mathcal{T}^{*}\right)_{I_v}$ is isomorphic to $\left(\mathbb{T}^{*}\right)_{I_v}\hat{\otimes}_{\mathbb{Z}_p}\mathbb{Z}_p[[\Gamma]]\left(\tilde{\kappa} \right)\otimes\omega^{-i}$. Under the assumption that $\chi$ is primitive and $N$ is square-free, there exist a ring of integers $\mathcal{O} \subset \mathbb{I}$ of a finite extension of $\mathbb{Q}_p$ and integers $0 \leq r, r^{\prime} \leq \infty$ such that $\left(\mathbb{T}^{*}  \right)_{I_v}\cong \left.\mathbb{I}\right/\left(\varpi\right)^{r} \oplus \left.\mathbb{I}\right/\left(\varpi\right)^{r^{\prime}}$ (cf. the proof of Proposition \ref{Ochiai06s5}), where $\varpi$ is a fixed uniformizer of $\mathcal{O}$. Thus the assertion (vi) follows. 

Let us verify the assertion (vii). First we prove that $H^1\left(\mathbb{Q}_p, \mathcal{A} \right)^{\lor}$ is a torsion-free $\mathcal{R}$-module which is equivalent to that $H^1\left(\mathbb{Q}_p, \mathcal{A} \right)$ is a divisible $\mathcal{R}$-module. Let $a$ be an irreducible element of $\mathcal{R}$. We have that $\left.H^1\left(\mathbb{Q}_p, \mathcal{A} \right)\right/{a H^1\left(\mathbb{Q}_p, \mathcal{A} \right)}$ is isomorphic to an $\mathcal{R}$-submodule of $H^2\left(\mathbb{Q}_p, \mathcal{A}[a] \right)$. Thus it is enough to prove that $H^2\left(\mathbb{Q}_p, \mathcal{A}[a] \right)$ is trivial. We have that $H^2\left(\mathbb{Q}_p, \mathcal{A}[a] \right)$ is the Pontrygin dual of $H^{0}\left(\mathbb{Q}_p, \left.\mathcal{T}^{*}(1)\right/{a \mathcal{T}^{*}(1)} \right)$ by local Tate duality. By the argument in the verification of the assertion (iv), we have that $H^{0}\left(I_p, \left.\mathcal{T}^{*}(1)\right/{\mathfrak{m}_{\mathcal{R}} \mathcal{T}^{*}(1)} \right)$ is trivial under the assumptions $\omega^i\neq\omega$ and $\omega^i\neq\left.\chi^{-1}\omega\right|_{\left(\mathbb{Z}/p\mathbb{Z} \right)^{\times}}$. This implies that $H^{0}\left(\mathbb{Q}_p, \left.\mathcal{T}^{*}(1)\right/{\mathfrak{m}_{\mathcal{R}} \mathcal{T}^{*}(1)} \right)$ is trivial. Thus we have $H^{0}\left(\mathbb{Q}_p, \left.\mathcal{T}^{*}(1)\right/{a \mathcal{T}^{*}(1)} \right)$ is trivial by \cite[Lemma 2.2.6]{RG10} (We remark \cite[Lemma 2.2.6]{RG10} is still valid for $\mathbb{Q}_p$ since $G_{\mathbb{Q}_p}$ is topologically finitely generated). This proves that $H^1\left(\mathbb{Q}_p, \mathcal{A} \right)^{\lor}$ is a torsion-free $\mathcal{R}$-module. 

Next we claim that $H_{\mathrm{Gr}}^1\left(\mathbb{Q}_p, \mathcal{A} \right)$ fits into the following exact sequence
\begin{equation}\label{200810}
0 \rightarrow H_{\mathrm{Gr}}^1\left(\mathbb{Q}_p, \mathcal{A} \right) \rightarrow H^1\left(\mathbb{Q}_p, \mathcal{A} \right) \rightarrow H^1\left(\mathbb{Q}_p, F^{-}\mathcal{A} \right) \rightarrow 0.
\end{equation}
Indeed, the group $H^2\left(\mathbb{Q}_p, F^{+}\mathcal{A} \right)$ is the Pontryagin dual of $H^0\left(\mathbb{Q}_p, \left(F^{+}\mathcal{T}\right)^{*}(1) \right)$ by local Tate duality. Hence the group $H^2\left(\mathbb{Q}_p, F^{+}\mathcal{A} \right)$ vanishes by the argument in the assertion (iv). This implies that the map $H^1\left(\mathbb{Q}_p, \mathcal{A} \right) \rightarrow H^1\left(\mathbb{Q}_p, F^{-}\mathcal{A} \right)$ is surjective. Thus it is sufficient to prove that the map $H^1\left(\mathbb{Q}_p, F^{-}\mathcal{A} \right) \rightarrow H^1\left(I_p, F^{-}\mathcal{A} \right)^{G_{\mathbb{Q}_p}}$ is an isomorphism. The cokernel of the map is isomorphic to a subgroup of $H^2\left(\left.G_{\mathbb{Q}_p}\right/{I_p}, \left(F^{-}\mathcal{A}\right)^{I_p} \right)$, which vanishes since $\left.G_{\mathbb{Q}_p}\right/{I_p}$ has cohomological dimension one. The kernel of the map is isomorphic to $\left(\left(F^{-}\mathcal{A}\right)^{I_p} \right)_{\left.G_{\mathbb{Q}_p}\right/{I_p}}$. Since the group $\left.G_{\mathbb{Q}_p}\right/{I_p}$ is generated by $\mathrm{Frob}_p$ which acts on $F^{-}\mathbb{T}$ by $a\left(p, \mathcal{F} \right)$, the group $\left(\left(F^{-}\mathcal{A}\right)^{I_p} \right)_{\left.G_{\mathbb{Q}_p}\right/{I_p}}$ vanishes by \eqref{2008121}. Thus the exactness of \eqref{200810} follows. 

Let us take the Pontryagin dual of \eqref{200810}. Since we have proved that $H^1\left(\mathbb{Q}_p, \mathcal{A} \right)^{\lor}$ is a torsion-free $\mathcal{R}$-module, which implies that has no non-trivial pseudo-null $\mathcal{R}$-submodule. On the other hand, the $\mathcal{R}$-module $H^1\left(\mathbb{Q}_p, F^{-}\mathcal{A} \right)^{\lor}$ is reflexive by \cite[Lemma 8.5]{Ochiai06}. This verifies the assertion (vii) by \cite[Lemma 8.7]{Ochiai06}. 

Under the assumption $\omega^i \neq \omega$ and $\chi^{-1}\omega$, the assertion (viii) follows immediately. Thus the proof is complete.

\end{proof}

\subsection{Discussion of the main conjecture under specialization}

In this section, we discuss the Iwasawa main conjecture for one- and two-variable Hida deformation by specialization. 

Firstly we recall the main conjecture for general Galois deformation. Let $\mathscr{R}$ be a complete Noetherian local domain such that the residue field is a finite extension of $\mathbb{F}_p$. Let $\mathscr{T}$ be a free $\mathscr{R}$-module of finite rank on which $G_{\mathbb{Q}}$ acts continuously and unramified outside a finite set of primes $\Sigma \supset \set{p, \infty}$. Let $\mathscr{A}:=\mathscr{T}\otimes_{\mathscr{R}}\mathscr{R}^{\lor}$. We refer the reader to \cite[\S 1]{Ochiai10} for the conditions ($\mathbf{Geom}$), ($\mathbf{Pan}$), ($\mathbf{Crit}$) and ($\mathbf{NV}$). In particular, the condition ($\mathbf{Pan}$) enables us to define a $G_{\mathbb{Q}_p}$-submodule $F^{+}\mathscr{A}$ of $\mathscr{A}$. The Iwasawa main conjecture for $\mathscr{T}$ is the following conjecture which is proposed by Greenberg \cite[Conjecture 4.1]{RG} and modified by Ochiai \cite[Conjecture C]{Ochiai10}.

\begin{conj}\label{IMCGalois}
Let $S \subset \mathrm{Hom}_{\mathrm{cont}}\left(\mathscr{R}, \overline{\mathbb{Q}_p} \right)$ be a Zariski dense subset. Assume the conditions ($\mathbf{Geom}$), ($\mathbf{Pan}$), ($\mathbf{Crit}$) and ($\mathbf{NV}$) for the pair $\left(\mathscr{R}, \mathscr{T}, S \right)$. Assume further the existence of the analytic $p$-adic $L$-function $L_p\left(\mathscr{T} \right)\in \mathrm{Frac}\left(\mathscr{R}\right)$ (cf. \cite[page 218]{RG}). Then we have the following equality
\begin{equation}\label{20200623}
e_{\mathscr{A}}\cdot\mathrm{char}_{\mathscr{R}}\left(\mathrm{Sel}_{\mathscr{A}} \right)^{\lor}\mathrm{char}_{\mathscr{R}}\left(H^0\left(\mathbb{Q}, \mathscr{A} \right)^{\lor} \right)^{-1}\mathrm{char}_{\mathscr{R}}\left(H^0\left(\mathbb{Q}, \mathscr{A}^{*} \right)^{\lor} \right)^{-1}=\left(L_p\left(\mathscr{T} \right) \right),
\end{equation}
where $\mathscr{A}:=\mathscr{T}\otimes_{\mathscr{R}}\mathscr{R}^{\lor}, \mathscr{A}^{*}:=\mathrm{Hom}\left(\mathscr{T}, \mu_{p^{\infty}} \right)$ and 
\begin{equation*}
e_{\mathscr{A}}=\begin{cases}
\mathrm{char}_{\mathscr{R}}\left(H^0\left(\mathbb{Q}_p, \left.\mathscr{A}\right/{F^{+}\mathscr{A}} \right)^{\lor} \right)& \text{if}\ \left.\mathscr{A}\right/{F^{+}\mathscr{A}}\ \text{is unramified} \\
0& \text{otherwise}.
\end{cases}
\end{equation*}
\end{conj}

When $\mathscr{T}$ is the cyclotomic deformation of a Dirichlet character, the term $\mathrm{char}_{\mathcal{R}}\left(H^0\left(\mathbb{Q}, \mathscr{A} \right)^{\lor} \right)$ and $\mathrm{char}_{\mathcal{R}}\left(H^0\left(\mathbb{Q}, \mathscr{A}^{*} \right)^{\lor} \right)$ corresponds to the denominator of Kubota-Leopoldt $p$-adic $L$-function (cf. \cite[\S 1]{RG}). The first example where the main conjecture dose not hold without $e_{\mathscr{A}}$ is established in \cite[\S 7-(c)]{Ochiai06}. 

Let us keep the notation of Hida deformation as \S 4.2. First we formulate Conjecture \ref{IMCGalois} when $\mathscr{T}:=\mathbb{T}\hat{\otimes}_{\mathbb{Z}_p}\mathbb{Z}_p[[\Gamma]]\left(\tilde{\kappa}^{-1} \right)\otimes_{\mathbb{Z}_p}\omega^i$ comes from a residually reducible Hida family $\mathcal{F}$. We assume that $\mathbb{I}$ is Gorenstein and $\left.\chi\right|_{\left(\left.\mathbb{Z}\right/p\mathbb{Z} \right)^{\times}} \neq \mathbf{1}, \omega$. Then $\rho_{\mathcal{F}}$ satisfies the condition (Lat-fr). Let $\mathbb{T}$ be a stable free lattice of $\mathbb{V}_{\mathcal{F}}$. Let $\mathbb{A}=\mathbb{T}\otimes_{\mathbb{I}}\mathbb{I}^{\lor}, \mathcal{T}=\mathbb{T}\hat{\otimes}_{\mathbb{Z}_p}\mathbb{Z}_p[[\Gamma]]\left(\tilde{\kappa}^{-1} \right)$ and $\mathcal{A}=\mathcal{T}\otimes_{\mathcal{R}}\mathcal{R}^{\lor}$. We recall that under the assumption that $\mathbb{I}$ is Gorenstein, $\chi$ is primitive modulo $Np$ and $\left.\chi\right|_{\left(\left.\mathbb{Z}\right/p\mathbb{Z} \right)^{\times}} \neq \omega$, the two-variable $p$-adic $L$-function $L_p^{\mathrm{Ki}}\left(\mathcal{T}\otimes\omega^i \right)$ is constructed to be an element in $\mathcal{R}\otimes_{\mathbb{Z}_p}\mathbb{Q}_p$ by Mazur and Kitagawa \cite{Ki94}. The one-variable $p$-adic $L$-function $L_p^{\mathrm{Ki}}\left(\mathbb{T}\otimes\omega^i \right) \in \mathbb{I}\otimes_{\mathbb{Z}_p}\mathbb{Q}_p$ is defined by the image of $L_p^{\mathrm{Ki}}\left(\mathcal{T}\otimes\omega^i \right)$ under $\mathcal{R} \twoheadrightarrow \left.\mathcal{R}\right/{\left(\gamma-1 \right)}$. Since we are interested in the case where the terms $e_{\mathcal{A}\otimes\omega^i}$ and $\mathrm{char}_{\mathcal{R}}\left(H^0\left(\mathbb{Q}, \mathcal{A}\otimes\omega^i \right)^{\lor} \right)$ do not vanish, we consider the case when $\omega^i=\mathbf{1}$ (i.e $i \equiv 0 \pmod{p-1}$). Then by Proposition \ref{H0}, $H^{0}\left(\mathbb{Q}, \mathcal{A} \right)$ is a pseudo-null $\mathcal{R}$-module. Thus Conjecture \ref{IMCGalois} for $\mathcal{T}$ is the following conjecture. 

\begin{conj}\label{IMCtwovar}
Suppose that $\mathbb{I}$ is a UFD and Gorenstein. Suppose $\left.\chi\right|_{\left(\left.\mathbb{Z}\right/p\mathbb{Z} \right)^{\times}} \neq \mathbf{1}, \omega$. Assume $p \nmid \varphi\left(N \right)$ and that the Eisenstein ideal $\mathbb{J}$ is principal. Let $\mathbb{T}$ be a stable free lattice. Then we have the following equality
\begin{equation}\label{20200625}
\mathrm{char}_{\mathcal{R}}\left(\mathrm{Sel}_{\mathcal{A}} \right)^{\lor}=\left(L_p^{\mathrm{Ki}}\left(\mathcal{T} \right) \right).
\end{equation}
\end{conj}

Let us keep the assumptions of Conjecture \ref{IMCtwovar}. We consider the main conjecture for $\mathbb{T}$. Under the assumption $\chi \neq \omega$, the $\mathbb{I}$-module $H^0\left(\mathbb{Q}, \mathbb{A}^* \right)$ is trivial. However by Proposition \ref{H0} we know that the $\mathbb{I}$-module $H^0\left(\mathbb{Q}, \mathbb{A} \right)$ is not pseudo-null. Thus Conjecture \ref{IMCGalois} for $\mathbb{T}$ is the following conjecture.

\begin{conj}\label{IMCHida}
Let us keep the assumptions of Conjecture \ref{IMCtwovar}. Then we have the following equality
\begin{equation}\label{20200626}
\left(a\left(p, \mathcal{F} \right)-1 \right)\mathrm{char}_{\mathbb{I}}\left(\mathrm{Sel}_{\mathbb{A}} \right)^{\lor}{\mathrm{char}_{\mathbb{I}}\left({H^0\left(\mathbb{Q}, \mathbb{A} \right)}^{\lor} \right)}^{-1}=\left(L_p^{\mathrm{Ki}}\left(\mathbb{T} \right) \right).
\end{equation}

\end{conj}

We will study the relation between Conjecture \ref{IMCtwovar} and Conjecture \ref{IMCHida} by specializing the characteristic ideal $\mathrm{char}_{\mathcal{R}}\left(\mathrm{Sel}_{\mathcal{A}} \right)^{\lor}$. We have the following corollary on the relation between the IMC for $\mathcal{T}$ and $\mathbb{T}$. 
\begin{cor}\label{IMCs}
Let us keep the assumptions of Conjecture \ref{IMCHida}. Let us take a stable free lattice $\mathbb{T}$. Assume further that $N$ is square free. Then Conjecture \ref{IMCtwovar} implies Conjecture \ref{IMCHida} i.e. the two-variavle IMC for $\mathcal{T}$ implies the one-variable IMC for $\mathbb{T}$. 
\end{cor}
\begin{proof}
We decompose $\mathbb{J}$ into product of height-one prime ideals as $\mathbb{J}=\mathfrak{p}_1^{n_1}\cdots\mathfrak{p}_r^{n_r}$, where $\mathfrak{p}_i$ and $\mathfrak{p}_j$ are distinct when $i\neq j$. Let $\left(j_1, \cdots, j_r \right)$ be the vertex of $\mathrm{Rect}_{\left(0, \cdots, 0\right)}^{\left(n_1, \cdots, n_r \right)}$ which is the image of $[\mathbb{T}]$ under $\Phi_{\mathbf{1}}$ in Theorem \ref{Gal}. We have 
\begin{equation*}
\mathrm{char}_{\mathbb{I}}\left({H^0\left(\mathbb{Q}, \mathbb{A} \right)}^{\lor}\right)=\dfrac{\mathbb{J}}{\mathfrak{p}_1^{j_1}\cdots\mathfrak{p}_r^{j_r}} 
\end{equation*}
by Proposition \ref{H0}. On the other hand, by combining Proposition \ref{nontrivialpseudo-null} and the structure theorem of finitely generated $\mathcal{R}$-modules, we have the following equality
\begin{equation*}
\Psi\left(  \mathrm{char}_{\mathcal{R}}\left(\mathrm{Sel}_{\mathcal{A}} \right)^{\lor}  \right)=\mathrm{char}_{\left.\mathcal{R}\right/{\left(\gamma-1 \right)}}\left.\left(\mathrm{Sel}_{\mathcal{A}} \right)^{\lor}\right/{\left(\gamma-1 \right)\left(\mathrm{Sel}_{\mathcal{A}} \right)^{\lor}},
\end{equation*}
where $\Psi$ denotes the surjection $\mathcal{R} \twoheadrightarrow \left.\mathcal{R}\right/{\left(\gamma-1 \right)} \stackrel{\sim}{\rightarrow} \mathbb{I}$. Then by Theorem \ref{T3}-(2), we have 
\begin{equation}\label{20201216}
\mathrm{char}_{\mathbb{I}}\left(\mathrm{Sel}_{\mathbb{A}} \right)^{\lor}=\Psi\left(  \mathrm{char}_{\mathcal{R}}\left(\mathrm{Sel}_{\mathcal{A}} \right)^{\lor}  \right)\dfrac{\mathbb{J}}{\mathfrak{p}_1^{j_1}\cdots\mathfrak{p}_r^{j_r}\left(a\left(p, \mathcal{F} \right)-1 \right)}.
\end{equation}
This completes the proof.
\end{proof}

Now we study the change of $\mathrm{char}_{\mathbb{I}}\left(\mathrm{Sel}_{\mathbb{A}\otimes\omega^i} \right)^{\lor}$ for one-variable Hida deformation when $\mathbb{T}$ varies in the set of all stable lattices for even $i$. When $\omega^i$ is non-trivial, one could apply Theorem \ref{O}. Then by Theorem \ref{Gal}, we have the same result as Corollary \ref{C2}. However the condition $(\mathbf{F}_{\mathbb{Q}})$ fails when $\omega^i$ is trivial by Proposition \ref{H0}. Thus one could not apply Ochiai's Theorem (Theorem \ref{O}) to study the change of $\mathrm{char}_{\mathbb{I}}\left(\mathrm{Sel}_{\mathbb{A}} \right)^{\lor}$. However, we could lift them to the two-variable case and specialize them by control theorem i.e by combining Corollary \ref{C2} and Theorem \ref{T3}. We obtain the following corollary.
\begin{cor}\label{alglonevar}
Let us keep the assumptions of Corollary \ref{IMCs}. Assume further that $\mathbb{I}$ is isomorphic to $\mathcal{O}[[X]]$ for the ring of integers $\mathcal{O}$ of a finite extension of $\mathbb{Q}_p$. Let $\mathbb{T}$ and $\mathbb{T}^{\prime}$ be stable free lattices. Then we have $\mathrm{char}_{\mathbb{I}}\left(\mathrm{Sel}_{\mathbb{A}^{\prime}} \right)^{\lor}=\mathrm{char}_{\mathbb{I}}\left(\mathrm{Sel}_{\mathbb{A}} \right)^{\lor}$.

\end{cor}
\begin{proof}
Let us keep the notation in the proof of Corollary \ref{IMCs}. Let $\left(j_1, \cdots, j_r \right)$ be the vertex of $\mathrm{Rect}_{\left(0, \cdots, 0\right)}^{\left(n_1, \cdots, n_r \right)}$ which is the image of $[\mathbb{T}]$ under $\Phi_{\mathbf{1}}$ in Theorem \ref{Gal} and $\mathbb{T}^{\mathrm{min}}$ a stable free lattice which is a representative of $\Phi_{\mathbf{1}}^{-1}\left(\left(0, \cdots, 0 \right) \right)$. Let $\mathcal{A}:=\mathbb{T}\hat{\otimes}\mathbb{Z}_p[[\Gamma]]\left(\tilde{\kappa}^{-1} \right)\otimes_{\mathcal{R}}\mathcal{R}^{\lor}$ and we define $\mathcal{A}^{\mathrm{min}}$ in a similar way. Then by Corollary \ref{C2}, the equality \eqref{20201216} becomes to
\begin{align*}
\mathrm{char}_{\mathbb{I}}\left(\mathrm{Sel}_{\mathbb{A}} \right)^{\lor}&=\Psi\left(  \mathrm{char}_{\mathcal{R}}\left(\mathrm{Sel}_{\mathcal{A}^{\mathrm{min}}} \right)^{\lor}  \mathfrak{p}_1^{j_1}\cdots\mathfrak{p}_r^{j_r}\mathcal{R}\right)\dfrac{\mathbb{J}}{\mathfrak{p}_1^{j_1}\cdots\mathfrak{p}_r^{j_r}\left(a\left(p, \mathcal{F} \right)-1 \right)}.\\
&=\Psi\left(  \mathrm{char}_{\mathcal{R}}\left(\mathrm{Sel}_{\mathcal{A}^{\mathrm{min}}} \right)^{\lor}  \right)\dfrac{\mathbb{J}}{\left(a\left(p, \mathcal{F} \right)-1 \right)}
\end{align*}
which is independent on $\mathbb{T}$.
\end{proof}
\begin{rem}\label{modifychange}
Let $\mathbb{T}$ be a free lattice, since $\mathbb{J}$ is a factor of Kubota-Leopoldt $p$-adic $L$-function which is not divided by $p$, the term $\mathrm{char}_{\mathbb{I}}\left({H^0\left(\mathbb{Q}, \mathbb{A} \right)}^{\lor}\right)$ does not correspond the denominator of $L_p\left(\mathbb{T} \right)$. In particular if we assume the two-variable IMC (Conjecture \ref{IMCtwovar}), the two-variable $p$-adic $L$-function $L_p^{\mathrm{Ki}}\left(\mathcal{T} \right)$ is in $\mathcal{R}$ and hence $L_p\left(\mathbb{T} \right)$ is an element of $\mathbb{I}$. 

On the other hand, $L_p^{\mathrm{Ki}}\left(\mathcal{T} \right)$ is defined by a basis of $\mathbb{I}$-adic modular symbol associated to $\mathbb{T}$. Then by multiplying $\mathbb{T}$ by elements of $\mathbb{K}^{\times}$, we have that $\left.\mathbb{T}\right/{\mathbb{T}^{\mathrm{min}}}$ is a cyclic $\mathbb{I}$-module on which $G_{\mathbb{Q}}$ acts trivially. Thus one could have the same results for the change of $L_p^{\mathrm{Ki}}\left(\mathcal{T} \right)$ as Corollary \ref{C2} and hence $L_p^{\mathrm{Ki}}\left(\mathbb{T} \right)$ changes in the same manner. However, the Selmer group $\mathrm{char}_{\mathbb{I}}\left(\mathrm{Sel}_{\mathbb{A}} \right)^{\lor}$ does not change by Corollary \ref{alglonevar}. Thus by Proposition \ref{H0}, one may view $\mathrm{char}_{\mathbb{I}}\left(H^0\left(\mathbb{Q}, \mathbb{A} \right)^{\lor}\right)$ as making the algebraic side depending on $\mathbb{T}$ and matching every $p$-adic $L$-function $L_p^{\mathrm{Ki}}\left(\mathbb{T} \right)$ in the main conjecture.

\end{rem}

In this subsection, we study the statement of the main conjecture for one-variable ordinary deformation $\mathbb{T}$. In the following section, we study for the main conjecture of ``Eisenstein deformation" $\left.\mathcal{T}\right/\mathfrak{p}_t\mathcal{T}$ for a height-one prime ideal $\mathfrak{p}_t$ dividing $\mathbb{J}$.

\section{The phenomenon of trivial zeros}
Let us keep the notation of the previous section and let $\mathfrak{p}$ be a height-one prime ideal diving $\mathbb{J}$. In this section, we study Selmer groups $\mathrm{Sel}_{\mathcal{A}[\mathfrak{p}]}$ for $\left.\mathcal{T}\right/\mathfrak{p}\mathcal{T}$. We assume the following condition throughout this section
\begin{list}{}{}
\item[(Vand $\chi^{-1}\omega$)]The $\chi^{-1}\omega$-part of the ideal class group of $\mathbb{Q}\left(\mu_{Np} \right)$ is trivial.
\end{list}
\begin{lem}\label{0}
Suppose that $\mathbb{I}$ is isomorphic to $\mathcal{O}[[X]]$ for the ring of integers $\mathcal{O}$ of a finite extension of $\mathbb{Q}_p$. Assume $p \nmid \varphi\left(N \right)$, that $N$ is square free, (Vand $\chi^{-1}\omega$) and that the Eisenstein ideal $\mathbb{J}$ is principal. Suppose $\left.\chi\right|_{\left(\left.\mathbb{Z}\right/p\mathbb{Z} \right)^{\times}} \neq \mathbf{1}$ and $\omega$. Then $\mathrm{Sel}_{\mathbb{A}}$ is trivial for any stable free lattice. 
\end{lem}
\begin{proof}
We decompose $\mathbb{J}$ into product of height-one prime ideals as $\mathbb{J}=\mathfrak{p}_1^{n_1}\cdots\mathfrak{p}_r^{n_r}$, where $\mathfrak{p}_i$ and $\mathfrak{p}_j$ are distinct when $i\neq j$.  Let $\mathbb{T}^{\mathrm{max}}$ be a stable free lattice corresponding the vertex $\left(n_1, \cdots, n_r\right)$ under $\Phi_{\mathbf{1}}$ in Theorem \ref{Gal}. We first prove the triviality of $\mathrm{Sel}_{\mathbb{A}^{\mathrm{max}}}$. Let us take a height-one prime ideal $P=\mathrm{Ker}\left(\phi \right)\mathcal{R}$ for some arithmetic specialization $\phi$. We have that $\mathrm{Sel}_{\mathbb{A}^{\mathrm{max}}[P]}$ is isomorphic to $\mathrm{Sel}_{\mathbb{A}^{\mathrm{max}}}[P]$ by Theorem \ref{T3}-(1). By Step \ref{s5} in the proof of Proposition \ref{H0}, we have that $\left.\mathbb{T}^{\mathrm{max}}\right/\mathfrak{m}_{\mathbb{I}}\mathbb{T}^{\mathrm{max}}$ is a nontrivial extension of $\mathbf{1}$ by $\chi$. Thus under the condition (Vand $\chi^{-1}\omega$), we have that $\mathrm{Sel}_{\mathbb{A}^{\mathrm{max}}[P]}$ is trivial by \cite[Proposition 5.8]{BP19}. Thus $\mathrm{Sel}_{\mathbb{A}^{\mathrm{max}}}[P]$ is trivial and the triviality of $\mathrm{Sel}_{\mathbb{A}^{\mathrm{max}}}$ follows by Nakayama's lemma. 

For other stable free lattice $\mathbb{T}$, we have that $\mathrm{char}_{\mathbb{I}}\left(\mathrm{Sel}_{\mathbb{A}} \right)^{\lor}$ is equal to $\mathrm{char}_{\mathbb{I}}\left(\mathrm{Sel}_{\mathbb{A}^{\mathrm{max}}} \right)^{\lor}$ by Corollary \ref{alglonevar}. This implies that $\left(\mathrm{Sel}_{\mathbb{A}} \right)^{\lor}$ is a pseudo-null $\mathbb{I}$-module. Thus $\mathrm{Sel}_{\mathbb{A}}$ is trivial by Proposition \ref{nontrivialpseudo-null} and its proof, the proof is complete. 
\end{proof}

We would like to study the Iwasawa theory of the cyclotomic deformation of the residual representation $\left.\mathbb{T}\right/\mathfrak{p}\mathbb{T}$ for a height-one prime ideal $\mathfrak{p}$ dividing $\mathbb{J}$. First we study that for a given $\mathfrak{p}$, which free lattice makes the Selmer group torsion. We assume the following condition from now on to the end of the paper. 
\begin{list}{}{}
\item[($p$Four)]There exists a height-one prime ideal $\mathfrak{p}$ such that $\mathfrak{p} \in V\left(\mathbb{J} \right)$ and $\mathfrak{p} \not\in V\left(\left(a\left(p, \mathcal{F} \right)-1 \right)\mathbb{J}^{-1} \right)$.
\end{list}
\begin{thm}\label{toreis}
Let us keep the assumptions of Lemma \ref{0}. We decompose $\mathbb{J}$ into product of height-one prime ideals as $\mathbb{J}=\mathfrak{p}_1^{n_1}\cdots\mathfrak{p}_r^{n_r}$, where $\mathfrak{p}_i$ and $\mathfrak{p}_j$ are distinct when $i\neq j$. Let $\mathfrak{p}_i$ be a height-one prime ideal which satisfies ($p$Four). Then we have the following statements
\begin{enumerate}
\item[(1)]Let $\mathbb{T}$ be a stable free lattice which corresponds to the vertex $\left(j_1. \cdots, j_r \right)$ with $j_i=0$, then $\left(\mathrm{Sel}_{\mathcal{A}[\mathfrak{p}_i]}\right)^{\lor}$ is a finitely generated torsion $\left.\mathcal{R}\right/{\mathfrak{p}_i}$-module. Moreover, we have
\begin{equation*}
\mathrm{char}_{\left.\mathcal{R}\right/{\mathfrak{p}_i}}\left(\mathrm{Sel}_{\mathcal{A}[\mathfrak{p}_i]}\right)^{\lor} \subset \left(\gamma-1 \right).
\end{equation*} 
\item[(2)]Let $\mathbb{T}$ be a stable free lattice which corresponds to the vertex $\left(j_1. \cdots, j_r \right)$ with $j_i>0$, then $\left(\mathrm{Sel}_{\mathcal{A}[\mathfrak{p}_i]}\right)^{\lor}$ has $\left.\mathcal{R}\right/{\mathfrak{p}_i}$-rank euqal to one. In particular, when $j_i=n_i$, $\left(\mathrm{Sel}_{\mathcal{A}[\mathfrak{p}_i]}\right)^{\lor}$ is a free $\left.\mathcal{R}\right/{\mathfrak{p}_i}$-module. 
\end{enumerate}
\end{thm}
\begin{proof}
Let $\mathbb{T}$ be a stable free lattice and $\left(j_1, \cdots, j_r \right)$ the corresponding vertex. Let $\mathfrak{p}_i$ be a prime factor of $\mathbb{J}$ which satisfies ($p$Four). We have the following exact sequence by Theorem \ref{T3}-(2):
\begin{equation}\label{conex}
0 \rightarrow \mathrm{Sel}_{\mathcal{A}[\gamma-1]} \rightarrow \mathrm{Sel}_{\mathcal{A}}[\gamma-1] \rightarrow \mathcal{R}^{\lor}\left[\gamma-1, \dfrac{\left(a\left(p, \mathcal{F}\right)-1\right)}{\mathfrak{p}_1^{n_1-j_1}\cdots\mathfrak{p_r}^{n_r-j_r}} \right] \rightarrow  0.
 \end{equation}
The module $\mathrm{Sel}_{\mathcal{A}[\gamma-1]}$ is trivial by Lemma \ref{0}. Then \eqref{conex} becomes an isomorphism
\begin{equation}\label{27}
\mathrm{Sel}_{\mathcal{A}}[\gamma-1]\stackrel{\sim}{\rightarrow}\mathcal{R}^{\lor}\left[\gamma-1, \dfrac{\left(a\left(p, \mathcal{F}\right)-1\right)}{\mathfrak{p}_1^{n_1-j_1}\cdots\mathfrak{p_r}^{n_r-j_r}} \right],
\end{equation}
whence
\begin{equation}\label{isom}
\mathrm{Sel}_{\mathcal{A}}[\gamma-1, \mathfrak{p}_i] \stackrel{\sim}{\rightarrow} \mathcal{R}^{\lor}\left[\gamma-1, \dfrac{\left(a\left(p, \mathcal{F}\right)-1\right)}{\mathfrak{p}_1^{n_1-j_1}\cdots\mathfrak{p_r}^{n_r-j_r}}, \mathfrak{p}_i \right].
\end{equation}

We recall that $\mathfrak{p}_i$ is not lying above $\left(p \right) \in \Lambda$ by Ferrero-Washington's theorem. This implies that $\left.\mathcal{R}\right/\mathfrak{p}_i$ is isomorphic to $\mathcal{O}^{\prime}[[\Gamma]]$ for a finite extension $\mathcal{O}^{\prime}$ of $\mathbb{Z}_p$. Suppose that $\left(\mathrm{Sel}_{\mathcal{A}}[\mathfrak{p}_i] \right)^{\lor}$ is not $\left.\mathcal{R}\right/\mathfrak{p}_i$-torsion, then $\left.\left(\mathrm{Sel}_{\mathcal{A}}[\mathfrak{p}_i] \right)^{\lor}\right/{\left(\gamma-1 \right)\left(\mathrm{Sel}_{\mathcal{A}}[\mathfrak{p}_i] \right)^{\lor}}$ must be infinite by the structure theorem of Iwasawa modules. Let $j_i=0$. Since $\mathfrak{p}_i$ does not divide $\left(a\left(p, \mathcal{F}\right)-1\right)\cdot\mathbb{J}^{-1}$, the module in the right hand side of \eqref{isom} is finite. Thus by the above arguments on $\left(\mathrm{Sel}_{\mathcal{A}}[\mathfrak{p}_i] \right)^{\lor}$ we have that $\left(\mathrm{Sel}_{\mathcal{A}}[\mathfrak{p}_i] \right)^{\lor}$ must be $\left.\mathcal{R}\right/\mathfrak{p}_i$-torsion when $j_i=0$. We also have the following exact sequence by Theorem \ref{T3}-(3):
\begin{equation}\label{0112}
0 \rightarrow \left(\left.\mathcal{R}\right/\mathfrak{p}_i\right)^{\lor}[\gamma-1] \rightarrow \mathrm{Sel}_{\mathcal{A}[\mathfrak{p}_i]} \rightarrow \mathrm{Sel}_{\mathcal{A}}[\mathfrak{p}_i] \rightarrow 0. 
\end{equation}
Thus $\left(\mathrm{Sel}_{\mathcal{A}[\mathfrak{p}_i]} \right)^{\lor}$ is $\left.\mathcal{R}\right/\mathfrak{p}_i$-torsion and
\begin{equation*}
\mathrm{char}_{\left.\mathcal{R}\right/\mathfrak{p}_i}\left(\mathrm{Sel}_{\mathcal{A}[\mathfrak{p}_i]} \right)^{\lor}=\left(\gamma-1 \right)\mathrm{char}_{\left.\mathcal{R}\right/\mathfrak{p}_i}\left(\mathrm{Sel}_{\mathcal{A}}[\mathfrak{p}_i] \right)^{\lor}.
\end{equation*}

Now let us assume $j_i>0$ and let $\Xi: \mathcal{R} \rightarrow \overline{\mathbb{Q}}_p[[\Gamma]]$ be the composition of the projection $\mathcal{R} \twoheadrightarrow \left.\mathcal{R}\right/{\mathfrak{p}_i}$ with a fixed embedding $\left.\mathcal{R}\right/{\mathfrak{p}_i} \hookrightarrow \overline{\mathbb{Q}}_p[[\Gamma]]$. We have the following equality by Corollary \ref{C2}
\begin{equation*}
\mathrm{char}_{\left.\mathcal{R}\right/\mathfrak{p}_i}\left(\mathrm{Sel}_{\mathcal{A}}[\mathfrak{p}_i] \right)^{\lor}=\Xi\left(\mathfrak{p}_1^{j_1}\cdots\mathfrak{p}_n^{j_n} \right)\cdot\Xi\left(\mathrm{char}_{\mathcal{R}}\left(\mathrm{Sel}_{\mathcal{A}^{\mathrm{min}}} \right)^{\lor} \right).
\end{equation*}
The right hand side is $0$ since $j_i>0$. Hence $\left(\mathrm{Sel}_{\mathcal{A}}[\mathfrak{p}_i] \right)^{\lor}$ is not $\left.\mathcal{R}\right/\mathfrak{p}_i$-torsion. We have that $\left(\mathrm{Sel}_{\mathcal{A}}[\gamma-1] \right)^{\lor}$ is a cyclic $\mathcal{R}$-module by \eqref{27} and hence $\left(\mathrm{Sel}_{\mathcal{A}} \right)^{\lor}$
is cyclic follows by Nakayama's lemma. Thus $\left(\mathrm{Sel}_{\mathcal{A}}[\mathfrak{p}_i] \right)^{\lor}$ is a free $\left.\mathcal{R}\right/\mathfrak{p}_i$-module of rank one. Note that the exact sequence \eqref{0112} holds unless $j_i \neq n_i$ by Theorem \ref{T3}-(3). Thus $\left(\mathrm{Sel}_{\mathcal{A}[\mathfrak{p}_i]} \right)^{\lor}$ has $\left.\mathcal{R}\right/{\mathfrak{p}_i}$-rank equal to one. When $j_i=n_i$, $\mathrm{Sel}_{\mathcal{A}[\mathfrak{p}_i]}$ is isomorphic to $\mathrm{Sel}_{\mathcal{A}}[\mathfrak{p}_i]$ and hence $\left(\mathrm{Sel}_{\mathcal{A}[\mathfrak{p}_i]} \right)^{\lor}$ is a free $\left.\mathcal{R}\right/{\mathfrak{p}_i}$-module of rank one. This completes the proof of Theorem \ref{toreis}.

\end{proof}

Now we study the statement of the main conjecture for $\left.\mathcal{T}\right/\mathfrak{p}_i\mathcal{T}$. We denote by $L_p\left(\left.\mathcal{T}\right/\mathfrak{p}_i\mathcal{T} \right)$ the image of $L_p^{\mathrm{Ki}}\left(\mathcal{T} \right)$ in $\left.\mathcal{R}\right/{\mathfrak{p}_i}\otimes_{\mathbb{Z}_p}\mathbb{Q}_p$. By Conjecture \ref{IMCGalois} and Proposition \ref{H0}, the main conjecture for $\left.\mathcal{T}\right/\mathfrak{p}_i\mathcal{T}$ is the following statement
\begin{conj}\label{IMCEiss}
Suppose that $\mathbb{I}$ is a UFD and Gorenstein. Suppose $\left.\chi\right|_{\left(\left.\mathbb{Z}\right/p\mathbb{Z} \right)^{\times}} \neq \mathbf{1}$ and $\omega$. Assume $p \nmid \varphi\left(N \right)$ and that $\mathbb{J}$ is principal. Let $\mathbb{T}$ be a stable free lattice such that $\left(\mathrm{Sel}_{\mathcal{A}[\mathfrak{p}_i]}\right)^{\lor}$ is a torsion $\left.\mathcal{R}\right/{\mathfrak{p}_i}$-module. Then we have the following equality
\begin{equation*}
\mathrm{char}_{\left.\mathcal{R}\right/\mathfrak{p}_i}\left(H^0\left(\mathbb{Q}, \mathcal{A} [\mathfrak{p}_i]\right)^{\lor} \right)^{-1}\mathrm{char}_{\left.\mathcal{R}\right/\mathfrak{p}_i}\left(\mathrm{Sel}_{\mathcal{A} [\mathfrak{p}_i]} \right)^{\lor}=\left(L_p\left(\left.\mathcal{T}\right/\mathfrak{p}_i\mathcal{T}\right) \right).
\end{equation*}

\end{conj}

We obtain the following corollary

\begin{cor}\label{sp}
Let us keep the assumptions and the notation of Lemma \ref{0} and Theorem \ref{toreis}. Let $\mathfrak{p}_i$ be a height-one prime ideal which satisfies  the condition ($p$Four). Let $\mathbb{T}$ be a stable free lattice which corresponds to the vertex $\left(j_1. \cdots, j_r \right)$ with $j_i=0$. Then $\left(\mathrm{Sel}_{\mathcal{A}[\mathfrak{p}_i]}\right)^{\lor}$ is a finitely generated torsion $\left.\mathcal{R}\right/{\mathfrak{p}_i}$-module and Conjecture \ref{IMCtwovar} implies Conjecture \ref{IMCEiss}. 
\end{cor}
\begin{proof}
We have
\begin{equation*}
\mathrm{char}_{\left.\mathcal{R}\right/\mathfrak{p}_i}\left(H^0\left(\mathbb{Q}, \mathcal{A} [\mathfrak{p}_i]\right)^{\lor} \right)=\left(\gamma-1 \right)
\end{equation*}
by Proposition \ref{H0}. Thus by combining Theorem \ref{H0}-(3), the proof of Corollary \ref{sp} is done in the same way as Corollary \ref{IMCs}.
\end{proof}
\begin{rem}\label{remeis}
For the cyclotomic deformation of ordinary Galois representations attached to a cusp form $f$, the phenomenon of trivial zeros would occur only when the weight is equal to two and $a\left(p, f \right)=1$. In this section, we know that the phenomenon of trivial zeros may also occur for a certain ordinary (but not modular) representation $\left.\mathbb{T}\right/\mathfrak{p}_i\mathbb{T}$ and the term $\mathrm{char}_{\left.\mathcal{R}\right/\mathfrak{p}_i}\left(H^0\left(\mathbb{Q}, \mathcal{A} [\mathfrak{p}_i]\right)^{\lor} \right)$ is a modification related to the trivial zero. 

\end{rem}

Yau Mathematical Science Center, Tsinghua University, Beijing, China

E-mail address: yandong890911@mail.tsinghua.edu.cn


\begin{thebibliography}{9}




\bibitem{Belllec}
J. Bella\"iche,
\textit{Ribet's lemma, generalizations, and pseudocharacters.} available at \url{http://people.brandeis.edu/~jbellaic/RibetHawaii3.pdf}.











\bibitem{BC05}
J. Bella\"iche, G. Chenevier,
\textit{Lissit\'e de la courbe de Hecke de GL2 aux points Eisenstein critiques.}
J. Inst. Math. Jussieu 5 (2006), no. 2, 333-349.

\bibitem{BC14}
J. Bella\"iche, G. Chenevier,
\textit{Sous-groupes de $GL_2$ et arbres.}
J. Algebra 410 (2014), 501-525.




\bibitem{BG06}
J. Bella\"iche, P. Graftieaux, 
\textit{Repr\'esentations sur un anneau de valuation discr$\grave{e}$te complet.}
Math. Ann. 334 (2006), no. 3, 465-488.

\bibitem{BP19}
J. Bella\"iche, R. Pollack, 
\textit{Congruence with Eisenstein series and $mu$-invariants.}
Compos. Math. 155, No. 5, 863-901 (2019).




\bibitem{FW77}
B. Ferrero, L. Washington, 
\textit{The Iwasawa invariant $\mu_p$ vanishes for abelian number fields.}
Ann. of Math. (2) 109 (1979), no. 2, 377-395.







\bibitem{FO12}
O. Fouquet, T. Ochiai,
\textit{Control theorems for Selmer groups of nearly ordinary deformations.}
J. Reine Angew. Math. 666 (2012), 163-187.





\bibitem{RG89}
R. Greenberg,
\textit{Iwasawa theory for $p$-adic representations.}
Algebraic number theory, 97-137, Adv. Stud. Pure Math., 17, Academic Press, Boston, MA, 1989.



\bibitem{RG}
R. Greenberg,
\textit{Iwasawa theory and $p$-adic deformations of motives.}
Motives (Seattle, WA, 1991), 193-223, Proc. Sympos. Pure Math., 55, Part 2, Amer. Math. Soc., Providence, RI, 1994.

\bibitem{RG06}
R. Greenberg,
\textit{On the structure of certain Galois cohomology groups.}
Doc. Math. 2006, Extra Vol., 335-391.

\bibitem{RG10}
R. Greenberg,
\textit{Surjectivity of the global-to-local map defining a Selmer group.}
Kyoto J. Math. 50 (2010), 853-888.

\bibitem{RG16}
R. Greenberg,
\textit{On the structure of Selmer groups.}
in Elliptic Curves, Modular Forms and Iwasawa Theory, Springer Proc. Math. Stat. 188, Springer, Cham, 2016, 225-252.







\bibitem{H1}
H. Hida
\textit{Iwasawa modules attached to congruences of cusp forms.}
Ann. Sci. \'Ecole Norm. Sup. (4) 19 (1986), no. 2, 231-273.



\bibitem{H2}
H. Hida,
\textit{Galois representations into $\mathrm{GL}_2(\mathbb{Z}_p[[X]])$ attached to ordinary cusp forms.}
Invent. Math. 85 (1986), no. 3, 545-613. 

\bibitem{H3}
H. Hida,
\textit{Elementary theory of $L$-functions and Eisenstein series.}
London Mathematical Society Student Texts, 26. Cambridge University Press, Cambridge, 1993.





\bibitem{Ki94}
K. Kitagawa, 
\textit{On standard $p$-adic $L$-functions of families of elliptic cusp forms, in $p$adic monodromy and the Birch and Swinnerton-Dyer conjecture.}
Contemporary Mathematics, vol. 165 (American Mathematical Society, Providence, RI, 1994), 81-110.


\bibitem{M11}
B. Mazur,
\textit{How can we construct abelian Galois extensions of basic number fields ?}
Bull. Amer. Math. Soc. (N.S.) 48 (2011), no. 2, 155-209.

\bibitem{MW84}
B. Mazur, A. Wiles,
\textit{Class fields of abelian extensions of $\mathbb{Q}$.}
Invent. Math. 76 (1984), no. 2, 179-330.

\bibitem{Ochiai01}
T. Ochiai,
\textit{Control theorem for Greenberg's Selmer groups for Galois deformations.}
J. Number Theory 88 (2001), 59-85.



\bibitem{Ochiai05}
T. Ochiai,
\textit{Euler system for Galois deformations.}
Ann. Inst. Fourier (Grenoble) 55 (2005), no. 1, 113-146.




\bibitem{Ochiai06}
T. Ochiai,
\textit{On the two-variable Iwasawa main conjecture.}
Compos. Math. 142 (2006), no. 5, 1157-1200.



\bibitem{Ochiai08}
T. Ochiai,
\textit{The algebraic $p$-adic $L$-function and isogeny between families of Galois representations.}
J. Pure Appl. Algebra 212 (2008), no. 6, 1381-1393.

\bibitem{Ochiai10}
T. Ochiai,
\textit{Iwasawa theory for nearly ordinary Hida deformations of Hilbert modular forms.}
Algebraic number theory and related topics 2008, 301-319, RIMS K\^{o}ky\^{u}roku Bessatsu, B19, Res. Inst. Math. Sci. (RIMS), Kyoto, 2010.

\bibitem{MO05}
M. Ohta,
\textit{Companion forms and the structure of $p$-adic Hecke algebras.}
J. Reine Angew. Math. 585 (2005), 141-172.




\bibitem{PR}
B. Perrin-Riou,
\textit{Variation de la fonction $L$ $p$-adique par isog\'enie.}
Algebraic number theory, 347-358, Adv. Stud. Pure Math., 17, Academic Press, Boston, MA, 1989.

\bibitem{Ri76}
K. Ribet,
\textit{A modular construction of unramified $p$-extensions of $\mathbb{Q}({\mu_p})$.}
Invent. Math. 34 (1976), no. 3, 151-162.

\bibitem{PS}
P. Schneider,
\textit{The $\mu$-invariant of isogenies.}
J. Indian Math. Soc. (N.S.) 52 (1987), 159-170 (1988).

\bibitem{SeTe}
J. P. Serre, 
\textit{Trees.}
Translated from the French original by John Stillwell. Corrected 2nd printing of the 1980 English translation. Springer Monographs in Mathematics. Springer-Verlag, Berlin, 2003.

\bibitem{T}
J. Tilouine,
\textit{Un sous-groupe $p$-divisible de la jacobienne de $X_1(Np^r)$ comme module sur l'alg$\grave{e}$bre de Hecke.}
Bull. Soc. Math. France 115 (1987), no. 3, 329-360.





\bibitem{Wi88}
A. Wiles, 
\textit{On ordinary $\lambda$-adic representations associated to modular forms.}
Invent. Math. 94 (1988), no. 3, 529-573.

\bibitem{Wi90}
A. Wiles, 
\textit{The Iwasawa conjecture for totally real fields.}
Ann. of Math. (2) 131 (1990), no. 3, 493-540. 


\bibitem{Y}
D. Yan,
\textit{Stable lattices in modular Galois representations and Hida deformation.}
J. Number theory 197 (2019), 62-88.

















\end{thebibliography}
\end{document}